\documentclass[10pt,a4]{amsart}
\usepackage{amsfonts,amsmath,amstext,amsbsy,amsthm,amscd,amsmath,amssymb,dsfont}
\usepackage[utf8]{inputenc}
\usepackage{hyperref}
\usepackage[T1]{fontenc}

\usepackage{tikz}
\usetikzlibrary{decorations.markings,backgrounds}

\newcommand{\norme}[1]{\left\Vert #1\right\Vert}
\newtheorem{Lemme}{Lemma}[section]
\newtheorem{Prop}{Proposition}[section]  

\def\var{\varepsilon}

\newtheorem{Rmq}{Remark}[section]
\newtheorem{Thm}{Theorem}[section]

\theoremstyle{remark}

\usepackage{indentfirst}
\newcommand{\be}{\begin{equation}}
\newcommand{\ee}{\end{equation}}
\newcommand{\ba}{\begin{array}}
\newcommand{\ea}{\end{array}}
\newcommand{\bea}{\begin{eqnarray}}
\newcommand{\eea}{\end{eqnarray}}
\newcommand{\bee}{\begin{eqnarray*}}
\newcommand{\eee}{\end{eqnarray*}}

\renewcommand{\div}{\mbox{\rm div}\;\!}

\newcommand{\R} {\mathbb{R}}    
    
\newcommand{\Z} {\mathbb{Z}} 
\newcommand{\cA} {\mathcal{A}}

\newcommand{\cI} {\mathcal{I}}

\newcommand{\cR} {\mathcal{R}}

\def \with {\quad\!\hbox{with}\!\quad}
\def \andf {\quad\!\hbox{and}\!\quad}
\def\Id{\hbox{Id}}

\def\dZ_1{\delta\!Z_1}

\def\d{\partial}

\def\wt{\widetilde}

\def\ddj{\dot\Delta_j}

\usepackage{color}

\usepackage[foot]{amsaddr}

\title[Navier-Stokes-Cattaneo-Christov system]{The Cattaneo-Christov approximation of Fourier heat-conductive compressible fluids}

\date{}

\author[T. Crin-Barat]{Timothée Crin-Barat$^*$}
\author[S. Kawashima]{Shuichi Kawashima}
\author[J. Xu]{Jiang Xu}

\keywords{Cattaneo's law, Critical regularity, Relaxation limit, Heat-conductive flows, Navier-Stokes Fourier system, Hyperbolic approximation\\ *Corresponding author: timotheecrinbarat@gmail.com}

\begin{document}
\maketitle
\begin{abstract}
We investigate the Navier-Stokes-Cattaneo-Christov (NSC) system in $\mathbb{R}^d$ ($d\geq3$), a model of heat-conductive compressible flows serving as a finite speed of propagation approximation of the Navier-Stokes-Fourier (NSF) system. 
Due to the presence of Oldroyd's upper-convected derivatives, the system (NSC) exhibits a \textit{lack of hyperbolicity} which makes it challenging to establish its well-posedness, especially in multi-dimensional contexts. In this paper, within a critical regularity functional framework, we prove the global-in-time well-posedness of (NSC) for initial data that are small perturbations of constant equilibria, uniformly with respect to the approximation parameter $\varepsilon>0$. Then, building upon this result, we obtain the sharp large-time asymptotic behaviour of (NSC) and, for all time $t>0$, we derive quantitative error estimates between the solutions of (NSC) and (NSF). To the best of our knowledge, our work provides the first strong convergence result for this relaxation procedure in the three-dimensional setting and for ill-prepared data.

The (NSC) system is partially dissipative and incorporates both partial diffusion and partial damping mechanisms. To address these aspects and ensure the large-time stability of the solutions, we construct localized-in-frequency perturbed energy functionals based on the hypocoercivity theory.
More precisely, our analysis relies on partitioning the frequency space into \textit{three} distinct regimes: low, medium and high frequencies. Within each frequency regime, we introduce effective unknowns and Lyapunov functionals,
revealing the spectrally expected dissipative structures.

\end{abstract}







\section{Introduction}

\subsection{Presentation of the systems}
In the Eulerian description, a general compressible fluid evolving in $\R^d$ $(d\geq3)$ is characterized
at every  material point~$x\in\R^d$ and time  $t\in\R$
by its {\it  density} $\rho=\rho(t,x)\in\R_+,$  {\it velocity field} $u= u(t,x)\in\R^d$ and its  
{\it internal energy} $e=e(t,x)\in\R_+$. Those physical quantities are governed by:
\begin{itemize}
\item The mass conservation:
$$
\d_t\rho+\div(\rho u)=0,
$$
 \item The momentum conservation:
$$
\d_t(\rho u)+\div(\rho u \otimes u)+\nabla P= \div\tau,
$$
 \item The energy conservation:
$$\d_t\left(\rho \left(\dfrac{|u|^2}{2}+e\right)\right)+\div\left(\rho u \left(\dfrac{|u|^2}{2}+e\right)+u P\right)+\div q=\div(\tau\cdot u),$$
\end{itemize}
where  $\tau$ is the viscous stress tensor, $P$ the pressure and $q$ the heat flux.
In the regime of Newtonian fluids, $\tau$ is given by
$$
\tau\triangleq 2\mu D(u)+\lambda\div u\,\Id,
$$
where $\lambda$ and $\mu$ are the shear and bulk viscosities satisfying $\mu>0, \nu:=\lambda+2\mu>0$ and $D( u)\triangleq\frac12(\nabla u+{}^T\!\nabla u)$ is the deformation tensor. We consider a pressure function of the form
\begin{align}\label{pressurelaw}P=P(\rho,T)=T\pi(\rho),
\end{align}
where $\pi$ is a smooth function such that $\pi'(\rho)>0$ for $\rho>0$, and we assume that the fluid obeys Joule's law: the internal energy $e$ is a function of the temperature $T\in \R_+$ only and, for a positive constant $C_v$, we have 
$e=C_vT$. Hence, using Gibbs relations for the internal energy and the Helmholtz free energy, we obtain the following temperature equation 
\begin{align}
\label{eq:temp}
\rho C_v(\d_tT+u\cdot\nabla T)+P\div u+\div q=\div(\tau \cdot u).
\end{align}
 Assuming that the heat flux $q$ follows the Fourier law
\begin{equation} \label{FourierLaw}
    q=-\kappa\nabla T,
\end{equation}
where $\kappa>0$ is the heat conductivity coefficient, we obtain the Navier-Stokes-Fourier equations modelling viscous heat-conductive compressible flows:
\begin{equation}\label{FullNSC}
\left\{
\begin{array}
[c]{l}%
\d_t\rho+\div(\rho u)=0,\\
\d_t(\rho u)+\div(\rho u \otimes u)+\nabla P= \div\tau,\\\rho C_v(\d_tT+u\cdot\nabla T)+P\div u-\kappa\Delta T=\div(\tau \cdot u).
\end{array}
\right.
\end{equation} 

The Fourier law \eqref{FourierLaw} has been widely and successfully used to approximate the phenomenon
of heat propagation in continuous media.
However, its relevance comes into question in various applications where alternative approaches are more appropriate for accurate heat conduction modelling. A notable limitation appears when employing the Fourier law to close the system \eqref{FullNSC}, introducing an intrinsic hypothesis regarding heat transfer —the instantaneous response of the heat flux to a temperature gradient. In other words, this assumption, while \textit{mathematically convenient}, leads to an unrealistic prediction: the infinite speed of heat propagation, commonly referred to as the \textit{paradox of heat conduction}.
In particular, the inadequacy of the Fourier law becomes apparent at the nanoscale or in scenarios with short timescales, as detailed in \cite{zhang}, \cite{LTT} (highlighting high-energy laser technology), and \cite{VGV} (addressing nano-fluid heat transport). In such contexts, the response of the heat flux to the temperature gradient in the material is no longer small enough to be neglected and deemed instantaneous. 

To address this limitation, hyperbolic heat conduction models have been introduced, proposing different constitutive equations for the heat flux. One of the best-known is the Maxwell-Cattaneo \cite{Cattaneo} heat transfer law:
\begin{align}
    \label{MaxwellCattaneo}
 \varepsilon^2\d_t q +q=-\kappa \nabla T, 
\end{align}
where the relaxation parameter $\varepsilon>0$\footnote{Its value can be experimentally determined for different materials \cite{Christov2009,ChristovJordan}.} represents the time lag required to establish steady heat conduction in a volume element once a temperature gradient has been imposed across it. The law \eqref{MaxwellCattaneo} corrects the paradox of heat conduction as it ensures a finite speed of propagation of the thermal signal. Indeed, inserting \eqref{MaxwellCattaneo} in the equation of the temperature \eqref{eq:temp} gives, for $u=0$, $\rho\equiv1$ and $C_v= 1$,
\begin{align}
    \label{heattransport}
    \varepsilon^2 \d_{tt}T+\d_tT-\kappa \Delta T=0,
\end{align}
which is a damped wave equation ensuring the propagation of damped thermal waves with the finite speed $\sqrt{\kappa}/\varepsilon$. Such reformulation is often referred to as \textit{second sound} in the context of thermoelasticity \cite{LandauHelium}. 

Although
the Maxwell-Cattaneo law  \eqref{MaxwellCattaneo} preserves the causality principle for heat propagation in steady continuous media, it is incompatible with the Galilean invariance of frame-indifference when the medium is in motion. To address that, Christov and Jordan \cite{ChristovJordan} replaced the partial time derivative with the standard material derivative, leading to the following formulation
\begin{align}
    \label{MCCJ}
\varepsilon^2(\d_tq+u\cdot \nabla q) + q = -\kappa\nabla T.
\end{align}
The constitutive law \eqref{MCCJ} is Galilean invariant and resolves the moving frame paradox as the wave speeds of thermal disturbance are now $c_{1,2}=u\pm \sqrt{\kappa}/\varepsilon$, which is coherent inside a body moving with velocity $u$. However, this law remains imperfect as it does not lead to a single equation for the temperature field.

To remedy that, Christov \cite{Christov2009} proposed an alternative formulation using the Lie-Oldroyd upper-convected time derivative. The following equation for the heat flux leads to a \textit{truly} frame-indifferent formulation\footnote{Observe that the equations \eqref{MCCJ} and \eqref{MCC} are the same in the one-dimensional setting.}
\begin{align}
    \label{MCC}    \varepsilon^2\left(\d_t\mathfrak{q}^\varepsilon+u^\varepsilon\cdot \nabla \mathfrak{q}^\varepsilon-\mathfrak{q}^\varepsilon\cdot\nabla u^\varepsilon+\mathfrak{q}^\varepsilon \div u^\varepsilon\right) + \mathfrak{q}^\varepsilon = -\kappa\nabla T^\varepsilon.
\end{align}
Coupling the fundamental conservation laws of fluid mechanics with \eqref{MCC} leads to the Navier-Stokes-Cattaneo-Christov equations:
\begin{equation}\label{Cattaneo-NSC} 
\left\{
\begin{array}
[c]{l}%
\d_t\rho^\varepsilon+\div(\rho^\varepsilon u^\varepsilon)=0,\\
\d_t(\rho^\varepsilon u^\varepsilon)+\div(\rho^\varepsilon u^\varepsilon \otimes u^\varepsilon)+\nabla P^\varepsilon= \div\tau^\varepsilon,\\\rho^\varepsilon C_v(\d_tT^\varepsilon+u^\varepsilon\cdot\nabla T^\varepsilon)+P^\varepsilon\div u^\varepsilon+\div \mathfrak{q}^\varepsilon=\div(\tau^\varepsilon \cdot u^\varepsilon),
\\
\varepsilon^2(\d_t\mathfrak{q}^\varepsilon+u^\varepsilon\cdot\nabla \mathfrak{q}^\varepsilon-\mathfrak{q}^\varepsilon\cdot\nabla u^\varepsilon+\mathfrak{q}^\varepsilon\div u^\varepsilon)+\mathfrak{q}^\varepsilon=-\kappa\nabla T^\varepsilon.\\
\end{array}
\right.
\end{equation} 
In summary, the physical significance of system \eqref{Cattaneo-NSC} can be distilled into three aspects:
\begin{itemize}
    \item It corrects the unrealistic feature of heat propagation presented in the Navier-Stokes-Fourier system \eqref{FullNSC}: the thermal signal in \eqref{Cattaneo-NSC} exhibits a finite speed of propagation;
      \item It satisfies the Galilean invariance principle of frame indifference;
      \item Formally, in the relaxation limit $\varepsilon\rightarrow0$, system \eqref{Cattaneo-NSC} can be understood as a partially\footnote{It is a \textit{partially} hyperbolic approximation of system \eqref{FullNSC} as there are still parabolic effects in the equation of the velocity in the approximating system \eqref{Cattaneo-NSC}. Fully hyperbolic approximations of the Navier-Stokes-Fourier system are discussed in Section \ref{sec:ext}.} hyperbolic approximation of \eqref{FullNSC}.
\end{itemize}
For comprehensive reviews on the topic of hyperbolic approximations, interested readers can refer to \cite{Chandrasekharaiah98,JosephPreziosi89,PeshkovRomenski,DhaouadiGavrilyuk} and references therein.
\subsection{Overview of our findings.}
Our paper aims to improve the understanding of the Cattaneo-Christov approximation in addressing the \textit{paradox of heat conduction}. 

First, we establish the uniform-in-$\varepsilon$ existence of unique global-in-time small strong solutions to the Navier-Stokes-Cattaneo-Christov system \eqref{Cattaneo-NSC} in dimensions $d\geq3$\footnote{The dimension restriction is explained in Remark \ref{RqThm}.}, cf. Theorem \ref{thm:exist}. The solutions we construct are small perturbations around the constant equilibria
\begin{align}\label{eq:eq}
(\bar{\rho},\bar{u},\bar{T},\bar{q})=(\bar{\rho},0,\bar{T},0)
\end{align}
where $\bar{\rho},\bar{T}>0$. The choice $\bar{q}=0$ comes from the fact that the equilibrium considered has to satisfy the system \eqref{Cattaneo-NSC} to have a suitable linearized structure (see \cite{Yong,BZ}) and, in the context of fluid mechanics, the assumption $\bar{u}=0$ is natural due to the Galilean transformation, as explained in \cite[p. 6]{Arnold} or \cite{DanchinFullNSC}.

Then, building upon our global existence result, we recover the sharp time-decay rates of the solutions in Theorem \ref{Thm:decay}, and we establish global-in-time quantitative error estimates between the solutions of \eqref{Cattaneo-NSC} and \eqref{FullNSC} as $\varepsilon\to 0$, leading to our relaxation result: Theorem \ref{thm:relax}.



To the best of our knowledge, our work is the first to show that the formal link between \eqref{Cattaneo-NSC} and \eqref{FullNSC} is rigorously valid in a multi-dimensional setting. Additionally, we justify that the convergence between both systems holds in a strong sense and without assuming that the initial data satisfy the limit system constraint, i.e. the Fourier law: $\mathfrak{q}_0^\varepsilon=-\kappa\nabla T_0^\varepsilon$. In this sense, our result holds in an ill-prepared context as we handle the initial-time boundary layer.
Furthermore, as discussed in Section \ref{sec:ext}, our approach is robust enough to be applied in justifying the hyperbolization of other parabolic models.

Overall, our research enhances the understanding of the Cattaneo-Christov law by combining hypocoercive techniques with frequency domain decompositions.


\subsection{Existing literature on Cattaneo-type relaxation model}

In contrast to the attention given to \eqref{FullNSC}, as depicted in the next subsection, there is relatively less research dedicated to the Navier-Stokes-Cattaneo-Christov system \eqref{Cattaneo-NSC}. Angeles, Malaga and Plaza \cite{AngelesMalagaPlaza2020} showed that system \eqref{Cattaneo-NSC} is dissipative in the one-dimensional setting, in the sense of the Shizuta-Kawashima (SK) condition \cite{SK}, and established time-decay estimates for a linearization of the system around constant equilibrium states. Then, some references were interested in fully hyperbolic approximations of the Navier-Stokes Fourier system in 1D. In addition to the equation for the heat flux, one also incorporates an evolution equation for the viscous tensor, thereby transforming the velocity equation into a hyperbolic form; see \eqref{MaxwellLaw} in the Appendix. For such models, in \cite{HuRacke2020}, Hu and Racke established the global existence of smooth solutions for small data in $H^2(\R)$ and the local-in-time relaxation limit towards \eqref{FullNSC} as $\varepsilon\to0$ in $H^5(\R)$. Peng and Zhao \cite{PengZhao2022} established the global uniform existence of smooth solutions and justified the global weak 
 convergence in $H^2(\R)$. Conversely, in \cite{HuRacke2023}, Hu and Racke proved a blow-up result for the fully hyperbolic approximation of the Navier-Stokes-Fourier system when the initial data are large. This observation is consistent with the inherent contrast between hyperbolic and parabolic PDEs.
 
  Recently, Angeles  \cite{Angeles22,FA2023} showed that the coupling between the compressible Euler-Cattaneo-Christov system is unfit to model the propagation of thermal and acoustic waves
in several space dimensions. More precisely, it signifies that the inviscid version of \eqref{Cattaneo-NSC} cannot be written in a conservative form, leading to a lack of hyperbolicity. The well-posedness theory remains unresolved in that case. In the viscous case, Angeles \cite{FA2023-2} developed a metric fixed point theorem for Fibonacci contractions and proved the local-in-time existence and uniqueness of solutions to \eqref{Cattaneo-NSC} in any dimensions. In the present paper, we build upon this result to perform our global-in-time analysis.

We also mention the work of Dhaouadi and Gavrilyuk \cite{DhaouadiGavrilyuk} where, using Hamilton's principle and an augmented Lagrangian procedure, the authors derive a purely hyperbolic approximation of the Euler-Fourier system in every dimension. In Section \ref{sec:ext}, we discuss the extension of our methodology to study models arising from their approach.

\subsection{Some literature on the Navier-Stokes-Fourier system}
So far there is a huge literature on the existence, blow-up and large-time behaviour of solutions to the Navier-Stokes-Fourier system \eqref{FullNSC}. The local existence and uniqueness of smooth solutions away from vacuum were proved by Serrin \cite{S} and Nash \cite{N}. The local existence of strong solutions in Sobolev spaces was constructed by Solonnikov \cite{So}, Valli \cite{V} and Fiszdon and Zajaczkowski \cite{FZ}. Matsumura and Nishida \cite{MN1,MatsumuraNishida} established the global-in-time existence of strong solutions being small perturbations of a linearly stable constant state $\left(\varrho_{\infty},0, \theta_{\infty}\right)$ (with $\varrho_{\infty}>0$) in three dimensions. Moreover, with an additional $L^1$ \textit{smallness assumption} on the initial data, the optimal decay rate coinciding with that of the heat kernel was obtained. Later, those results were generalized to other regions: for example, exterior domains were investigated by Kobayashi \cite{K} and Kobayashi and Shibata \cite{KS}, and the half-space by Kagei and Kobayashi \cite{KK1,KK2}. Results related to wave propagation are also available: Zeng \cite{Z} investigated the $L^1$ convergence to the nonlinear Burgers' diffusion wave. Hoff and Zumbrun \cite{HZ} performed a detailed analysis of the Green function in the multi-dimensional case and obtained the $L^\infty$ decay rates of diffusion waves. In \cite{LW}, Liu and Wang exhibited the pointwise convergence of solutions to diffusion waves with the optimal time-decay rate in odd dimensions, where a weaker Huygens' principle was also shown. For the existence of solutions with arbitrary data, Xin \cite{X} found that any smooth solution to the Cauchy problem of compressible Navier-Stokes system without heat conduction (including the barotropic case) would blow up in finite time if the initial density contains vacuum. Huang, Li and Xin \cite{HLX} established the global existence of classical solutions with small energy that may have large oscillations and contain vacuum states. 
A breakthrough is due to Lions \cite{L}, who obtained the global existence of weak solutions with finite energy when the adiabatic exponent is suitably large. Subsequently, some improvements were achieved by Feireisl, Novotny and Petzeltov\'{a} \cite{FNP} and Jiang and Zhang \cite{JZ}. However, the uniqueness of weak solutions remains an open question.
\smallbreak
Regarding frameworks similar to the one we employ in this paper—specifically, strong solutions being small perturbations of constant equilibria—we refer to the following references. Danchin \cite{NSCL2, DanchinFullNSC} established the global existence of unique strong solutions to \eqref{FullNSC} in critical homogeneous Besov spaces of $L^2$-type. That result was extended to Besov spaces of $L^p$-$L^2$-type by Charve and Danchin \cite{CD} and Chen, Miao and  Zhang \cite{CMZ}. Haspot \cite{Haspot} achieved a similar result by employing a more elementary energy approach based on the viscous effective flux introduced by Hoff in \cite{HD}. In a $L^p$-type critical regularity framework, Danchin and He \cite{HeDanchin} justified the low Mach number convergence to the incompressible Navier–Stokes equations for viscous compressible flows in the ill-prepared data case. Then, Danchin and Xu \cite{DX,XJ} developed a time-weighted energy approach in the Fourier semigroup framework and derived optimal decay rates in $L^p$-type critical spaces. Following this, Xin and Xu \cite{XX} introduced a Lyapunov-type energy method for deriving time-decay rates. In this approach, there is still a requirement for a $\dot{B}^{\sigma}_{2,\infty}$-condition on the low-frequency part of the initial data, though it does not necessarily need to be small. More recently, Danchin and Tolksdorf \cite{DT} investigated the scenario in which the equations are posed on bounded domains of $\mathbb{R}^d$.

\subsection{Outline of the paper}
Our paper is structured as follows. In Sections \ref{sec:reform}-\ref{sec:funcframe}, we give the linearization of system \eqref{Cattaneo-NSC} and introduce the functional framework that we use in our analysis. Sections \ref{sec:mainr}-\ref{sec:strategyofproof} are dedicated to presenting our main results and outlining the methodology employed. In Section \ref{sec:linear}, we prove uniform-in-$\varepsilon$ a priori estimates for the linearization of system \eqref{Cattaneo-NSC}. Section \ref{sec:nonlinear} is devoted to establishing our global well-posedness result, while Section \ref{sec:decay} delves into the study of the large-time behaviour of the solution. In Section \ref{sec:relaxation}, we prove our strong relaxation result. Section \ref{sec:ext} presents a discussion regarding possible extensions of our findings. Some technical lemmas are provided in the Appendix.
    
\section{Reformulation of \eqref{Cattaneo-NSC} and main results}
\subsection{Reformulation of the system}\label{sec:reform}

Let $\bar{\rho}>0$ and $\bar{T}>0$. Linearizing the system \eqref{Cattaneo-NSC} around $(\bar{\rho},0,\bar{T},0)$ and, as in \cite{DanchinFullNSC,DanchinXuFullNSC}, nondimensionalizing it, we obtain
\begin{equation}
\left\{
\begin{array}
[c]{l}%
\d_ta^\varepsilon+\div v^\varepsilon=F^\varepsilon,\\
\d_t v^\varepsilon-\mathcal{A}v^\varepsilon+\nabla a^\varepsilon +\gamma\nabla \theta^\varepsilon=G^\varepsilon,\\
\d_t \theta^\varepsilon+\beta\div q^\varepsilon+\gamma\div v^\varepsilon=H^\varepsilon,\\
\varepsilon^2\d_tq^\varepsilon+\alpha q^\varepsilon+\kappa\nabla\theta^\varepsilon=\var^2I^\varepsilon,
\end{array}
\right.
\label{LinearHypNSCzero}
\end{equation}
where the new unknowns $a^\varepsilon$, $v^\varepsilon$, $\theta^\varepsilon$ and $q^\varepsilon$ are time and space rescaling of\footnote{The particular way we write the perturbation of $\rho^\varepsilon$ is tailored to handle the factor $\rho^\varepsilon$ appearing in front of the time-derivatives in the equations of the velocity and temperature.} $(\rho^\varepsilon-\bar{\rho})/\bar{\rho}$, $u^\varepsilon$, $T^\varepsilon-\bar{T}$ and $\mathfrak{q}^\varepsilon$, respectively, and the Lamé operator is defined as $\mathcal{A}:=(\mu \Delta+(\lambda+\mu)\nabla \div)/\nu$.  The exact rescaling, the positive constant coefficients $\alpha,\beta,\gamma$ and the source terms appearing in \eqref{LinearHypNSCzero} are defined in Appendix \ref{sec:appreform}. Applying a similar procedure to the Navier-Stokes-Fourier system \eqref{FullNSC} leads to
\begin{equation}
\left\{
\begin{array}
[c]{l}%
\d_ta+\div v=F,\\
\d_t v-\mathcal{A}v+\nabla a +\gamma\nabla \theta= G,\\
\d_t \theta-\dfrac{\beta\kappa }{\alpha} \Delta \theta+\gamma\div v=H,\\
\end{array}
\right.
\label{LinearHypNSFzero}
\end{equation}
where $a,v,\theta$$,F$, $G$, and $H$ are defined similarly to $a^\varepsilon,v^\varepsilon,\theta^\varepsilon
,F^\varepsilon$, $G^\varepsilon$ and $H^\varepsilon$. 

\subsection{Functional framework}\label{sec:funcframe}

Due to the dual partially dissipative nature of the Navier-Stokes-Cattaneo-Christov system \eqref{LinearHypNSCzero} (elucidated in Section \ref{sec:strategyofproof}) our strategy relies on the analysis of the solutions in three distinct frequency regimes. Within each of these regimes, the solutions exhibit very different behaviours, necessitating the use of hypocoercive methodologies adapted to each regime. 

A spectral analysis of the model suggests to consider the following thresholds to separate the frequency domain
\begin{align}\label{Thresholds0}
j_0:=K \andf j_{\var}:=\dfrac{k}{\varepsilon},
\end{align}
where $K$ is a suitably large integer and $k$ is a suitably small integer to be determined later\footnote{The values of $K$ and $k$ can be explicitly determined in our computations. It is essential to choose $K$ and $k$ suitably large and small, respectively, to ensure that some linear source terms can be absorbed.}. Given that our analysis will rely on the dyadic Littlewood-Paley decomposition to split the frequencies, we also introduce the thresholds
\begin{align}\label{Thresholds}
J_0:=\log_2K \andf J_{\var}:=-[\log_{2}{\var}]+\log_2k,
\end{align}
 Remark that the thresholds $J_0$ and $J_\varepsilon$ correspond to those employed in the study of partially diffusive systems in \cite{HeDanchin} and partially damped systems in \cite{CBD1,CBD3}, respectively.
Then, to justify our computations, we assume that
\begin{align}\label{ThresholdAssum}
J_0\leq J_\var.
\end{align}
This assumption is crucial in our computations to ensure that \eqref{LinearHypNSCzero} verifies hypocoercive stability properties. In particular, it implies that the well-known (SK) condition\footnote{ We recall that the (SK) condition is a sufficient algebraic criterion ensuring the stability of partially dissipative systems and that it is equivalent to the Kalman rank condition, as pointed out in \cite{BZ}. However, it is not a necessary condition in the multi-dimensional context. Consequently, whether global-in-time solutions exist when \eqref{ThresholdAssum} fails remains an open question.} \cite{SK} is satisfied. It is also in line with \cite[Theorem 2.2]{HuRacke2016} where the authors show that \eqref{LinearHypNSCzero} does not satisfy the (SK) condition when $\varepsilon$ is too large. Since $K$ and $k$ are constants, the assumption \eqref{ThresholdAssum} implies that  $\varepsilon$ must be sufficiently small.

\medbreak
This preliminary analysis suggests working in a functional framework that would facilitate the decomposition of the frequencies. In this regard, the Littlewood-Paley decomposition and homogeneous Besov spaces emerge as natural tools for analyzing our model. We define the following frequency-restricted homogeneous Besov semi-norms corresponding to the decomposition induced by \eqref{Thresholds}:
\begin{equation}
\begin{aligned}
&\|f\|_{\dot{B}^{s}_{p,1}}^{\ell}:=\sum_{j\leq J_{0}}2^{js}\|f_{j}\|_{L^{2}},\quad\|f\|_{\dot{B}^{s}_{p,1}}^{m,\varepsilon}:=\sum_{J_0\leq j\leq J_{\varepsilon}}2^{js}\|f_{j}\|_{L^{2}} \andf \|f\|_{\dot{B}^{s}_{p,1}}^{h,\varepsilon}:=\sum_{j\geq J_{\varepsilon}-1}2^{js}\|f_{j}\|_{L^{2}},
\end{aligned}
\end{equation}
where $f_j:=\ddj f$ and $\ddj$ is the classical Littlewood-Paley frequency-localization operator, see \cite[Chapter 2]{HJR}.
We also introduce the following semi-norms that will be useful in our analysis:
\begin{align}  \|f\|_{\dot{B}^{s}_{p,1}}^{\ell,\varepsilon}:=\sum_{j\leq J_\varepsilon}2^{js}\|f_{j}\|_{L^{p}} \andf \|f\|_{\dot{B}^{s}_{p,1}}^{h}:=\sum_{j\geq J_0}2^{js}\|f_{j}\|_{L^{p}}.
\end{align}
Note that the norms without the $\varepsilon$-dependency correspond to the spaces used to treat the limit system \eqref{FullNSC} in \cite{DanchinFullNSC,HeDanchin,DanchinXuFullNSC}. We emphasize here that our utilization of Besov spaces serves not only to obtain results in a critical regularity setting but is pivotal for deriving our strong relaxation limit result, as it enables us to recover sharp dissipative properties. The key insight is that, without implementing a frequency splitting, only the worst behaviour among the three regimes would prevail, making it impossible to establish uniform bounds to justify the relaxation limit.

\smallbreak
Given the dependence of the norms on $K$, $k$ and $\varepsilon$, it is necessary to monitor these parameters when employing Bernstein-type embeddings within each frequency regime. We recall that, in the context of the Littlewood-Paley theory, Bernstein-type estimates can be applied to each localized-in-frequency component and allow for the control of the behaviour of the function across different frequency scales. For instance, one can show that for a distribution localized in the low-frequency regime $|\xi|\leq1$, the $L^2$ norm of its gradient can be bounded by the $L^2$ norm of the distribution (see the first inequality of \eqref{eq:medBernstein} for $s=1$ and $s'=1$). The following proposition does this seamlessly in our context and directly follows from the standard Bernstein inequalities derived, for instance, in \cite[Chapter 2]{HJR}.
\begin{Prop}[Bernstein-type inequalities] \label{prop:Bernstein}
Let $f$ be a smooth function, $p\in[1,\infty]$, $s\in\R$ and $s'>0$. The following inequalities hold true
\begin{align} \label{eq:medBernstein}
&\|f\|_{\dot{B}^{s}_{p,1}}^{\ell}\lesssim K^{s'} \|f\|_{\dot{B}^{s-s'}_{p,1}}^{\ell}, \quad \:\|f\|_{\dot{B}^{s}_{p,1}}^{h,\var}\lesssim k^{s'}\varepsilon^{s'} \|f\|_{\dot{B}^{s+s'}_{p,1}}^{h,\varepsilon},
\\\|f\|_{\dot{B}^{s}_{p,1}}^{\ell,\varepsilon}\lesssim k^{s'}&\varepsilon^{-s'} \|f\|_{\dot{B}^{s-s'}_{p,1}}^{\ell,\var},\quad  \:
\|f\|_{\dot{B}^{s}_{p,1}}^{m,\varepsilon}\lesssim k^{s'}\varepsilon^{-s'} \|f\|_{\dot{B}^{s-s'}_{p,1}}^{m,\var}\quad \text{and} \quad \:\|f\|_{\dot{B}^{s}_{p,1}}^{m,\var}\lesssim K^{-s'} \|f\|_{\dot{B}^{s+s'}_{p,1}}^{m,\varepsilon}.\label{eq:lowhfBernstein}
\end{align}
\end{Prop}

\subsection{Main results}\label{sec:mainr}
Before stating our results, we define the functional norm $X_0^\varepsilon$ associated with the initial data:
\begin{align}\label{X0}
X_0^\varepsilon:=X_0^\ell+X_0^{m,\varepsilon}+X_0^{h,\varepsilon},
\end{align}
where
\begin{align*}
   &X_0^\ell= \|(a_0^\varepsilon,v_0^\varepsilon,\theta_0^\varepsilon,\varepsilon q_0^\varepsilon)\|_{\dot{B}^{\frac d2-1}_{2,1}}^\ell, \quad X_0^{m,\varepsilon}=\|(\theta_0^\varepsilon,\varepsilon q_0^\varepsilon)\|^{m,\var}_{\dot{B}^{\frac{d}{p}-2}_{p,1}\cap \dot{B}^{\frac{d}{p}-1}_{p,1}}+\|v_0^\varepsilon\|^{m,\var}_{\dot{B}^{\frac{d}{p}-1}_{p,1}}+ \|a_0^\varepsilon\|^{m,\var}_{\dot{B}^{\frac{d}{p}}_{p,1}},
\\&X_0^{h,\varepsilon}=\varepsilon\|w_0^\varepsilon\|^{h,\var}_{\dot{B}^{\frac{d}{2}}_{2,1}}+\varepsilon\|v_0^\varepsilon\|^{h,\var}_{\dot{B}^{\frac{d}{2}+1}_{2,1}}+\varepsilon\|a_0^\varepsilon\|^{h,\var}_{\dot{B}^{\frac{d}{2}+1}_{2,1}}+\varepsilon\|(\theta_0^\varepsilon,\var q_0^\varepsilon)\|^{h,\var}_{\dot{B}^{\frac{d}{2}+1}_{2,1}}.
\end{align*}
We are now ready to state our first result: The well-posedness of system \eqref{LinearHypNSCzero} within a critical regularity framework, provided that the relaxation parameter $\varepsilon$ is small enough so the (SK) stability condition holds.

\begin{Thm}[Uniform-in-$\varepsilon$ global well-posedness for small data]\label{thm:exist} Let $\varepsilon>0$, $d\geq3$, $p\in[2,d)$ and  $p\in[2,\frac{2d}{d-2}]$. There exist constants $K,k\in\Z$ and $\eta_0>0$ such that for all $\varepsilon$ satisfying \eqref{ThresholdAssum} and if
$$ X_0^\varepsilon \leq \eta_0,$$
then system \eqref{LinearHypNSCzero} admits a unique global-in-time solution $(a^\varepsilon,v^\varepsilon,\theta^\varepsilon,q^\varepsilon)$ satisfying, for all $t>0$,
\begin{align}\label{bound:thmexist}
X^\varepsilon(t)\leq C X_0^\varepsilon, 
\end{align}
where $X^\varepsilon(t)$ is defined in \eqref{X} and $C>0$ is a constant independent of $t$, $\varepsilon$ and the initial data.
\end{Thm}
\begin{Rmq} Some remarks are in order.
\begin{itemize}

\item To the best of our knowledge, Theorem \ref{thm:exist} is the first result to show the uniform global well-posedness of the Navier-Stokes-Cattaneo-Christov equations in a multi-dimensional framework.

\item The condition \eqref{ThresholdAssum} allows to employ a hypocoercivity-based methodology. If not satisfied, the frequency regimes are interchanged, rendering our method for deriving uniform a priori estimates inapplicable. Notably, as elucidated in our computations, the condition \eqref{ThresholdAssum} is equivalent to the (SK) condition (\cite{SK}).
\item The bound \eqref{bound:thmexist} provides uniform-in-$\varepsilon$ and $\mathcal{O}(\varepsilon)$ controls of the solution for all times that will be essential to justify the relaxation limit $(\varepsilon\to0)$ in Theorem \ref{thm:relax}. It results from a sharp linear analysis combined with refined and new product laws, in hybrid Besov spaces of $L^2$-$L^p$-$L^2$-type, to bound the nonlinear terms (see Propositions \ref{LPP}-\ref{prop:PL-low}). 
\item  The limitation to dimensions larger than $3$ arises from technical difficulties in establishing product laws within the critical regularity framework, as apparent in \eqref{R-E955c} for instance. 
An alternative approach to deal with the two-dimensional case, based on a Lagrangian reformulation, is discussed in Section \ref{sec:ext}.
 \end{itemize}
\end{Rmq}

 Building upon Theorem \ref{thm:exist}, we analyze the large-time asymptotic behaviour of the solutions of \eqref{LinearHypNSCzero}.
\begin{Thm}[Large-time behaviour]\label{Thm:decay}
Let the assumption of Theorem \ref{thm:exist} be in force and $(a^\var,v^\var,\theta^\var,q^\var)$ be the corresponding global-in-time solution of \eqref{LinearHypNSCzero} associated to the initial data $(a^\var_0,v^\var_0,\theta^\var_0,q^\var_0)$. If $(a^\var_0,v^\var_0,\theta^\var_0,q^\var_0)^\ell\in \dot{B}^{-\sigma_1}_{2,\infty}$ for $1-\frac d2< \sigma_1\leq \sigma_0=\frac{2d}{p}-\frac d2	$, then, for all $t>0$,
\begin{align}\label{est:thmdecay}
&\|\Lambda^{\sigma}(a^\var,v^\var)(t)\|_{L^p}\leq C(1+t)^{-\frac d2(\frac 12-\frac 1p)-\frac{\sigma+\sigma_1}{2}}\quad \mbox{if}\quad -\tilde{\sigma}_1<\sigma \leq \frac dp-1
\end{align}
and 
\begin{align}\label{est:thmdecay2}
&\|\Lambda^{\sigma}(\theta^\var,\varepsilon q^\var)(t)\|_{L^p}\leq C(1+t)^{-\frac d2(\frac 12-\frac 1p)-\frac{\sigma+\sigma_1}{2}}\quad \mbox{if}\quad -\tilde{\sigma}_1<\sigma \leq \frac dp-2,
\end{align}
where $\tilde{\sigma }_1=\sigma_1+d(\frac 12-\frac 1p) $, the operator $\Lambda^{\sigma}$ is defined by $\Lambda^{\sigma}f\triangleq\mathcal{F}^{-1}\left(|\xi|^{\sigma}\mathcal{F}f\right)$ and the constant $C$ is independent of $t$ and $\varepsilon$.
\end{Thm}
\begin{Rmq} Some comments are in order.
\begin{itemize}
    \item In the seminal work \cite{MatsumuraNishida}, the $L^1$-control of the initial data was pinpointed as a sufficient condition to derive the large-time behaviour of solutions. Here, our assumption is less restrictive thanks to the embedding $L^1\hookrightarrow\dot{B}^{-\frac{d}{2}}_{2,\infty}$. 

\item The decay rates in \eqref{est:thmdecay} and \eqref{est:thmdecay2} are sharp in the sense that they are uniform-in-$\varepsilon$ and align with the optimal rates achievable for the limit system \eqref{FullNSC} (derived for instance in \cite{DanchinXuFullNSC}).

    \end{itemize}
\end{Rmq}
Before presenting our relaxation limit result, we recall (in a simplified statement) the global well-posedness result for the limit system \eqref{FullNSC} established by Danchin and He in \cite{HeDanchin}.
\begin{Thm}[\cite{HeDanchin}]\label{Thm:ExistHeDanchin}
 Let $d\geq3$, $p\in[2,d)$ and  $p\in[2,\frac{2d}{d-2}]$. There exist $K\in \mathbb{Z}$ and $\eta>0$ such that if
$$ X_0 \leq \eta_0,$$
then system \eqref{LinearHypNSFzero} admits a unique global solution $( a,u,\theta)$ that, for all $t>0$, satisfies
$$X(t)\leq X_0,$$
where
$$X_0=\| a_0\|^h_{\dot{B}^{\frac{d}{p}}_{p,1}}+\|v_0\|^h_{\dot{B}^{\frac{d}{p}-1}_{p,1}}+\|\theta_0\|^h_{\dot{B}^{\frac{d}{p}-1}_{p,1}}+\|( a_0, v_0, \theta_0)\|^\ell_{\dot{B}^{\frac{d}{2}-1}_{2,1}}$$
and
\begin{align}\label{X:HeDanchin}
 X(t)&=\|( a, v, \theta)\|^\ell_{L^\infty_T(B^{\frac{d}{2}-1}_{2,1})}+\|( a, v, \theta)\|^\ell_{L^1_T(B^{\frac{d}{2}+1}_{2,1})}\\&\nonumber\quad+\|a\|^h_{L^\infty_T\cap L^1_T(B^{\frac{d}{p}}_{p,1})}+\|v\|^h_{L^\infty_T(B^{\frac{d}{p}-1}_{p,1})}+\|v\|^h_{L^1_T(B^{\frac{d}{p}+1}_{p,1})}+\|\theta\|^h_{L^\infty_T(B^{\frac{d}{p}-2}_{p,1})}+\|\theta\|^h_{L^1_T(B^{\frac{d}{p}}_{p,1})}.
  \end{align}

\end{Thm}
 Finally, we present our result on the global-in-time relaxation limit of the Navier-Stokes-Cattaneo-Christov system \eqref{LinearHypNSCzero} to the Navier-Stokes-Fourier system \eqref{LinearHypNSFzero} as $\varepsilon\to0$.
\begin{Thm}[Strong relaxation limit]\label{thm:relax}
Let $\varepsilon>0$ and assume that the hypotheses of Theorem \ref{thm:exist} are fulfilled. Let $(a^\varepsilon,v^\varepsilon,\theta^\varepsilon,q^\varepsilon)$ be the global solution of \eqref{LinearHypNSCzero} given by Theorem \ref{thm:exist}
supplemented with initial data $(a_0^\varepsilon,v_0^\varepsilon,\theta_0^\varepsilon,q_0^\varepsilon)$ and let $( a,v,\theta)$ be the global solution of \eqref{LinearHypNSFzero} given by Theorem \ref{Thm:ExistHeDanchin}
supplemented with initial data $( a_0,v_0,\theta_0)$.
Define the error unknowns $$(\widetilde{a},\widetilde{v},\widetilde{\theta}):=(a^\varepsilon-a,v^\varepsilon-v,\theta^\varepsilon-\theta) \quad \text{and} \quad (\widetilde{a}_0,\widetilde{v}_0,\widetilde{\theta}_0):=(a^\varepsilon_0-a_0,v^\varepsilon_0-v_0,\theta^\varepsilon_0-\theta_0).$$
If we assume that
\begin{align}
\|(\wt a_0,\wt v_0,\wt\theta_0)\|_{\dot{B}^{\frac{d}{2}-2}_{2,1}}^\ell+\|\wt a_0\|^h_{\dot{B}^{\frac{d}{p}-1}_{p,1}}+\|(\wt v_0,\wt\theta_0)\|_{\dot{B}^{\frac{d}{p}-2}_{p,1}}^h\lesssim \varepsilon,
\end{align}
then, for all $T>0$, we have
$$\wt X(T)\lesssim \varepsilon,$$
where
\begin{align}\label{tildeX}
\wt X(T)&=\|(\wt a,\wt v,\wt \theta)\|^\ell_{L^\infty_T(B^{\frac{d}{2}-2}_{2,1})}+\|(\wt a,\wt v,\wt \theta)\|^\ell_{L^1_T(B^{\frac{d}{2}}_{2,1})} + \|\alpha q^\varepsilon+\kappa\nabla \theta^\varepsilon\|_{L^1_T(B^{\frac{d}{p}-1}_{p,1})}
  \\&\nonumber+
\|\wt a\|^h_{L^\infty_T\cap L^1_T(B^{\frac{d}{p}-1}_{p,1})}+\|(\wt v,\wt \theta)\|^h_{L^\infty_T(B^{\frac{d}{p}-2}_{p,1})}+\|(\wt v,\wt \theta)\|^h_{L^1_T(B^{\frac{d}{p}}_{p,1})}.
  \end{align}
Thus, as $\varepsilon\to0$,
$$\alpha q^\varepsilon+\kappa\nabla \theta^\varepsilon\longrightarrow 0 \quad \text{strongly in} \quad L^1(\R^+;\dot B^{\frac{d}{p}-1}_{p,1}), $$
and
$$(a^\varepsilon-  a, v^\varepsilon-v,\theta^\varepsilon-\theta)  \longrightarrow 0 \quad \text{strongly in } E,
$$
where $E$ is the functional space associated to the norm $\widetilde{X}$.

\end{Thm}

\begin{Rmq}\label{RqThm}
    Some remarks are in order.
    \begin{itemize}
        \item Theorem \ref{thm:relax} is, to the best of our knowledge, the first result establishing the global-in-time strong convergence, in dimensions $d\geq3$, of the Navier-Stokes-Cattaneo-Christov systems towards the Navier-Stokes-Fourier system. \item The convergence holds in an ill-prepared scenario in the sense that we do not assume the initial data of \eqref{LinearHypNSCzero} to satisfy the limit system constraint: the Fourier law $\alpha q_0^\varepsilon=-\kappa \nabla \theta^\varepsilon_0$. To deal with the initial-time boundary layer formation, we show that the quantity $Q:= \alpha q_0^\varepsilon+\kappa \nabla \theta^\varepsilon_0$ behaves as $e^{-t/\varepsilon}$ in suitable norms.
        \item As explained in Section \ref{sec:ext}, our methodology is robust enough to analyze fully hyperbolic approximations of the Navier-Stokes-Fourier systems explored in \cite{DhaouadiGavrilyuk,HuRacke2017, Peng21, PengZhao2022}.
    \end{itemize}
\end{Rmq}

\subsection{Strategies of proof}\label{sec:strategyofproof}We present our strategy to analyze the system \eqref{LinearHypNSCzero}.

\subsubsection{Partial dissipation}

In the Navier-Stokes-Cattaneo-Christov system \eqref{LinearHypNSCzero}, dissipative operators are present in only two of the four equations: there is diffusion for the velocity field $u^\varepsilon$ through the stress tensor $\tau^\varepsilon$ and damping for the heat flux $q^\varepsilon$. However, to justify the stability of the system, it is necessary to recover dissipation for all the components. To that matter, it is important to understand \eqref{LinearHypNSCzero} as a combination of
\begin{itemize}
 \item[(i)] A partially damped coupling between $\theta^\varepsilon$ and $q^\varepsilon$, enabling to recover dissipation for $\theta^\varepsilon$.
    \item[(ii)] A partially diffusive coupling between $a^\varepsilon$ and $v^\varepsilon$, allowing to recover dissipation for $a^\varepsilon$.
\end{itemize}
Both of these coupling are related to the study of partially dissipative hyperbolic systems, a topic initially developed by Shizuta and Kawashima in \cite{SK}. There, the authors developed an algebraic condition, the (SK) condition, ensuring the stability of the system when the hyperbolic eigendirections of the system avoid the kernel of the dissipation. More recently, Beauchard and Zuazua \cite{BZ} have framed the partially dissipative coupling (i) into Villani's hypocoercivity theory \cite{Villani} and enhanced its understanding. Employing tools from the control theory, they show that the interactions between the hyperbolic and dissipative parts of the system can propagate the dissipation to directions that are not affected by the damping operator. Then, inspired by this work, Crin-Barat and Danchin \cite{ThesisCB, CBD3} obtained new results for the relaxation associated with partial damping using frequency-localization arguments. It is their approach that we shall employ to deal with the partially damped coupling inside \eqref{LinearHypNSFzero}.

The analysis of the coupling of type (ii) goes back to the theory developed by Danchin \cite{NSCL2} concerning the compressible Navier-Stokes equations. However, as we will see below, this coupling can also be comprehended as an application of the hypocoercivity theory.
\medbreak

\medbreak
In the following two subsections, we revisit the essential aspects of the strategies employed to investigate linear partially diffusive systems (ii) and the relaxation of linear partially damped systems (i). The interested reader may also consult \cite{DanchinSurveyPartDissip} for a comprehensive survey on these two phenomena.

\subsubsection{Partially diffusive setting}
We examine the simplified model
    \begin{equation}\label{TM:partdiff}
\left\{
\begin{aligned}
 & \d_ta+\div u=0,
        \\ &\d_t u+\nabla a-\Delta u=0.
\end{aligned}\right.
\end{equation}
A spectral analysis of the system reveals the following behaviour:
\begin{itemize}
    \item In the low frequencies regime, $|\xi|\leq K$, the solution $(a,u)$ exhibits characteristics akin to solutions of the heat equation.
    \item In the high frequencies regime, $|\xi|\geq K$, $a$ undergoes a damping effect, and $u$ has a parabolic behaviour.
\end{itemize}
Hence, it is appropriate to analyze these two frequency regimes using different techniques.
\medbreak
\textbf{The low-frequency regime $|\xi|\leq K$.} In the reference \cite{NSCL2}, the author formulates a Lyapunov functional that allows to recover dissipation for the component $a$, it reads
\begin{align}\label{strat:lya1}  \mathcal{L}_1(t)=\|(a,u)(t)\|_{L^2}^2+\frac{1}{2}\int_{\R^d}u\cdot \nabla a.
\end{align}
Taking the time derivative of $\mathcal{L}_1$ and applying Young's inequality yields
\begin{align}
  \dfrac{1}{2}  \dfrac{d}{dt}\mathcal{L}_1+\|(\nabla   a,\nabla  u)\|_{L^2}^2\leq \text{r.h.s},
\end{align}
where the linear terms on the right-hand side can be absorbed in the low-frequency regime with Bernstein-type inequalities. Then, using that $\mathcal{L}_1(t)\sim \|(a,u)(t)\|_{L^2}^2$ leads to the spectrally expected stability estimates. Within this regime, computations are restricted to a $L^2$-in-space framework as the coupling between the equation is necessary to justify the stability of the system. It is worth noting that the utilization of the perturbed energy functional \eqref{strat:lya1}, as well as the others below, aligns with the theory of hypocoercivity developed by Villani \cite{Villani}.
\medbreak
\textbf{The high frequencies $|\xi|\geq K$.} In \cite{Haspot}, Haspot introduces the effective velocity unknown  $w=u+(-\Delta)^{-1}\nabla  a$ and rewrites \eqref{TM:partdiff} as
\begin{equation}\label{TM:partdiffHF}
\left\{
\begin{aligned}
 & \d_t a+ a=\div w,
        \\ &\d_t w-\Delta w=w+-(-\Delta)^{-1}\nabla a^\varepsilon.
\end{aligned}\right.
\end{equation}
Estimating each equation of \eqref{TM:partdiffHF} separately, one obtains
\begin{align}
    &\dfrac{1}{p}\dfrac{d}{dt}\|a(t)\|_{L^p}^p+\|a\|_{L^p}^p\leq \text{r.h.s},
\\
    &\dfrac{1}{p}\dfrac{d}{dt}\|w(t)\|_{L^p}^p+\|\nabla w\|_{L^p}^p\leq \text{r.h.s},
\end{align}
where the linear terms on the right-hand sides can be absorbed in the high-frequency regime with Bernstein-type inequalities. Furthermore, within this regime, due to the partial diagonalization of the system with the effective velocity $w$, computations can be conducted in a $L^p$-in-space framework for $p\geq2$.

\subsubsection{Relaxation of partially damped systems}
Now, let's explore the justification of the relaxation limit for a partially dissipative toy model. We focus on the linear heat equation
 \begin{align}
     \d_t \theta-\Delta \theta=0.
 \end{align}
 Its hyperbolic Cattaneo approximation read, for a $\varepsilon>0$,
\begin{equation}\label{TM:partdiss}
\left\{
\begin{aligned}
 & \d_t\theta^\varepsilon+\div q^\varepsilon=0,
        \\ &\varepsilon^2\d_t q^\varepsilon+\nabla\theta^\varepsilon+q^\varepsilon=0.
\end{aligned}\right.
\end{equation}
An analysis of the spectral properties of the matrix associated with the system:
\begin{align}\begin{pmatrix}0 &  i\xi &\\ i\xi  &\dfrac{1}{\varepsilon}\end{pmatrix}
\end{align}reveals that:
\begin{itemize}
   \item In low frequencies, $|\xi|\ll \dfrac{1}{\varepsilon}$,
   there are two real eigenvalues $\dfrac{1}{\varepsilon}$ and $\varepsilon\xi^2$.
\item In high frequencies, $|\xi|\gg \dfrac{1}{\varepsilon}$, two complex conjugate eigenvalues coexist, whose real parts are asymptotically equal to $\dfrac{1}{2\varepsilon}$.
 \item The threshold between low and high frequencies is at $1/\varepsilon$.
\end{itemize}
This analysis reveals that the solution's behaviour is significantly influenced by the relationship between $\xi$ and $\varepsilon$. Notably, there exists a purely damped mode at low frequencies, which contrasts with the purely parabolic behaviour at low frequencies described in \cite{BZ,SK}. Additionally, as $\varepsilon\to0$, the high-frequency regime disappears, leaving the low-frequency behavior to be dominant in the whole frequency space. This observation is consistent as what persists in the limit exhibits the same behaviour as the limit equation: a parabolic behaviour for $\theta$, and the Fourier-type law $q=-\nabla \theta$.
\smallbreak
In summary, the hyperbolic Cattaneo-type approximation introduces a purely damped regime in high frequencies, while preserving the nature of the limit system in low frequencies.
Next, we revisit the analysis of \eqref{TM:partdiss} in both frequency regimes.

\medbreak
\textbf{The low-frequency regime $|\xi|\leq 1/\varepsilon$.}
In the low-frequency regime, we introduce the effective unknown $Q^\varepsilon=q^\varepsilon+\nabla \theta^\varepsilon$ to rewrite \eqref{TM:partdiff} as
\begin{equation}\label{TM:partdissHF}
\left\{
\begin{aligned}
 & \d_t \theta^\varepsilon-\Delta \theta^\varepsilon=\div Q^\varepsilon,
        \\ &\varepsilon\d_t Q^\varepsilon+\dfrac{Q^\varepsilon}{\varepsilon}=-\varepsilon\Delta Q^\varepsilon-\varepsilon\Delta\nabla  \theta^\varepsilon.
\end{aligned}\right.
\end{equation}
Estimating each equation of \eqref{TM:partdiffHF} separately, we obtain
\begin{align}
    \dfrac{1}{2}\dfrac{d}{dt}\|( \theta^\varepsilon,\varepsilon Q^\varepsilon)(t)\|_{L^2}^2+\|\nabla  \theta^\varepsilon\|_{L^2}^2+\dfrac{1}{\varepsilon}\|Q^\varepsilon\|_{L^2}^2\leq \text{r.h.s},
\end{align}
where the right-hand side linear terms are of high order and can be absorbed in the low-frequency regime using Bernstein-type inequalities. As for the high-frequency regime in the partially diffusive case, the partial diagonalization of the system enables to work in a $L^p$-in-space framework with $p\geq2$.
\medbreak
\textbf{The high-frequency regime $|\xi|\geq 1/\varepsilon$.} Drawing inspiration from hypocoercivity-type arguments (refer to \cite{BZ, CBD2, CBD1, NSCL2, Villani}), we define the perturbed energy functional
\begin{align}
    \mathcal{L}_2(t)=\|( \theta^\varepsilon,\varepsilon q^\varepsilon)(t)\|_{L^2}^2+2^{-2j}\int_{\R^d}q^\varepsilon\cdot \nabla  \theta^\varepsilon.
\end{align}
Taking the time derivative of $\mathcal{L}_2$ and applying Young's inequality, we obtain:
\begin{align}
   \dfrac{1}{2} \dfrac{d}{dt}\mathcal{L}_2+\dfrac{1}{\varepsilon}\|(\theta^\varepsilon,  q^\varepsilon)\|_{L^2}^2\leq \text{r.h.s}.
\end{align}
Here, the right-hand side linear terms are of low order and can be absorbed in the high-frequency regime through Bernstein-type inequalities.  Then, using that $\mathcal{L}_2(t)\sim\|( \theta^\varepsilon,\varepsilon q^\varepsilon)(t)\|_{L^2}^2$ leads to the spectrally expected stability estimates. Again, due to the lack of partial diagonalization for the system, the computations are restricted to a $L^2$-in-space framework in this regime.

\subsubsection{Decomposition of the frequency space for \eqref{LinearHypNSCzero}}
Synthesizing the insights gathered in the preceding two subsections, to analyze \eqref{LinearHypNSCzero}, we partition the frequency space into three regimes as follows (refer to Figure \ref{fig:3regimes}):
\begin{itemize}
    \item Low frequencies: $j\leq J_0$.
\item Medium frequencies: $J_0 \leq j \leq J_\varepsilon$.
\item High frequencies: $J_\varepsilon\leq j.$
\end{itemize}
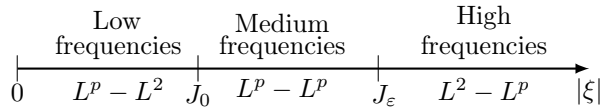
\begin{figure}[!h]
    \centering
	\begin{tikzpicture}[xscale=0.4,yscale=0.4, thick]
	\draw [-latex] (-1,0) -- (18,0) ;
\draw  (18,0)  node [below]  {$|\xi|$} ;
\draw  (11.2,-0.2)  node [below]  {$J_\varepsilon$} ;
\draw  (5,-0.2)  node [below]  {$J_0$};
\draw  (-1,-0.2)  node [below]  {$0$} ;

\draw (11,0) node {$|$};
\draw (5,0) node {$|$};
\draw (-1,0) node {$|$};

\draw (14.5,1) node [above] {High};
\draw (14.5,0) node [above] {frequencies};
\draw (14.5,0) node [below] {$L^2-L^p$};

\draw (7.8,0.9) node [above] {Medium};
\draw (7.8,0) node [above] {frequencies};
\draw (7.8,0) node [below] {$L^p-L^p$};

\draw (2.3,1) node [above] {Low};
\draw (2.35,0) node [above] {frequencies};
\draw (2.35,0) node [below] {$L^p-L^2$};
	\end{tikzpicture}
    \caption{Frequency domain splitting for \eqref{LinearHypNSCzero}}
        \label{fig:3regimes}
\end{figure}
In Figure \ref{fig:3regimes}, for $q, r \in \{2, p\}$, the notation $L^q - L^r$ indicates that the analysis of the partially damped coupling can be conducted in $L^q$-type spaces while the partially diffusive coupling can be studied in $L^r$ spaces.
As $\varepsilon\to0$, $J_\var\to\infty$, so (cf. Figure \ref{fig:2regimes})
\begin{itemize}
    \item The low-frequency regime is not modified.
    \item The medium-frequency regime covers the high-frequency regime.
    \item The high-frequency regime disappears.
\end{itemize}
\begin{figure}[!ht]
    \centering
	\begin{tikzpicture}[xscale=0.4,yscale=0.4, thick]
	\draw [-latex] (-1,0) -- (15,0) ;
\draw  (15,0)  node [below]  {$|\xi|$} ;
\draw  (7,-0.2)  node [below]  {$J_0$};
\draw  (-1,-0.2)  node [below]  {$0$} ;

\draw (7,0) node {$|$};
\draw (-1,0) node {$|$};


\draw (10.7,0.9) node [above] { Medium-high};
\draw (10.7,0) node [above] {frequencies};
\draw (10.7,0) node [below] {$L^p$};

\draw (3.3,1) node [above] {Low};
\draw (3.35,0) node [above] {frequencies};
\draw (3.35,0) node [below] {$L^2$};
	\end{tikzpicture}
    \caption{Frequency decomposition in the limit $\varepsilon\to0$}
        \label{fig:2regimes}
\end{figure}
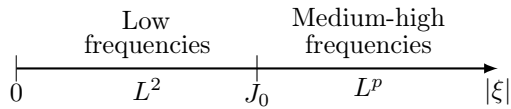
This is consistent as the frequency regimes in the limit correspond to the frequency decomposition used to study the Navier-Stokes-Fourier system \eqref{FullNSC} in \cite{DanchinFullNSC}. Furthermore, the analysis outlined here aligns with the hybrid Besov framework discussed in Section \ref{sec:funcframe}.

\section{Linear analysis: a priori estimates for \eqref{LinearHypNSCzero}}\label{sec:linear}

In this section, we derive a priori estimates for \eqref{LinearHypNSCzero}:
\begin{equation}
\left\{
\begin{array}
[c]{l}%
\d_ta^\varepsilon+\div v^\varepsilon=F^\varepsilon,\\
\d_t v^\varepsilon-\mathcal{A}^\varepsilon v^\varepsilon+\nabla a^\varepsilon +\gamma\nabla \theta^\varepsilon= G^\varepsilon,\\
\d_t \theta^\varepsilon+\beta\div q^\varepsilon+\gamma\div v^\varepsilon=H^\varepsilon,\\
\varepsilon^2\d_tq^\varepsilon+\alpha q^\varepsilon+\kappa\nabla\theta^\varepsilon=\var^2I^\varepsilon,
\end{array}
\right.
\label{LinearHypNSC}
\end{equation}
where the source terms $F^\varepsilon,G^\varepsilon,H^\varepsilon,I^\varepsilon$ are assumed to be smooth functions. In what follows, we assume all constant coefficients to be equal to $1$, except for $\varepsilon$. While the computations for general coefficients can be conducted similarly, we choose this approach to simplify the presentation.

The following proposition gives a priori estimates for the system \eqref{LinearHypNSC}.
\begin{Prop}\label{prop:Aprioriallfreq}
Let $(a^\varepsilon,v^\varepsilon,\theta^\varepsilon,q^\varepsilon)$ be a smooth solution of \eqref{LinearHypNSC} satisfying
\begin{equation}\label{eq:smallZX}X^\varepsilon(t) \ll1.
\end{equation} 
Then, we have
\begin{align} \label{LinearAnalysisEst}
    X^\varepsilon(t)\lesssim &\: X_0+\|(F^\varepsilon,G^\varepsilon,H^\varepsilon,\var I^\varepsilon)\|^\ell_{L^1_T(\dot{B}^{\frac d2-1})}
    \\&+\|F^\varepsilon\|^{m,\var}_{L^1_T(\dot{B}^{\frac{d}{p}}_{p,1})}+\|G^\varepsilon\|^{m,\var}_{L^1_T(\dot{B}^{\frac{d}{p}-1}_{p,1})}+ \| H^\varepsilon\|^{m,\var}_{L^1_T\cap L^2_T( \dot{B}^{\frac{d}{p}-1}_{p,1})}+\varepsilon\|I^\varepsilon\|^{m,\var}_{L^1_T( \dot{B}^{\frac{d}{p}-1}_{p,1})}   
     \nonumber \\&+\|F^\varepsilon\|^{h,\var}_{L^1_T(\dot{B}^{\frac{d}{2}}_{2,1})}+\|G^\varepsilon\|^{h,\var}_{L^1_T(\dot{B}^{\frac{d}{2}-2}_{2,1})}+\|(H^\varepsilon,\varepsilon I^\varepsilon)\|^{h,\var}_{L^1_T(\dot{B}^{\frac{d}{2}}_{2,1})}+\|G^\varepsilon\|^{h,\var}_{L^2_T(\dot{B}^{\frac{d}{2}-1}_{2,1})},  \nonumber
\end{align}
where $X^\varepsilon(t)$ and $X_0^\varepsilon$ are defined in \eqref{X} and \eqref{X0} respectively. 
\end{Prop}
The proof of the Proposition \ref{prop:Aprioriallfreq} will follow directly from Propositions \ref{prop:linlow}, \ref{prop:linmed} and \ref{prop:linhf} related to the analysis of each frequency region. 

Before stating its proof we define the following quantities that play an important role in justifying our bootstrap argument. Let
\begin{align}\label{X}
    X^\varepsilon(t):=X^\ell(t) + X^{m,\varepsilon}(t) + X^{h,\varepsilon}(t),
\end{align}
where, for $Q:=q+\nabla \theta$ and $w:=u-(-\Delta)^{-1}\nabla  a$, the low frequency part reads
\begin{align}\label{Xlow}
   X^\ell(t)=&\|(a^\varepsilon,v^\varepsilon,\theta^\varepsilon,\varepsilon q^\varepsilon)\|^\ell_{\widetilde L^\infty_t(\dot B^{\frac d2-1}_{2,1})}+\|(a^\varepsilon,v^\varepsilon,\theta^\varepsilon)\|^\ell_{L^1_t(\dot B^{\frac d2+1}_{2,1})}+\| q^\varepsilon\|^\ell_{L^1_t(\dot B^{\frac d2}_{2,1})}+\frac{1}{\varepsilon}\| Q^\varepsilon\|^\ell_{L^1_t(\dot B^{\frac d2-1}_{2,1})},
\end{align}
the medium-frequency part is defined by
\begin{align}\label{Xmedium}
    X^{m,\var}(t):=& \|(\theta^\varepsilon,\varepsilon q^\varepsilon)\|^{m,\var}_{L^\infty_T(\dot{B}^{\frac{d}{p}-2}_{p,1}\cap \dot{B}^{\frac{d}{p}-1}_{p,1})}+ \|\theta\|^{m,\var}_{L^1_T( B^{\frac dp}_{p,1}\cap B^{\frac dp+1}_{p,1})}\\&\nonumber+\|q^\varepsilon\|^{m,\var}_{L^1_T( B^{\frac dp-1}_{p,1}\cap B^{\frac dp}_{p,1})}+\|q^\varepsilon\|^{m,\var}_{L^2_T( B^{\frac dp-2}_{p,1}\cap B^{\frac dp-1}_{p,1})}+ \frac{1}{\varepsilon}\|Q^\varepsilon\|^{m,\var}_{L^1_T( B^{\frac dp-2}_{p,1}\cap B^{\frac dp-1}_{p,1})}\\&+\|w^\varepsilon\|^{m,\var}_{L^\infty_T(\dot{B}^{\frac{d}{p}-1}_{p,1})}+ \|w^\varepsilon\|^{m,\var}_{L^1_T( B^{\frac dp+1}_{p,1})}+\|a^\varepsilon\|^{m,\var}_{L^\infty_T\cap L^1_T(\dot{B}^{\frac{d}{p}}_{p,1})}
    \nonumber\\& + \|v^\varepsilon\|^{m,\var}_{L^\infty_T( B^{\frac dp-1}_{p,1}\cap B^{\frac dp}_{p,1})}+ \|v^\varepsilon\|^{m,\var}_{L^1_T( B^{\frac dp+1}_{p,1})}+ \|v^\varepsilon\|^{m,\var}_{L^2_T( B^{\frac dp+1}_{p,1})}\nonumber
\end{align}
and the high-frequency part reads
\begin{align}\label{Xhigh}
    X^{h,\var}(t)&=\varepsilon\|a^\varepsilon\|^{h,\var}_{L^\infty_T\cap L^1_T(\dot{B}^{\frac{d}{2}+1}_{2,1})}+\|(\varepsilon^2\theta^\varepsilon,\var^3 q^\varepsilon)\|^{h,\var}_{L^\infty_T(\dot{B}^{\frac{d}{2}+1}_{2,1})}+ \|(\theta^\varepsilon,\var q^\varepsilon)\|^{h,\var}_{L^1_T( B^{\frac d2+1}_{2,1})}
\\&\quad+\|Q^\varepsilon\|^{h,\var}_{L^1_T( B^{\frac dp}_{p,1})}+\varepsilon\|w^\varepsilon\|^{h,\var}_{L^\infty_T(\dot{B}^{\frac{d}{2}}_{2,1})}+\varepsilon \|w^\varepsilon\|^{h,\var}_{L^1_T( B^{\frac d2+2}_{2,1})}\nonumber\\&\quad+\varepsilon\|v^\varepsilon\|^{h,\var}_{L^\infty_T( B^{\frac d2+1}_{2,1})}+ \varepsilon\|v^\varepsilon\|^{h,\var}_{L^1_T\cap L^2_T( B^{\frac d2+2}_{2,1})}. \nonumber
\end{align}

\subsection{Low-frequency regime}\label{sec:low} Let $j\leq J_0$.
We introduce a suitable unknown that partially diagonalizes the system. We define the effective unknown\footnote{For general coefficients, one should consider $Q^\varepsilon:=\alpha q^\varepsilon+\kappa\nabla\theta^\varepsilon$ here and in the rest of the paper.}
\begin{equation}\label{EffectiveRelax}
  Q^\varepsilon:=q^\varepsilon+\nabla \theta^\varepsilon,
\end{equation}
which satisfies
\begin{equation}\label{eq:EffectiveRelaxQ}
 \d_t\varepsilon Q^\varepsilon+\frac{Q^\varepsilon}{\varepsilon}=\varepsilon f_1^\varepsilon+\varepsilon I^\varepsilon+\varepsilon\kappa \nabla H^\varepsilon,
\end{equation}
where $f_1^\varepsilon=\kappa (\nabla \div q^\varepsilon + \nabla \div v^\varepsilon)$.
Inserting \eqref{eq:EffectiveRelaxQ} in \eqref{LinearHypNSC}, we obtain
\begin{equation}
\left\{
\begin{array}
[c]{l}%
\d_ta^\varepsilon+\div v^\varepsilon=F,\\
\d_t v^\varepsilon+\nabla a^\varepsilon +\nabla \theta^\varepsilon- \Delta v^\varepsilon=G,\\
\d_t \theta^\varepsilon-\Delta \theta^\varepsilon+\div v^\varepsilon=-\div Q^\varepsilon+H,\\
\d_t\varepsilon Q^\varepsilon+\dfrac{Q^\varepsilon}{\varepsilon}:=\varepsilon f_1^\varepsilon+\varepsilon I+\varepsilon\kappa \nabla H.\\
\end{array}
\right.
\label{LinearEffective}
\end{equation}
In this subsection, we prove the following result.
\begin{Prop}\label{prop:linlow}
Let $(a^\varepsilon,v^\varepsilon,\theta^\varepsilon,q^\varepsilon)$ be a smooth solution of \eqref{LinearEffective} such that \eqref{eq:smallZX} holds. We have
\begin{align}\label{est:lowccl}
    X^\ell(t) \leq& \:X_0^\ell+ \|(F^\varepsilon,G^\varepsilon,H^\varepsilon,\var I^\varepsilon)\|^\ell_{L^1_T(\dot{B}^{\frac d2-1}_{2,1})}.
\end{align}
\end{Prop}
\begin{proof}
The equation of $Q^\varepsilon$ can be studied separately from the others. By Lemma \ref{maximalL1L2} and Proposition \ref{prop:Bernstein}, we have
\begin{align}\label{est:QlinearBf1}  \varepsilon\|Q^\varepsilon\|^\ell_{L^\infty_T(\dot{B}^{\frac{d}{2}-1}_{2,1})}+ \frac{1}{\varepsilon}\|Q^\varepsilon\|^\ell_{L^1_T(\dot{B}^{\frac d2-1}_{2,1})}&\leq  \varepsilon\|(q^\varepsilon,v^\varepsilon)\|^\ell_{L^1_T(\dot{B}^{\frac{d}{2}+1}_{2,1})} +\varepsilon\|(I^\varepsilon,\nabla H^\varepsilon)\|^\ell_{L^1_T(\dot{B}^{\frac{d}{2}-1}_{2,1})}\\ \nonumber &\lesssim \varepsilon\|(q^\varepsilon,v^\varepsilon)\|^\ell_{L^1_T(\dot{B}^{\frac{d}{2}+1}_{2,1})} +\varepsilon\|I^\varepsilon\|^\ell_{L^1_T(\dot{B}^{\frac{d}{2}-1}_{2,1})}+K\varepsilon\|H^\varepsilon\|^\ell_{L^1_T(\dot{B}^{\frac{d}{2}-1}_{2,1})}.
\end{align}
To derive a priori estimates for $a^\varepsilon$,$v^\varepsilon$ and $\theta^\varepsilon$, we define the following functional of Lyapunov-type
\begin{align}\label{LyaBF2}
\mathcal{L}^\ell_j=\|( a_j^\varepsilon,v^\varepsilon_j,\theta_j^\varepsilon)\|_{L^2}^2+\frac12\int_\R v^\varepsilon_j\nabla  a_j^\varepsilon \quad \text{for } j\leq J_0.
\end{align}
Using Young inequality and Proposition \ref{prop:Bernstein}, we have
\begin{align} \label{Lyaeq2}
\mathcal{L}^\ell_j\sim \|( a_j^\varepsilon,v_j^\varepsilon,\theta_j^\varepsilon)\|_{L^2}^2.   
\end{align}
Differentiating in time $\mathcal{L}^\ell_j$, we have
\begin{align}\label{feiasgn}
    \dfrac{d}{dt}\mathcal{L}^\ell_j+c2^{2j}\|( a_j^\varepsilon,v_j^\varepsilon,\theta_j^\varepsilon)\|_{L^2}^2\lesssim \|\div Q_j^\varepsilon\|_{L^2}\|\theta_j^\varepsilon\|_{L^2}+\|(F_j^\varepsilon,G_j^\varepsilon,H_j^\varepsilon)\|_{L^2}\|( a_j^\varepsilon,v^\varepsilon_j,\theta_j^\varepsilon)\|_{L^2}.
\end{align}
Employing Lemma \ref{SimpliCarre}, using \eqref{Lyaeq2}, multiplying  the resulting equation by $2^{j(\frac{d}{2}-1)}$ and summing on $j\leq J_0$, we reach
\begin{align}\label{est:QlinearBf2}
     \|(a^\varepsilon,v^\varepsilon,\theta^\varepsilon)\|^\ell_{L^\infty_T(\dot{B}^{\frac{d}{2}-1}_{2,1})}+ \|(a^\varepsilon,v^\varepsilon,\theta^\varepsilon)\|^\ell_{L^1_T(\dot{B}^{\frac d2+1}_{2,1})}\leq& \:\|(a_0^\varepsilon,v_0^\varepsilon,\theta_0^\varepsilon)\|_{\dot{B}^{\frac d2-1}_{2,1}}^\ell+\|Q^\varepsilon\|^\ell_{L^1_T(\dot{B}^{\frac d2}_{2,1})}    \\&+\|F^\varepsilon,G^\varepsilon,H^\varepsilon\|^\ell_{L^1_T(\dot{B}^{\frac d2-1})}. \nonumber
\end{align}
Then, using the Bernstein-type Proposition \ref{prop:Bernstein}, we have
\begin{align}\label{eq:absorbQ}
   \|Q^\varepsilon\|^\ell_{L^1_T(\dot{B}^{\frac d2}_{2,1})}\leq K\|Q^\varepsilon\|^\ell_{L^1_T(\dot{B}^{\frac d2-1}_{2,1})}.
\end{align}
Gathering \eqref{est:QlinearBf1} and \eqref{est:QlinearBf2}, and using \eqref{eq:absorbQ}, we obtain
\begin{align*}\label{est:QlinearBf3}
     \|(a^\varepsilon,v^\varepsilon,\theta^\varepsilon,\varepsilon Q^\varepsilon)\|^\ell_{L^\infty_T(\dot{B}^{\frac{d}{2}-1}_{2,1})}+ &(\frac{1}{\varepsilon}-K)\|Q^\varepsilon\|^\ell_{L^1_T(\dot{B}^{\frac d2-1}_{2,1})}+ \|(a^\varepsilon,v^\varepsilon,\theta^\varepsilon)\|^\ell_{L^1_T(\dot{B}^{\frac d2+1}_{2,1})}\\\leq& \|(a_0^\varepsilon,v_0^\varepsilon,\theta_0^\varepsilon,\varepsilon Q_0^\varepsilon)\|_{\dot{B}^{\frac d2-1}_{2,1}}^\ell +\|F^\varepsilon,G^\varepsilon,H^\varepsilon,\var I^\varepsilon\|^\ell_{L^1_T(\dot{B}^{\frac d2-1})}.
\end{align*}
Since $q^\varepsilon=Q^\varepsilon-\kappa\nabla \theta^\varepsilon$ and $K < 1/\varepsilon$, updating the constants, we have
\begin{align*}
     &\|(a^\varepsilon,v^\varepsilon,\theta^\varepsilon,\varepsilon q^\varepsilon)\|^\ell_{L^\infty_T(\dot{B}^{\frac{d}{2}-1}_{2,1})}+\frac{1}{\varepsilon}\|Q^\varepsilon\|^\ell_{L^1_T(\dot{B}^{\frac d2-1}_{2,1})}+ \|q^\varepsilon\|^\ell_{L^1_T(\dot{B}^{\frac d2}_{2,1})}+ \|(a^\varepsilon,v^\varepsilon,\theta^\varepsilon)\|^\ell_{L^1_T(\dot{B}^{\frac d2+1}_{2,1})}\\&\qquad\qquad\qquad\qquad\qquad\qquad\qquad\lesssim \|(a_0^\varepsilon,v_0^\varepsilon,\theta_0^\varepsilon,\varepsilon Q_0^\varepsilon)\|_{\dot{B}^{\frac d2-1}_{2,1}}^\ell   +\|F^\varepsilon,G^\varepsilon,H^\varepsilon,\var I^\varepsilon\|^\ell_{L^1_T(\dot{B}^{\frac d2-1})}
\end{align*}
which concludes the proof of Proposition \ref{prop:linlow}.
\end{proof}
\subsection{Medium-frequency regime}\label{sec:med} Let $J_0 \leq j \leq J_\var$.
In this intermediate regime, roughly, we will rely on $J_0\leq j$ when dealing with the unknowns $ a^\varepsilon$ and $u^\varepsilon$ and on $j\leq J_\var$ for the unknowns $\theta$ and $q$.
Inspired by the high-frequency analysis performed for the Navier-Stokes systems in \cite{Haspot}, in addition to the damped mode $Q^\varepsilon$, we introduce the effective velocity $w^\varepsilon=v^\varepsilon+(-\Delta)^{-1}\nabla a^\varepsilon$ to further diagonalize the system, it reads
\begin{equation}
\left\{
\begin{array}
[c]{l}%
\d_ta^\varepsilon+a^\varepsilon+u^\varepsilon\cdot\nabla a^\varepsilon=\div w^\varepsilon+F_1^\varepsilon,\\
\d_t w^\varepsilon-\Delta w^\varepsilon =w^\varepsilon-(-\Delta)^{-1}\nabla a^\varepsilon +\nabla \theta^\varepsilon+G^\varepsilon+(-\Delta)^{-1}\nabla F^\varepsilon,\\
\d_t \theta^\varepsilon-\Delta \theta^\varepsilon=-\div w^\varepsilon - a^\varepsilon+\div Q^\varepsilon+H^\varepsilon,\\
\d_t\varepsilon Q^\varepsilon+\dfrac{Q^\varepsilon}{\var}=\var f_1^\varepsilon+\var\nabla a^\varepsilon+\varepsilon I^\varepsilon+\varepsilon\kappa \nabla H^\varepsilon,
\end{array}
\right.
\label{LinearHypNSCMed}
\end{equation}
where $F_1^\varepsilon=F^\varepsilon+u^\varepsilon\cdot \nabla  a^\varepsilon.$ We have the following proposition.
\begin{Prop}\label{prop:linmed}
Let $(a^\varepsilon,v^\varepsilon,\theta^\varepsilon,q^\varepsilon)$ be a smooth solution of \eqref{LinearEffective} such that \eqref{eq:smallZX} holds. We have
\begin{align}\label{est:lowccl2}
    X^{m,\varepsilon} (t)&\lesssim X_0^{m,\varepsilon}+\|F_1^\varepsilon\|^{m,\var}_{L^1_T(\dot{B}^{\frac{d}{p}}_{p,1})}+\|G^\varepsilon\|^{m,\var}_{L^1_T\cap L^2_T(\dot{B}^{\frac{d}{p}-1}_{p,1})}+\varepsilon\|I^\varepsilon\|^{m,\var}_{L^1_T( \dot{B}^{\frac{d}{p}-1}_{p,1})}
   +\| H^\varepsilon\|^{m,\var}_{L^1_T( \dot{B}^{\frac{d}{p}-1}_{p,1})}.
\end{align}
\end{Prop}
\begin{proof}
Since all the linear part of the equations in \eqref{LinearHypNSCMed} are decoupled, up to low or high-order linear source terms, they can be estimated separately and one can derive a priori estimates in a $L^p$ framework.
For the first two equations, employing Lemmas \ref{maximalL1L2} and \ref{CP} gives
 \begin{align}\label{eq:medrho}   \nonumber\|a^\varepsilon\|^{m,\var}_{L^\infty_T(\dot{B}^{\frac{d}{p}}_{p,1})}+  \|a^\varepsilon\|^{m,\var}_{L^1_T(\dot{B}^{\frac dp}_{p,1})}&\lesssim  \|a_0^\varepsilon\|^{m,\var}_{\dot{B}^{\frac{d}{p}}_{p,1}}+  \|w^\varepsilon\|^{m,\var}_{L^1_T(\dot{B}^{\frac{d}{p}+1}_{p,1})}+\|\nabla u\|_{L^1_T(\dot{B}^{d/p}_{p,1})}\|a^\varepsilon\|_{L^\infty_T(\dot{B}^{d/p}_{p,1})}+\|F_1^\varepsilon\|^{m,\var}_{L^1_T(\dot{B}^{\frac{d}{p}}_{p,1})}
 \\&\lesssim  \|a_0^\varepsilon\|^{m,\var}_{\dot{B}^{\frac{d}{p}}_{p,1}}+  \|w^\varepsilon\|^{m,\var}_{L^1_T(\dot{B}^{\frac{d}{p}+1}_{p,1})}+X^\varepsilon(t)^2+\|F_1^\varepsilon\|^{m,\var}_{L^1_T(\dot{B}^{\frac{d}{p}}_{p,1})}
\end{align}
and
\begin{align}\label{eq:medw}   \nonumber\|w^\varepsilon\|^{m,\var}_{L^\infty_T(\dot{B}^{\frac{d}{p}-1}_{p,1})}+ \|w^\varepsilon\|^{m,\var}_{L^1_T(\dot{B}^{\frac dp+1}_{p,1})}&\lesssim \|w_0^\varepsilon\|^{m,\var}_{\dot{B}^{\frac{d}{p}-1}_{p,1}}+ \|w^\varepsilon\|^{m,\var}_{L^1_T(\dot{B}^{\frac{d}{p}-1}_{p,1})}+ \|a^\varepsilon\|^{m,\var}_{L^1_T(\dot{B}^{\frac{d}{p}-2}_{p,1})}\\ &+ \|\theta^\varepsilon\|^{m,\var}_{L^1_T(\dot{B}^{\frac{d}{p}}_{p,1})}     +\|G^\varepsilon\|^{m,\var}_{L^1_T(\dot{B}^{\frac{d}{p}-1}_{p,1})}+X^\varepsilon(t)^2+\|F_1^\varepsilon\|^{m,\var}_{L^1_T(\dot{B}^{\frac{d}{p}-2}_{p,1})} .
 \end{align}
Thanks to Proposition \ref{prop:Bernstein}, we have
\begin{align} \label{eq:medbernstein}
\|a^\varepsilon\|^{m,\var}_{L^1_T(\dot{B}^{\frac{d}{p}-2}_{p,1})}\lesssim \dfrac{1}{K^2}\|a^\varepsilon\|^{m,\var}_{L^1_T(\dot{B}^{\frac{d}{p}}_{p,1})}\andf\|w^\varepsilon\|^{m,\var}_{L^1_T(\dot{B}^{\frac{d}{p}-1}_{p,1})}\lesssim \dfrac{1}{K^2}\|w^\varepsilon\|^{m,\var}_{L^1_T(\dot{B}^{\frac{d}{p}+1}_{p,1})}. 
\end{align}
Multiplying \eqref{eq:medrho} by a small fixed constant, adding it to \eqref{eq:medw}, using \eqref{eq:medbernstein} and choosing $K$ large enough so that $1/K^2>1$, we get
\begin{align}\label{eq:medwrho}  \nonumber\|w^\varepsilon\|^{m,\var}_{L^\infty_T(\dot{B}^{\frac{d}{p}-1}_{p,1})}+\|w^\varepsilon\|^{m,\var}_{L^1_T(\dot{B}^{\frac dp+1}_{p,1})}+\|a^\varepsilon\|^{m,\var}_{L^\infty_T\cap L^1_T(\dot{B}^{\frac{d}{p}}_{p,1})}\lesssim& \:\|w_0^\varepsilon\|^{m,\var}_{\dot{B}^{\frac{d}{p}-1}_{p,1}}+\|a_0^\varepsilon\|^{m,\var}_{\dot{B}^{\frac{d}{p}}_{p,1}}+ \|\theta^\varepsilon\|^{m,\var}_{L^1_T(\dot{B}^{\frac{d}{p}}_{p,1})}
       \\ &+\|G^\varepsilon\|^{m,\var}_{L^1_T(\dot{B}^{\frac{d}{p}-1}_{p,1})}+X^\varepsilon(t)^2+\|F_1^\varepsilon\|^{m,\var}_{L^1_T(\dot{B}^{\frac{d}{p}}_{p,1})}. 
\end{align}
For $\theta^\varepsilon$ and $Q^\varepsilon$, we have
\begin{align*}
   \|\theta^\varepsilon\|^{m,\var}_{L^\infty_T(\dot{B}^{\frac{d}{p}-2}_{p,1})}+ \|\theta^\varepsilon\|^{m,\var}_{L^1_T(\dot{B}^{\frac dp}_{p,1})}&\lesssim \|\theta_0\|^{m,\var}_{\dot{B}^{\frac{d}{p}-2}_{p,1}}+ \|w^\varepsilon\|^{m,\var}_{L^1_T(\dot{B}^{\frac{d}{p}-1}_{p,1})}+\|a^\varepsilon\|^{m,\var}_{L^1_T(\dot{B}^{\frac{d}{p}-2}_{p,1})}\\&\quad+ \|Q^\varepsilon\|^{m,\var}_{L^1_T(\dot{B}^{\frac{d}{p}-1}_{p,1})}+\|H^\varepsilon\|^{m,\var}_{L^1_T( \dot{B}^{\frac{d}{p}-2}_{p,1})}
\end{align*}
and 
\begin{align*}
      \varepsilon\|Q^\varepsilon\|^{m,\var}_{L^\infty_T(\dot{B}^{\frac{d}{p}-2}_{p,1})}+ \frac{1}{\varepsilon}\|Q^\varepsilon\|^{m,\var}_{L^1_T(\dot{B}^{\frac dp-2}_{p,1})}\leq& X_0^\varepsilon+\varepsilon\|q^\varepsilon\|^{m,\var}_{L^1_T(\dot{B}^{\frac{d}{p}}_{p,1})}+ \var\|w^\varepsilon\|^{m,\var}_{L^1_T(\dot{B}^{\frac{d}{p}}_{p,1})}+\varepsilon\|a^\varepsilon\|^{m,\var}_{L^1_T(\dot{B}^{\frac{d}{p}-1}_{p,1})} \\&+\varepsilon\|I^\varepsilon\|^{m,\var}_{L^1_T( \dot{B}^{\frac{d}{p}-2}_{p,1})}
     +\var \|\nabla H^\varepsilon\|^{m,\var}_{L^1_T( \dot{B}^{\frac{d}{p}-2}_{p,1})}.
\end{align*}
Using Proposition \ref{prop:Bernstein}, we get
\begin{align*}
   \|\theta^\varepsilon\|^{m,\var}_{L^\infty_T(\dot{B}^{\frac{d}{p}-2}_{p,1})}+ \|\theta^\varepsilon\|^{m,\var}_{L^1_T(\dot{B}^{\frac dp}_{p,1})}&\leq \|\theta_0^\varepsilon\|^{m,\var}_{\dot{B}^{\frac{d}{p}-2}_{p,1}} +\dfrac{1}{K^2}(\|w^\varepsilon\|^{m,\var}_{L^1_T(\dot{B}^{\frac{d}{p}+1}_{p,1})}+\|a^\varepsilon\|^{m,\var}_{L^1_T(\dot{B}^{\frac{d}{p}}_{p,1})})\\&\quad+\frac{\kappa}{\var} \|Q^\varepsilon\|^{m,\var}_{L^1_T(\dot{B}^{\frac{d}{p}-2}_{p,1})}+\|H^\varepsilon\|^{m,\var}_{L^1_T( \dot{B}^{\frac{d}{p}-2}_{p,1})}
\end{align*}
and
\begin{align*}
      \varepsilon\|Q^\varepsilon\|^{m,\var}_{L^\infty_T(\dot{B}^{\frac{d}{p}-2}_{p,1})}+ \frac{1}{\varepsilon}\|Q^\varepsilon\|^{m,\var}_{L^1_T(\dot{B}^{\frac dp-2}_{p,1})}&\lesssim
 X_0^\varepsilon+\dfrac{k^2}{\varepsilon}\|Q^\varepsilon\|^{m,\var}_{L^1_T(\dot{B}^{\frac{d}{p}-2}_{p,1})}+ \dfrac{\var}{K}(\|w^\varepsilon\|^{m,\var}_{L^1_T(\dot{B}^{\frac{d}{p}+1}_{p,1})}+\|a^\varepsilon\|^{m,\var}_{L^1_T(\dot{B}^{\frac{d}{p}}_{p,1})}) \\&+k\|\theta^\varepsilon\|^{m,\var}_{L^1_T(\dot{B}^{\frac{d}{p}}_{p,1})}+\varepsilon\|I^\varepsilon\|^{m,\var}_{L^1_T( \dot{B}^{\frac{d}{p}-2}_{p,1})}
     +\|H^\varepsilon\|^{m,\var}_{L^1_T( \dot{B}^{\frac{d}{p}-2}_{p,1})},
\end{align*}
where we used that 
$$\varepsilon\|q^\varepsilon\|^{m,\var}_{L^1_T(\dot{B}^{\frac{d}{p}}_{p,1})}\leq \dfrac{\kappa^2}{\varepsilon}\|Q^\varepsilon\|^{m,\var}_{L^1_T(\dot{B}^{\frac{d}{p}-2}_{p,1})}+k\|\theta^\varepsilon\|^{m,\var}_{L^1_T(\dot{B}^{\frac{d}{p}}_{p,1})}.$$
Adding the above estimates for $\theta^\varepsilon$ and $Q^\varepsilon$ with \eqref{eq:medwrho} multiplied by $\frac12$, and adjusting the constant $k,\varepsilon$ and $K$ so that the linear right-hand side terms can be absorbed by the left-hand side terms, we get
\begin{align}
    \label{eq:mednoq}
    &\|w^\varepsilon\|^{m,\var}_{L^\infty_T(\dot{B}^{\frac{d}{p}-1}_{p,1})}+\|w^\varepsilon\|^{m,\var}_{L^1_T(\dot{B}^{\frac dp+1}_{p,1})}+\|a^\varepsilon\|^{m,\var}_{L^\infty_T\cap L^1_T(\dot{B}^{\frac{d}{p}}_{p,1})}+ \|\theta^\varepsilon\|^{m,\var}_{L^\infty_T(\dot{B}^{\frac{d}{p}-2}_{p,1})}\\&\nonumber+ \|\theta^\varepsilon\|^{m,\var}_{L^1_T(\dot{B}^{\frac dp}_{p,1})}+ \varepsilon\|Q^\varepsilon\|^{m,\var}_{L^\infty_T(\dot{B}^{\frac{d}{p}-2}_{p,1})}+ \frac{1}{\varepsilon}\|Q^\varepsilon\|^{m,\var}_{L^1_T(\dot{B}^{\frac dp-2}_{p,1})}\\& \lesssim X_0^{m,\varepsilon}+X^\varepsilon(t)^2+\|G^\varepsilon\|^{m,\var}_{L^1_T(\dot{B}^{\frac{d}{p}-1}_{p,1})}+\|F_1^\varepsilon\|^{m,\var}_{L^1_T(\dot{B}^{\frac{d}{p}}_{p,1})}+\|H^\varepsilon\|^{m,\var}_{L^1_T( \dot{B}^{\frac{d}{p}-2}_{p,1})}+ \varepsilon\|I^\varepsilon\|^{m,\var}_{L^1_T( \dot{B}^{\frac{d}{p}-2}_{p,1})}. \nonumber
\end{align}
Then, using that $q^\varepsilon=Q^\varepsilon-\kappa\nabla \theta^\varepsilon$, we recover
\begin{align}\label{eq:medq}
      \varepsilon\|q^\varepsilon\|^{m,\var}_{L^\infty_T(\dot{B}^{\frac{d}{p}-2}_{p,1})}+ \|q^\varepsilon\|^{m,\var}_{L^2_T(\dot{B}^{\frac dp-2}_{p,1})}+ \|q^\varepsilon\|^{m,\var}_{L^1_T(\dot{B}^{\frac dp-1}_{p,1})}\lesssim X_0^\varepsilon+\varepsilon\|I^\varepsilon\|^{m,\var}_{L^1_T( \dot{B}^{\frac{d}{p}-2}_{p,1})}+\|H^\varepsilon\|^{m,\var}_{L^1_T( \dot{B}^{\frac{d}{p}-2}_{p,1})}.
\end{align}
To control the nonlinear terms, additional regularity properties are necessary for $\theta^\varepsilon$, $Q^\varepsilon$ and $v^\varepsilon$.
\smallbreak

\noindent\textbf{Additional regularity for $\theta^\varepsilon$ and $Q^\varepsilon$.} Performing similar computations at a higher regularity index, we derive
\begin{align}\label{addTheta}\|\theta^\varepsilon\|^{m,\var}_{L^\infty_T(\dot{B}^{\frac{d}{p}-1}_{p,1})}+ \|\theta^\varepsilon\|^{m,\var}_{L^1_T(\dot{B}^{\frac dp+1}_{p,1})}&\lesssim \|\theta_0^\varepsilon\|^{m,\var}_{\dot{B}^{\frac{d}{p}-1}_{p,1}} +\dfrac{1}{K}(\|w^\varepsilon\|^{m,\var}_{L^1_T(\dot{B}^{\frac{d}{p}+1}_{p,1})}+\|a^\varepsilon\|^{m,\var}_{L^1_T(\dot{B}^{\frac{d}{p}}_{p,1})})\\&\quad+\frac{\kappa}{\var} \|Q^\varepsilon\|^{m,\var}_{L^1_T(\dot{B}^{\frac{d}{p}-1}_{p,1})}+\|H^\varepsilon\|^{m,\var}_{L^1_T( \dot{B}^{\frac{d}{p}-1}_{p,1})} \nonumber
\end{align}
and
\begin{align}   \label{addQ} \varepsilon\|Q^\varepsilon\|^{m,\var}_{L^\infty_T(\dot{B}^{\frac{d}{p}-1}_{p,1})}+ \frac{1}{\varepsilon}\|Q^\varepsilon\|^{m,\var}_{L^1_T(\dot{B}^{\frac dp-1}_{p,1})}\lesssim& \:X_0^\varepsilon+\dfrac{k^2}{\varepsilon}\|Q^\varepsilon\|^{m,\var}_{L^1_T(\dot{B}^{\frac{d}{p}-1}_{p,1})}+k\|\theta^\varepsilon\|^{m,\var}_{L^1_T(\dot{B}^{\frac{d}{p}+1}_{p,1})}\\&+ \varepsilon(\|w^\varepsilon\|^{m,\var}_{L^1_T(\dot{B}^{\frac{d}{p}+1}_{p,1})}+\|a^\varepsilon\|^{m,\var}_{L^1_T(\dot{B}^{\frac{d}{p}}_{p,1})}) \nonumber\\&+\varepsilon\|I^\varepsilon\|^{m,\var}_{L^1_T( \dot{B}^{\frac{d}{p}-1}_{p,1})}
     +\|H^\varepsilon\|^{m,\var}_{L^1_T( \dot{B}^{\frac{d}{p}-1}_{p,1})}.\nonumber
\end{align}
Using that $q^\varepsilon=Q^\varepsilon-\kappa\nabla \theta^\varepsilon$, we have
\begin{align}\label{addq}     \|\varepsilon q^\varepsilon\|^{m,\var}_{L^\infty_T(\dot{B}^{\frac{d}{p}-1}_{p,1})}+ \|q^\varepsilon\|^{m,\var}_{L^2_T(\dot{B}^{\frac dp-1}_{p,1})}+ \|q^\varepsilon\|^{m,\var}_{L^1_T(\dot{B}^{\frac dp}_{p,1})}\lesssim X_0^\varepsilon+\varepsilon\|I^\varepsilon\|^{m,\var}_{L^1_T( \dot{B}^{\frac{d}{p}-1}_{p,1})}+\|H^\varepsilon\|^{m,\var}_{L^1_T( \dot{B}^{\frac{d}{p}-1}_{p,1})}.
\end{align}

\noindent\textbf{Additional regularity for $v^\varepsilon$.} The equation of $v^\varepsilon$ reads
$$
\d_t v^\varepsilon-\Delta v^\varepsilon+\nabla a^\varepsilon +\nabla \theta^\varepsilon= G^\varepsilon.
$$
Using Lemma \ref{maximalL1L2}, we obtain
\begin{align}\label{eq:reguuh}   \|v^\varepsilon\|^{m,\var}_{L^\infty_T(\dot{B}^{\frac{d}{p}}_{p,1})}+ \|v^\varepsilon\|^{m,\var}_{L^2_T(\dot{B}^{\frac{d}{p}+1}_{p,1})}\leq \|v^\varepsilon_0\|^{m,\var}_{\dot{B}^{\frac{d}{p}}_{p,1}}+\|(\nabla a^\varepsilon,\nabla \theta^\varepsilon)\|^{m,\var}_{L^2_T(\dot{B}^{\frac{d}{p}-1}_{p,1})}+\|G^\varepsilon\|^{m,\var}_{L^2_T(\dot{B}^{\frac{d}{p}-1}_{p,1})}.
\end{align}
\noindent\textbf{Conclusion of the proof of Proposition \ref{prop:linmed}.}
Adding \eqref{addTheta}-\eqref{addQ}-\eqref{addq} to \eqref{eq:mednoq} and \eqref{eq:medq}, and using \eqref{eq:smallZX} and that
$$
\varepsilon\|I^\varepsilon\|^{m,\var}_{L^1_T( \dot{B}^{\frac{d}{p}-2}_{p,1})}+\|H^\varepsilon\|^{m,\var}_{L^1_T( \dot{B}^{\frac{d}{p}-2}_{p,1})} \leq \frac{1}{K}\left(\varepsilon\|I^\varepsilon\|^{m,\var}_{L^1_T( \dot{B}^{\frac{d}{p}-1}_{p,1})}+\|H^\varepsilon\|^{m,\var}_{L^1_T( \dot{B}^{\frac{d}{p}-1}_{p,1})}\right),
$$
conclude the proof of Proposition \ref{prop:linmed} when $\varepsilon$ and $\kappa$ are chosen small enough and $K$ large enough.
\end{proof}

\subsection{High-frequency regime} Let $j\geq J_\var$.
In this regime, we cannot use the unknown $Q$ to partially diagonalize the system but using $w^\varepsilon=v^\varepsilon+(-\Delta)^{-1}\nabla a^\varepsilon$ is still effective. The linear system we are interested in reads
\begin{equation}
\left\{
\begin{array}
[c]{l}%
\d_ta^\varepsilon+a^\varepsilon+v^\varepsilon\cdot \nabla a^\varepsilon=\div w^\varepsilon+F_1^\varepsilon,\\
\d_t w^\varepsilon-\Delta w^\varepsilon =w^\varepsilon-(-\Delta)^{-1}\nabla a^\varepsilon +\nabla \theta^\varepsilon+(-\Delta)^{-1}\nabla F^\varepsilon+G^\varepsilon,\\
\d_t \theta^\varepsilon+v^\varepsilon\cdot\nabla \theta^\varepsilon+(1+J(a^\varepsilon))\div q^\varepsilon+\div w^\varepsilon=H_1^\varepsilon,\\
\varepsilon^2\d_tq^\varepsilon+\varepsilon
^2 v^\varepsilon\cdot\nabla q^{\varepsilon}+q^\varepsilon+\nabla\theta^\varepsilon=\varepsilon^2I_1^\varepsilon,
\end{array}
\right.
\label{LinearHypNSC3}
\end{equation}
where $I_1^\varepsilon=I^\varepsilon+v^\varepsilon\cdot\nabla q^{\varepsilon}$ and $H_1^\varepsilon=H^\varepsilon-J(a^\varepsilon))\div q^\varepsilon+v^\varepsilon\cdot\nabla \theta^\varepsilon$. We prove the following statement.
\begin{Prop}\label{prop:linhf}
Let $(a^\varepsilon,v^\varepsilon,\theta^\varepsilon,q^\varepsilon)$ be a smooth solution of \eqref{LinearEffective} such that \eqref{eq:smallZX} holds. We have
\begin{align}\label{est:highccl}
    X^h(t)& \leq X_0^h+\varepsilon\|(G^\varepsilon,H^\varepsilon)\|^{h,\var}_{L^1_T(\dot{B}^{\frac{d}{2}}_{2,1})}+\varepsilon\|G^\varepsilon\|^{h,\var}_{L^2_T(\dot{B}^{\frac{d}{2}}_{2,1})}+\|(\varepsilon F_1^\varepsilon,\varepsilon^2H_1^\varepsilon, \varepsilon^3I^\varepsilon,\varepsilon^2 R_1^\varepsilon)\|^{h,\var}_{L^1_T(\dot{B}^{\frac{d}{2}+1}_{2,1})} \nonumber\\&\quad +\varepsilon\|\d_tJ(a^\varepsilon)\|_{L^\infty_T(\dot{B}^{\frac{d}{2}}_{2,1})} \|\var^2 q^\varepsilon\|^{h,\var}_{L^1_T(\dot{B}^{\frac{d}{2}+1}_{2,1})}. 
\end{align}
\end{Prop}

\begin{proof}

Since $j\geq J_\var$ implies $j\geq J_0$, the estimates for $a^\varepsilon$ and $w^\varepsilon$ follow the same lines as in the previous section, we have
\begin{align}\label{eq:highwrho2}    \|w^\varepsilon\|^{h,\var}_{L^\infty_T(\dot{B}^{\frac{d}{2}}_{2,1})}+\|w^\varepsilon\|^{h,\var}_{L^1_T(\dot{B}^{\frac d2+2}_{2,1})}+\|a^\varepsilon\|^{h,\var}_{L^\infty_T\cap L^1_T(\dot{B}^{\frac{d}{2}+1}_{2,1})}\leq& \|w_0^\varepsilon\|^{h,\var}_{\dot{B}^{\frac{d}{2}}_{2,1}}+\|a_0^\varepsilon\|^{h,\var}_{\dot{B}^{\frac{d}{2}+1}_{2,1}}+ \|\theta^\varepsilon\|^{h,\var}_{L^1_T(\dot{B}^{\frac{d}{2}+1}_{2,1})} 
       \\ &+\|G^\varepsilon\|^{h,\var}_{L^1_T(\dot{B}^{\frac{d}{2}}_{2,1})}+\|F_1^\varepsilon\|^{h,\var}_{L^1_T(\dot{B}^{\frac{d}{2}+1}_{2,1})}. \nonumber
\end{align}
Concerning $\theta$ and $q$, we define the functional of Lyapunov-type \begin{align}\label{LyaHF}
\mathcal{L}_j^h=\int_{\R^d}|\theta_j^\varepsilon|^2+\int_{\R^d}(1+J(a^\varepsilon)) |q_j^\varepsilon|^2+2^{-2j}\eta\int_{\R^d}  q_j^\varepsilon\cdot \nabla \theta_j^\varepsilon \quad \text{for }j\geq J_\varepsilon \end{align}
and for $\eta>0$ a positive constant to be adjusted.
Differentiating in time $\mathcal{L}_j^h$, we get
\begin{align*}\frac{1}{2}\dfrac{d}{dt}\left(\int_{\R^d} |\theta_j^\varepsilon|^2+\int_{\R^d}(1+J(a^\varepsilon)) |\varepsilon q_j^\varepsilon|^2\right)+\dfrac{1}{\varepsilon^2}\|\varepsilon q_j^\varepsilon\|_{L^2}^2
&\lesssim \int_{\R^d} \div w_j^\varepsilon \cdot \theta_j^\varepsilon + \int_{\R^d} \ddj(v^\varepsilon\cdot \nabla q^\varepsilon)q_j^\varepsilon(1+J(a^\varepsilon)) 
\\&\quad + \int_{\R^d} \ddj(v^\varepsilon\cdot \nabla \theta^\varepsilon)q_j^\varepsilon(1+J(a^\varepsilon)) 
\\&\quad+\int_{\R^d}H_{1,j}^\varepsilon\cdot\theta_j^\varepsilon+\int_{\R^d}\varepsilon^2 I_{1,j}^\varepsilon\cdot q_j^\varepsilon(1+J(a^\varepsilon)) 
\\\quad-\int_{\R^d}&\int_{\R^d}J(a^\varepsilon)|q_j^\varepsilon|^2 -\int_{\R^d}\d_t (J(a^\varepsilon))|\varepsilon q_j^\varepsilon|^2 +\int_{\R^d}R_j^\varepsilon \theta_j^\varepsilon,
\end{align*}
where $R_j^\varepsilon=[J(a^\varepsilon),\ddj]\div q^\varepsilon$.
Then, using Lemma \ref{CP} to deal with the advection terms, we obtain
\begin{align*}\frac{1}{2}\dfrac{d}{dt}\left(\int_{\R^d} |\theta_j^\varepsilon|^2+\int_{\R^d}(1+J(a^\varepsilon)) |q_j^\varepsilon|^2\right)+\dfrac{1}{\varepsilon^2}\|\varepsilon q_j^\varepsilon\|_{L^2}^2 &\lesssim\|\div w_j^\varepsilon\|_{L^2}\|\theta_j^\varepsilon\|_{L^2}\\&\quad+c_j2^{-j\frac d2}X^\varepsilon(t)^2\|q^\varepsilon\|_{L^2}(1+\|J(a^\varepsilon)\|_{L^\infty})
\\&\quad+\|(H^\varepsilon_1,\varepsilon I_{1,j}^\varepsilon,R_j^\varepsilon)\|_{L^2}\|(\theta_j^\varepsilon,\varepsilon q_j^\varepsilon)\|_{L^2}
\\&\quad+\|(J(a^\varepsilon),\d_tJ(a^\varepsilon))\|_{L^\infty}\|q_j^\varepsilon\|_{L^2}^2,
\end{align*} where $(c_j)_{j\geq J_\varepsilon}$ is a sequence such that $\sum_{j\geq J_\varepsilon}c_j=1$. Then, thanks to \eqref{eq:smallZX} and composition estimates, we have $\|J(a^\varepsilon)\|_{L^\infty}\ll1$ and
\begin{align}\label{eq:415}\frac{1}{2}\dfrac{d}{dt}\left(\int_{\R^d}|\theta_j^\varepsilon|^2+\int_{\R^d}(1+J(a^\varepsilon)) |q_j^\varepsilon|^2\right)+\dfrac{1}{\varepsilon^2}\|\varepsilon q_j^\varepsilon\|_{L^2}^2 &\lesssim \|\div w_j^\varepsilon\|_{L^2}\|\theta_j^\varepsilon\|_{L^2}+c_j2^{-j\frac d2}X^\varepsilon(t)^2\|q^\varepsilon\|_{L^2}
\\&\quad+\|(H^\varepsilon_1,\varepsilon I_{1,j}^\varepsilon,R_j^\varepsilon)\|_{L^2}\|(\theta_j^\varepsilon,\varepsilon q_j^\varepsilon)\|_{L^2}\nonumber
\\&\quad + \|\d_tJ(a^\varepsilon))\|_{L^\infty}\|q_j^\varepsilon\|_{L^2}^2. \nonumber
\end{align}
Differentiating in time the third term of the Lyapunov functional \eqref{LyaHF}, we obtain
\begin{align*}
\dfrac{d}{dt}\int_{\R^d} q_j^\varepsilon\cdot \nabla \theta_j^\varepsilon+\frac{1}{\varepsilon^2}\|\nabla \theta_j^\varepsilon\|_{L^2}^2&\lesssim \|\div q_j^\varepsilon\|_{L^2}^2+ \frac{1}{\varepsilon^2}\int_{\R^d} q_j^\varepsilon\cdot \nabla \theta_j^\varepsilon+\int_{\R^d} \div w_j^\varepsilon\cdot \div q_j^\varepsilon
\\&\quad +\int_{\R^d}\nabla H_j^\varepsilon\cdot q_j^\varepsilon+\int_{\R^d}I_j^\varepsilon\cdot \nabla\theta_j^\varepsilon.
\end{align*}
Using Cauchy-Schwarz, Young and Bernstein inequalities, we have
\begin{align*}
     2^{-2j}\frac{\eta}{\varepsilon^2}\int_{\R^d}q_j^\varepsilon\cdot\nabla\theta_j^\varepsilon&\lesssim 2^{-2j}\frac{\eta}{\varepsilon^2}\|q_j^\varepsilon\|_{L^2}\|\nabla \theta_j^\varepsilon\|_{L^2}\\&\lesssim 2^{-j}\dfrac{\eta}{\varepsilon^2}\|q_j^\varepsilon\|_{L^2}\|\theta_j\|_{L^2} \\&\lesssim \frac{k^{-1}\eta}{\varepsilon}\|q_j^\varepsilon\|_{L^2}\|\theta_j^\varepsilon\|_{L^2}
    \\&\lesssim \frac{k^{-2}\eta^2}{2}\|q_j^\varepsilon\|_{L^2}^2+\frac{1}{2\varepsilon^2}\|\theta_j^\varepsilon\|^2_{L^2}.
\end{align*}
Similarly,
$$2^{-2j}\int_{\R^d} \div w_j^\varepsilon\cdot \div q_j^\varepsilon\lesssim\dfrac{1}{2}\|w_j^\varepsilon\|_{L^2}^2+\dfrac{1}{2
}\|q_j^\varepsilon\|_{L^2}^2.$$
Thus, choosing $\eta$ small enough such that $\eta\leq k$, we infer \begin{align} \label{eq:lyahf1}
    \dfrac{d}{dt}\mathcal{L}^2_j+\frac{1}{\varepsilon^2}\|(\theta_j^\varepsilon,\varepsilon q_j^\varepsilon)\|_{L^2}^2 \lesssim&\:\|\div w_j^\varepsilon\|_{L^2}\|\theta_j\|_{L^2}+c_j2^{-j\frac d2}X^\varepsilon(t)^2\\&+\|(H^\varepsilon_1,\varepsilon I_{1,j}^\varepsilon,R_j^\varepsilon)\|_{L^2}\|(\theta_j^\varepsilon,\varepsilon q_j^\varepsilon)\|_{L^2}+\dfrac{1}{\varepsilon}2^{-j}\|H_j^\varepsilon\|_{L^2}\| \varepsilon q_j^\varepsilon\|_{L^2} \nonumber\\&+\|\d_tJ(a^\varepsilon))\|_{L^\infty}\|\varepsilon q_j^\varepsilon\|_{L^2}^2,\nonumber
\end{align}
where we used that 
$$2^{-2j}\int_{\R^d}\nabla H_j^\varepsilon\cdot q_j^\varepsilon+2^{-2j}\int_{\R^d}I_j^\varepsilon\cdot \nabla\theta_j^\varepsilon\lesssim \dfrac{1}{\varepsilon}2^{-j}\|H_j^\varepsilon\|_{L^2}\|\varepsilon q_j^\varepsilon\|_{L^2}+\|\varepsilon I_{1,j}^\varepsilon\|_{L^2}\|\theta_j^\varepsilon\|_{L^2}. $$
Then, we prove the following lemma
\begin{Lemme} \label{Lyaeq}
The function $\mathcal{L}^h_j$ is equivalent to the $L^2$-norm of the solution, we have $$\mathcal{L}^h_j\sim \|(\theta_j^\varepsilon,\varepsilon  q_j^\varepsilon)\|_{L^2}^2.$$ 
\end{Lemme}
\begin{proof}
Using Young's inequality, we obtain
\begin{align*} 
    2^{-2j}\int_{\R^d}  q_j^\varepsilon\cdot \nabla \theta_j^\varepsilon &\lesssim 2^{-2j}(\|q_j^\varepsilon\|_{L^2}^2+2^{2j}\|\theta_j^\varepsilon\|_{L^2}^2)
    \\&\lesssim 2^{-2j}\|q_j^\varepsilon\|_{L^2}^2+\|\theta_j^\varepsilon\|_{L^2}^2
    \\&\lesssim \varepsilon^2\| q_j^\varepsilon\|_{L^2}^2+\|\theta_j^\varepsilon\|_{L^2}^2
    \\&\lesssim \|(\theta_j^\varepsilon,\varepsilon  q_j^\varepsilon)\|_{L^2}^2
\end{align*}
and since \eqref{eq:smallZX} implies that $\|J(a^\varepsilon)\|_{L^\infty}\ll1$, we obtain the desired result.
\end{proof}
Applying Lemma \ref{Lyaeq} and Lemma \ref{SimpliCarre} to \eqref{eq:lyahf1}, we get
\begin{align} \label{test1234}
 \|(\theta^\varepsilon,\var q^\varepsilon)\|^{h,\var}_{L^\infty_T(\dot{B}^{\frac{d}{2}+1}_{2,1})}+\dfrac{1}{\varepsilon^2}\|(\theta^\varepsilon,\var q^\varepsilon)\|^{h,\var}_{L^1_T(\dot{B}^{\frac d2+1}_{2,1})}&\lesssim \|(\theta_0^\varepsilon,\varepsilon q_0^\varepsilon)\|^{h,\var}_{\dot{B}^{\frac{d}{2}+1}_{2,1}}+\|w^\varepsilon\|^{h,\var}_{L^1_T(\dot{B}^{\frac{d}{2}+2}_{2,1})}\\&+\|(H_1^\varepsilon,\varepsilon I^\varepsilon,R^\varepsilon)\|^{h,\var}_{L^1_T(\dot{B}^{\frac{d}{2}+1}_{2,1})}
+\dfrac{1}{\varepsilon}\|H^\varepsilon\|^{h,\var}_{L^1_T(\dot{B}^{\frac{d}{2}}_{2,1})} \nonumber
\\&+\int_0^T\|\d_tJ(a^\varepsilon)\|_{\dot{B}^{\frac{d}{2}}_{2,1}} \|\var q^\varepsilon\|^{h,\var}_{\dot{B}^{\frac{d}{2}+1}_{2,1}}+X^\varepsilon(t)^2. \nonumber
\end{align}
Multiplying \eqref{test1234} by $\varepsilon$, adding it to \eqref{eq:highwrho2}, yields
\begin{align}\label{eq:highlast2}\nonumber
     \|w^\varepsilon\|^{h,\var}_{L^\infty_T(\dot{B}^{\frac{d}{2}}_{2,1})}+ \|(\varepsilon\theta^\varepsilon,\var^2 q^\varepsilon)\|^{h,\var}_{L^\infty_T(\dot{B}^{\frac{d}{2}+1}_{2,1})}+ \|w^\varepsilon\|^{h,\var}_{L^1_T(\dot{B}^{\frac{d}{2}+2}_{2,1})} +&\dfrac{1}{\varepsilon}\|(\theta^\varepsilon,\var q^\varepsilon)\|^{h,\var}_{L^1_T(\dot{B}^{\frac d2+1}_{2,1})}+\|a^\varepsilon\|^{h,\var}_{L^\infty_T\cap L^1_T(\dot{B}^{\frac{d}{2}+1}_{2,1})}\\\nonumber&\lesssim\|w_0^\varepsilon\|^{h,\var}_{\dot{B}^{\frac{d}{2}}_{2,1}}+ \|(\varepsilon\theta_0^\varepsilon,\varepsilon^2 q_0^\varepsilon,a_0^\varepsilon)\|^{h,\var}_{\dot{B}^{\frac{d}{2}+1}_{2,1}}\nonumber
     \\&+X^\varepsilon(t)^2+\|G^\varepsilon\|^{h,\var}_{L^1_T(\dot{B}^{\frac{d}{2}}_{2,1})}  +\|F^\varepsilon\|^{h,\var}_{L^1_T(\dot{B}^{\frac{d}{2}}_{2,1})}
     \\&\nonumber+\|(\varepsilon H_1^\varepsilon, \varepsilon^2I^\varepsilon,\varepsilon R_1^\varepsilon)\|^{h,\var}_{L^1_T(\dot{B}^{\frac{d}{2}+1}_{2,1})} \nonumber+\|H^\varepsilon\|^{h,\var}_{L^1_T(\dot{B}^{\frac{d}{2}}_{2,1})}\\&+\int_0^T\|\d_tJ(a^\varepsilon)\|_{\dot{B}^{\frac{d}{2}}_{2,1}} \|\var^2 q^\varepsilon\|^{h,\var}_{\dot{B}^{\frac{d}{2}+1}_{2,1}}.\nonumber
\end{align}
Again, to deal with nonlinearities we need additional $L^2$-in-time information for $v^\varepsilon$.
\smallbreak
\noindent\textbf{Additional regularity for the velocity.}
 Recall that  $v^\varepsilon$ satisfies
\begin{align}
    \label{eq:v123}
\d_t v^\varepsilon-\Delta v^\varepsilon+\nabla a^\varepsilon +\nabla \theta^\varepsilon= G^\varepsilon.
\end{align}
Applying Lemma \ref{maximalL1L2} to \eqref{eq:v123} yields
\begin{align}\label{eq:regu2}
     \|v^\varepsilon\|^{h,\var}_{L^\infty_T(\dot{B}^{\frac{d}{2}+1}_{2,1})}+ \|v^\varepsilon\|^{h,\var}_{L^2_T(\dot{B}^{\frac{d}{2}+1}_{2,1})}\leq \|v_0^\varepsilon\|^{h,\var}_{\dot{B}^{\frac{d}{2}}_{2,1}}+\|(\nabla a^\varepsilon,\nabla \theta^\varepsilon)\|^{h,\var}_{L^2_T(\dot{B}^{\frac{d}{2}-1}_{2,1})}+\|G^\varepsilon\|^{h,\var}_{L^2_T(\dot{B}^{\frac{d}{2}-1}_{2,1})}.
\end{align}
Multiplying \eqref{eq:regu2} by $1/2$, adding it to \eqref{eq:highlast2} and multiplying the resulting inequality by $\varepsilon$ concludes the proof of Proposition \ref{prop:linhf}.
\end{proof}

\section{Proof of Theorem \ref{thm:exist}: nonlinear analysis }\label{sec:nonlinear}
In this section, we estimate the nonlinear terms appearing on the right-hand side of \eqref{LinearAnalysisEst} in Proposition \ref{prop:Aprioriallfreq}.
\subsection{Low frequencies: nonlinear analysis} \label{nonlinearLf}
We prove the following lemma.
\begin{Lemme}\label{lem:NLlow} Let $(a^\varepsilon,v^\varepsilon,\theta^\varepsilon,q^\varepsilon)$ be a smooth solution of \eqref{LinearEffective}, we have
$$\|(F^\varepsilon,G^\varepsilon,H^\varepsilon,\var I^\varepsilon)\|^\ell_{L^1_T(\dot{B}^{\frac d2-1}_{2,1})}\leq X(t)^2.$$
\end{Lemme}
\begin{proof}
 First, we focus on the term $I^\varepsilon=v^\varepsilon\cdot\nabla q^\varepsilon-q^\varepsilon\cdot\nabla v^\varepsilon+ q^\varepsilon\div v^\varepsilon$. To this end, we shall employ the following two inequalities (see \cite{HeDanchin}):
\begin{equation}\label{R-E955b}
\|T_{f}g\|_{\dot{B}^{s-1+\frac d2-\frac dp}_{2,1}}\lesssim \|f\|_{\dot{B}^{\frac dp-1}_{p,1}}\|g\|_{\dot{B}^{s}_{p,1}}\quad \mbox{if}\ d\geq2 \ \mbox{and}\ \frac{d}{d-1}\leq p \leq\min (4, d^{*}),
\end{equation}
\begin{equation}\label{R-E955c}
\|R(f,g)\|_{\dot{B}^{s-1+\frac d2-\frac dp}_{2,1}}\lesssim \|f\|_{\dot{B}^{\frac dp-1}_{p,1}}\|g\|_{\dot{B}^{s}_{p,1}} \quad \mbox{if} \ s>1-\min \Big(\frac dp,\frac {d}{p'}\Big)\ \mbox{and}\  1\leq p \leq4,
\end{equation}
where $1/p+1/p'=1$ and $d^{*}\triangleq \frac{2d}{d-2}$. Using Bony's para-product decomposition, we have
\begin{equation}\label{R-E955e} v^\varepsilon\cdot\nabla q^\varepsilon=T_{\nabla q^\varepsilon}v^\varepsilon+R(\nabla q^\varepsilon,v^\varepsilon)+T_{v^\varepsilon}\nabla q^{\ell,\varepsilon}+T_{v^\varepsilon}\nabla q^{h}.\end{equation}
 Thanks to \eqref{R-E955b} and \eqref{R-E955c} with $s=\frac dp$, we get
\begin{align*}
&\|T_{\nabla q^\varepsilon}v^\varepsilon\|^{\ell}_{L^{1}_{T}(\dot{B}^{\frac d2-1}_{2,1})} \lesssim \|\nabla q^\varepsilon\|_{L^{1}_{T}(\dot{B}^{\frac dp-1}_{p,1})}\|v^\varepsilon\|_{L^{\infty}_{T}(\dot{B}^{\frac dp}_{p,1})},\\&
\|R(\nabla q^\varepsilon,v^\varepsilon)\|^{\ell}_{L^{1}_{T}(\dot{B}^{\frac d2-1}_{2,1})} \lesssim  \|\nabla q^\varepsilon\|_{L^{1}_{T}(\dot{B}^{\frac dp-1}_{p,1})}\|v^\varepsilon\|_{L^{\infty}_{T}(\dot{B}^{\frac dp}_{p,1})}.
\end{align*}
It follows from the definition of $X^{\varepsilon}$ and \eqref{eq:medBernstein}-\eqref{eq:lowhfBernstein} that
\begin{equation}\label{R-E1}
\|\nabla q^\varepsilon\|_{L^{1}_{T}(\dot{B}^{\frac dp-1}_{p,1})}\lesssim \|q^\varepsilon\|^{\ell}_{L^{1}_{T}(\dot{B}^{\frac d2}_{2,1})}+\|q^\varepsilon\|^{m}_{L^{1}_{T}(\dot{B}^{\frac dp}_{p,1})}+\|q^\varepsilon\|^{h}_{L^{1}_{T}(\dot{B}^{\frac d2}_{2,1})}
\lesssim X^{\varepsilon}
\end{equation}
and
\begin{equation}\label{R-E2}
\|v^\varepsilon\|_{L^{\infty}_{T}(\dot{B}^{\frac dp}_{p,1})}\lesssim K \|v^\varepsilon\|^{\ell}_{L^{\infty}_{T}(\dot{B}^{\frac d2-1}_{2,1})}+\|v^\varepsilon\|^{m}_{L^{\infty}_{T}(\dot{B}^{\frac dp}_{p,1})}+\|v^\varepsilon\|^{h}_{L^{\infty}_{T}(\dot{B}^{\frac d2}_{2,1})}\lesssim X^{\varepsilon}.
\end{equation}
Since $T$ maps $L^\infty\times \dot{B}^{\frac d2-1}_{2,1}$ to $\dot{B}^{\frac d2-1}_{2,1}$, 
\begin{equation}\label{R-E3}
\|T_{v^\varepsilon}\nabla q^{\ell,\varepsilon}\|^{\ell}_{L^{1}_{T}(\dot{B}^{\frac d2-1}_{2,1})}\lesssim \|v^\varepsilon\|_{L^{\infty}_{T}(L^{\infty})}\|\nabla q^{\ell,\varepsilon}\|_{L^{1}_{T}(\dot{B}^{\frac d2-1}_{2,1})}\lesssim  \|v^\varepsilon\|_{L^{\infty}_{T}(\dot{B}^{\frac dp}_{p,1})}\| q^{\varepsilon}\|^{\ell}_{L^{1}_{T}(\dot{B}^{\frac d2}_{2,1})}\lesssim X^{\varepsilon}X^{\ell,\varepsilon}.
\end{equation}
In order to handle the term $T_{v^\varepsilon}\nabla q^{h}$, we observe that owing to the spectral cut-off, there exists a universal integer $N_0$ such that
$$\Big(T_{v^\varepsilon}\nabla q^{h}\Big)^{\ell}=\dot{S}_{J_{0}+1}\Big(\sum_{|j-J_0|\leq N_{0}}\dot{S}_{j-1}v^{\varepsilon}\dot{\Delta}_{j}\nabla q^{h}\Big).$$
Hence $\|T_{v^\varepsilon}\nabla q^{h}\|^{\ell}_{\dot{B}^{\frac d2-1}_{2,1}}\approx 2^{J_0(\frac d2-1)}\sum_{|j-J_0|\leq N_{0}}\|\dot{S}_{j-1}v^{\varepsilon}\dot{\Delta}_{j}\nabla q^{h}\|_{L^2}$. If $2\leq p \leq \min (d,d^*)$ then one may use that, for $|j-J_0|\leq N_{0}$,
\begin{equation*}
\begin{aligned}
&2^{J_0(\frac d2-1)}\|\dot{S}_{j-1}v^{\varepsilon}\dot{\Delta}_{j}\nabla q^{h}\|_{L^2}\lesssim \|\dot{S}_{j-1}v^{\varepsilon}\|_{L^d}\Big(2^{j(\frac{d}{d^*}-1)}\|\dot{\Delta}_{j}\nabla q^{h}\|_{L^{d^{*}}}\Big)\\
&\qquad\qquad\qquad\qquad\qquad\qquad\lesssim \|v^{\varepsilon}\|_{\dot{B}^{0}_{d,1}}\|\nabla q^{h}\|_{\dot{B}^{\frac {d}{d*}-1}_{d*,\infty}}
\lesssim \|v^{\varepsilon}\|_{\dot{B}^{\frac dp-1}_{p,1}}\|q^{h}\|_{\dot{B}^{\frac dp}_{p,1}},
\end{aligned}
\end{equation*}
where we have used the embeddings $\dot{B}^{\frac dp-1}_{p,1}\hookrightarrow \dot{B}^{0}_{d,1}\hookrightarrow L^{d}$ and $\dot{B}^{\frac dp-1}_{p,\infty}\hookrightarrow\dot{B}^{\frac {d}{d*}-1}_{d*,\infty}$. If $d\leq p\leq 4$, then it holds that
\begin{equation}\nonumber
\begin{aligned}
2^{J_0(\frac d2-1)}\|\dot{S}_{j-1}v^{\varepsilon}\dot{\Delta}_{j}\nabla q^{h}\|_{L^2}&\lesssim\Big(2^{j\frac d4}\|\dot{S}_{j-1}v^{\varepsilon}\|_{L^4}\Big)\Big(2^{j(\frac d4-1)}\|\dot{\Delta}_{j}\nabla q^{h}\|_{L^4}\Big)\\ &\lesssim 2^{J_0}\Big(2^{j(\frac dp-1)}\|\dot{S}_{j-1}v^{\varepsilon}\|_{L^p}\Big)\Big(2^{j(\frac dp-1)}\|\nabla q^{h}\|_{L^p}\Big) \lesssim\|v^{\varepsilon}\|_{\dot{B}^{\frac dp-1}_{p,1}}\|q^{h}\|_{\dot{B}^{\frac dp}_{p,1}}.
\end{aligned}
\end{equation}
Hence, we deduce that
\begin{equation}\label{R-E412}
\begin{aligned}
\|T_{v^\varepsilon}\nabla q^{h}\|^{\ell}_{L^{1}_{T}(\dot{B}^{\frac d2-1}_{2,1})}&\lesssim \|v^{\varepsilon}\|_{L^{\infty}_{T}(\dot{B}^{\frac dp-1}_{p,1})}\|q^{h}\|_{L^{1}_{T}(\dot{B}^{\frac dp}_{p,1})}\\
&\lesssim (X^{\ell,\varepsilon} +X^{m,\varepsilon} + \varepsilon X^{h,\varepsilon})(X^{m,\varepsilon} + X^{h,\varepsilon})\lesssim (X^{\varepsilon})^2.
\end{aligned}
\end{equation}

Bounding the other nonlinear terms follows from a similar process, we give the sketch of computations since we need to track the uniformity of relaxation parameter $\varepsilon$. Let us take a look at the term $q^\varepsilon\cdot \nabla v^\varepsilon$. We have
\begin{equation}\label{R-E5}
\|T_{\nabla v^\varepsilon}q^{\varepsilon}+R(\nabla v^\varepsilon,q^{\varepsilon})\|^{\ell}_{L^{1}_{T}(\dot{B}^{\frac d2-1}_{2,1})}\lesssim \|\nabla v^\varepsilon\|_{L^{\infty}_{T}(\dot{B}^{\frac dp-1}_{p,1})}\|q^{\varepsilon}\|_{L^{1}_{T}(\dot{B}^{\frac dp}_{p,1})}\lesssim (X^{\varepsilon})^2,
\end{equation}
\begin{equation}\label{R-E6}
\|T_{q^{\varepsilon}}\nabla v^{\ell,\varepsilon}\|^{\ell}_{L^{1}_{T}(\dot{B}^{\frac d2-1}_{2,1})}\lesssim \|q^{\varepsilon}\|_{L^{1}_{T}(\dot{B}^{\frac dp}_{p,1})}\|v^{\ell,\varepsilon}\|_{L^{\infty}_{T}(\dot{B}^{\frac d2-1}_{2,1})}\lesssim X^{\varepsilon}X^{\ell,\varepsilon}
\end{equation}
and
\begin{equation}\label{R-E7}
\begin{aligned}
&\|T_{q^{\varepsilon}}\nabla v^{h}\|^{\ell}_{L^{1}_{T}(\dot{B}^{\frac d2-1}_{2,1})}\lesssim \|q^{\varepsilon}\|_{L^{\infty}_{T}(\dot{B}^{\frac dp-1}_{p,1})}\|v^{h}\|_{L^{1}_{T}(\dot{B}^{\frac dp}_{p,1})}\nonumber \\ &\lesssim (\frac{1}{\varepsilon}X^{\ell,\varepsilon} +\frac{1}{\varepsilon}X^{m,\varepsilon} + \frac{1}{\varepsilon} X^{h,\varepsilon})(X^{m,\varepsilon}+\varepsilon X^{h,\varepsilon})=\frac{1}{\varepsilon}X^{\varepsilon}(X^{m,\varepsilon}+\varepsilon X^{h,\varepsilon}).
\end{aligned}
\end{equation}
Similarly, we have
\begin{equation}\label{R-E780}
\|q^\varepsilon\div v^\varepsilon\|^{\ell}_{L^{1}_{T}(\dot{B}^{\frac d2-1}_{2,1})}\lesssim \frac{1}{\varepsilon}(X^{\varepsilon})^2.
\end{equation}
Next, we estimate the nonlinear terms $F^\varepsilon,G^\varepsilon,H^\varepsilon$. We write
$$v^{\varepsilon}\cdot\nabla a^{\varepsilon}=T_{\nabla a^{\varepsilon}}v^{\varepsilon}+R(v^{\varepsilon}\cdot\nabla a^{\varepsilon})+T_{v^{\varepsilon}}\nabla a^{\ell,\varepsilon}+T_{v^{\varepsilon}}\nabla a^{h,\varepsilon}.$$
We have
\begin{equation}\label{R-E9}
\|T_{\nabla a^{\varepsilon}}v^{\varepsilon}+R(v^{\varepsilon}\cdot\nabla a^{\varepsilon})\|^{\ell}_{L^{1}_{T}(\dot{B}^{\frac d2-1}_{2,1})}\lesssim \|\nabla a^\varepsilon\|_{L^{2}_{T}(\dot{B}^{\frac dp-1}_{p,1})}\|v^{\varepsilon}\|_{L^{2}_{T}(\dot{B}^{\frac dp}_{p,1})}\lesssim (X^{\varepsilon})^2,
\end{equation}
where we employed the interpolation and the definition of $X^{\varepsilon}$ to get
$$\|a^\varepsilon\|_{L^{2}_{T}(\dot{B}^{\frac dp}_{p,1})}\lesssim \|a^{\ell,\varepsilon}\|_{L^{\infty}_{T}(\dot{B}^{\frac dp-1}_{p,1})\cap L^{1}_{T}(\dot{B}^{\frac dp+1}_{p,1})}+\|a^{m,\varepsilon}\|_{L^{\infty}_{T}(\dot{B}^{\frac dp}_{p,1})\cap L^{1}_{T}(\dot{B}^{\frac dp}_{p,1})}+\|a^{h,\varepsilon}\|_{L^{\infty}_{T}(\dot{B}^{\frac dp}_{p,1})\cap L^{1}_{T}(\dot{B}^{\frac dp}_{p,1})}\lesssim X^{\varepsilon}$$
and
\begin{eqnarray*}
\|v^{\varepsilon}\|_{L^{2}_{T}(\dot{B}^{\frac dp}_{p,1})}&\lesssim& \|v^{\ell,\varepsilon}\|_{L^{\infty}_{T}(\dot{B}^{\frac dp-1}_{p,1})\cap L^{1}_{T}(\dot{B}^{\frac dp+1}_{p,1})}+\|v^{m,\varepsilon}\|_{L^{\infty}_{T}(\dot{B}^{\frac dp-1}_{p,1})\cap L^{1}_{T}(\dot{B}^{\frac dp+1}_{p,1})}+\|v^{h,\varepsilon}\|_{L^{\infty}_{T}(\dot{B}^{\frac dp}_{p,1})\cap L^{1}_{T}(\dot{B}^{\frac dp}_{p,1})}\\&\lesssim & X^{\ell,\varepsilon}+X^{m,\varepsilon}+X^{h,\varepsilon}+\varepsilon X^{h,\varepsilon}\lesssim X^{\varepsilon}.
\end{eqnarray*}
It follows from Sobolev embedding that
\begin{equation}\label{R-E10}
\|T_{v^{\varepsilon}}\nabla a^{\ell,\varepsilon}\|^{\ell}_{L^{1}_{T}(\dot{B}^{\frac d2-1}_{2,1})}\lesssim \|v^{\varepsilon}\|_{L^2(L^\infty)}\|\nabla a^{\ell,\varepsilon}\|_{L^{2}_{T}(\dot{B}^{\frac d2-1}_{2,1})}\lesssim \|v^{\varepsilon}\|_{L^{2}_{T}(\dot{B}^{\frac dp}_{p,1})}\|a^{\ell,\varepsilon}\|_{L^{2}_{T}(\dot{B}^{\frac d2}_{2,1})}\lesssim X^{\varepsilon}X^{\ell,\varepsilon}.
\end{equation}
Similarly,
\begin{align}  
\label{R-E11}
\|T_{v^{\varepsilon}}\nabla a^{h}\|^{\ell}_{L^{1}_{T}(\dot{B}^{\frac d2-1}_{2,1})}&\lesssim \|v^{\varepsilon}\|_{L^{\infty}_{T}(\dot{B}^{\frac dp-1}_{p,1})}\|a^{h}\|_{L^{1}_{T}(\dot{B}^{\frac dp}_{p,1})}\nonumber \\ &\lesssim (X^{\ell,\varepsilon} +X^{m,\varepsilon} + \varepsilon X^{h,\varepsilon})(X^{m,\varepsilon}+X^{h,\varepsilon})\lesssim X^{\varepsilon}(X^{m,\varepsilon}+ X^{h,\varepsilon}).
\end{align}
For $a^{\varepsilon}\div v^{\varepsilon}$, we obtain
\begin{align}\label{R-E12}
&\|a^{\varepsilon}\div v^{\varepsilon}\|^{\ell}_{L^{1}_{T}(\dot{B}^{\frac d2-1}_{2,1})}\\&\lesssim \|v^{\varepsilon}\|_{L^{2}_{T}(\dot{B}^{\frac dp}_{p,1})}\| a^\varepsilon\|_{L^{2}_{T}(\dot{B}^{\frac dp}_{p,1})}+\|a^\varepsilon\|_{L^{2}_{T}(\dot{B}^{\frac dp}_{p,1})} \|v^{\ell,\varepsilon}\|_{L^{2}_{T}(\dot{B}^{\frac d2-1}_{2,1})}+\|a^{\varepsilon}\|_{L^{\infty}_{T}(\dot{B}^{\frac dp-1}_{p,1})} \|v^{h}\|_{L^{1}_{T}(\dot{B}^{\frac dp}_{p,1})}\nonumber\\& \lesssim (X^{\varepsilon})^2+X^{\varepsilon}X^{\ell,\varepsilon}+(X^{\ell,\varepsilon} +X^{m,\varepsilon} + \varepsilon X^{h,\varepsilon})(X^{m,\varepsilon}+\varepsilon X^{h,\varepsilon}).
\end{align}
We now focus on $G^\varepsilon$. Regarding $v^{\varepsilon}\cdot\nabla v^{\varepsilon}$, we have
\begin{align}\label{R-E13}
\|v^{\varepsilon}\cdot\nabla v^{\varepsilon}\|^{\ell}_{L^{1}_{T}(\dot{B}^{\frac d2-1}_{2,1})}\nonumber&\lesssim\|v^{\varepsilon}\|^2_{L^{2}_{T}(\dot{B}^{\frac dp}_{p,1})}+\|v^{\varepsilon}\|_{L^{2}_{T}(\dot{B}^{\frac dp}_{p,1})}\|v^{\ell,\varepsilon}\|_{L^{2}_{T}(\dot{B}^{\frac d2-1}_{2,1})}+\|v^{\varepsilon}\|_{L^{\infty}_{T}(\dot{B}^{\frac dp-1}_{p,1})} \|v^{h}\|_{L^{1}_{T}(\dot{B}^{\frac dp}_{p,1})}\nonumber\\&\lesssim (X^{\varepsilon})^2+X^{\varepsilon}X^{\ell,\varepsilon}+(X^{\ell,\varepsilon} +X^{m,\varepsilon} + \varepsilon X^{h,\varepsilon})(X^{m,\varepsilon}+\varepsilon X^{h,\varepsilon}).
\end{align}
Using the composition estimate \ref{prop:comphf}, we obtain 
\begin{eqnarray}\label{R-E14}
\|J(a^\varepsilon)\mathcal{A}v^\varepsilon\|^{\ell}_{L^{1}_{T}(\dot{B}^{\frac d2-1}_{2,1})}&\lesssim&\|v^\varepsilon\|_{L^{1}_{T}(\dot{B}^{\frac dp+1}_{p,1})}
\|a^\varepsilon\|_{L^{\infty}_{T}(\dot{B}^{\frac dp}_{p,1})}+\|a^\varepsilon\|_{L^{\infty}_{T}(\dot{B}^{\frac dp}_{p,1})}\|v^{\ell,\varepsilon}\|_{L^{1}_{T}(\dot{B}^{\frac d2+1}_{2,1})}\nonumber\\&&+\|a^{\varepsilon}\|_{L^{\infty}_{T}(\dot{B}^{\frac dp-1}_{p,1})} \|v^{h}\|_{L^{1}_{T}(\dot{B}^{\frac dp+1}_{p,1})}\nonumber\\&\lesssim& (X^{\varepsilon})^2+X^{\varepsilon}X^{\ell,\varepsilon}+(X^{\ell,\varepsilon}+X^{m,\varepsilon}+\varepsilon X^{h,\varepsilon})(X^{m,\varepsilon}+X^{h,\varepsilon}),
\end{eqnarray}
\begin{eqnarray}\label{R-E15}
\|K_1(a^\varepsilon)\nabla a^\varepsilon\|^{\ell}_{L^{1}_{T}(\dot{B}^{\frac d2-1}_{2,1})}&\lesssim& \|a^{\varepsilon}\|^2_{L^{2}_{T}(\dot{B}^{\frac dp}_{p,1})}
+\|a^{\varepsilon}\|_{L^{2}_{T}(\dot{B}^{\frac dp}_{p,1})}\|a^{\ell,\varepsilon}\|_{L^{2}_{T}(\dot{B}^{\frac dp}_{p,1})}+\|a^{\varepsilon}\|_{L^{\infty}_{T}(\dot{B}^{\frac dp-1}_{p,1})}\|a^{h}\|_{L^{1}_{T}(\dot{B}^{\frac dp}_{p,1})}\nonumber\\&\lesssim&(X^{\varepsilon})^2+X^{\varepsilon}X^{\ell,\varepsilon}+(X^{\ell,\varepsilon}+X^{m,\varepsilon}+\varepsilon X^{h,\varepsilon})(X^{m,\varepsilon}+X^{h,\varepsilon})
\end{eqnarray} and
\begin{eqnarray}\label{R-E16}
\|K_2(a^\varepsilon)\nabla \theta^\varepsilon\|^{\ell}_{L^{1}_{T}(\dot{B}^{\frac d2-1}_{2,1})}&\lesssim& \|\theta^{\varepsilon}\|_{L^{2}_{T}(\dot{B}^{\frac dp}_{p,1})}\|a^{\varepsilon}\|_{L^{2}_{T}(\dot{B}^{\frac dp}_{p,1})}
\nonumber\\&&+\|a^{\varepsilon}\|_{L^{2}_{T}(\dot{B}^{\frac dp}_{p,1})}\|\theta^{\ell,\varepsilon}\|_{L^{2}_{T}(\dot{B}^{\frac d2}_{2,1})}+\|a^{\varepsilon}\|_{L^{\infty}_{T}(\dot{B}^{\frac dp-1}_{p,1})}\|\theta^{h}\|_{L^{1}_{T}(\dot{B}^{\frac dp}_{p,1})}.
\end{eqnarray}
Using an interpolation inequality, we have
\begin{eqnarray*}\|\theta^{\varepsilon}\|_{L^{2}_{T}(\dot{B}^{\frac dp}_{p,1})}&\lesssim& \|\theta^{\ell,\varepsilon}\|_{L^{\infty}_{T}(\dot{B}^{\frac dp-1}_{p,1})\cap L^{1}_{T}(\dot{B}^{\frac dp+1}_{p,1})}+\|\theta^{m,\varepsilon}\|_{L^{\infty}_{T}(\dot{B}^{\frac dp-1}_{p,1})\cap L^{1}_{T}(\dot{B}^{\frac dp+1}_{p,1})}+\|\theta^{h,\varepsilon}\|_{L^{\infty}_{T}(\dot{B}^{\frac dp}_{p,1})\cap L^{1}_{T}(\dot{B}^{\frac dp}_{p,1})}\nonumber \\ &\lesssim & X^{\ell,\varepsilon}+X^{m,\varepsilon}+\varepsilon\|\theta^{h,\varepsilon}\|_{L^{\infty}_{T}(\dot{B}^{\frac d2}_{2,1})}+\frac{1}{\varepsilon}\|\theta^{h,\varepsilon}\|_{L^{1}_{T}(\dot{B}^{\frac d2}_{2,1})}\lesssim X^{\varepsilon},
\end{eqnarray*}
which leads to
\begin{align}\label{R-E17}
&\|K_2(a^\varepsilon)\nabla \theta^\varepsilon\|^{\ell}_{L^{1}_{T}(\dot{B}^{\frac d2-1}_{2,1})}\lesssim ( X^{\varepsilon})^2+X^{\varepsilon}X^{\ell,\varepsilon}+(X^{\ell,\varepsilon}+X^{m,\varepsilon}+\varepsilon X^{h,\varepsilon})(X^{m,\varepsilon}+\varepsilon X^{h,\varepsilon}),
\\&\label{R-E18}
\|T_{\nabla K_3(a^{\varepsilon})}\theta^{\varepsilon}+R(\nabla K_3(a^{\varepsilon}),\theta^{\varepsilon})\|^{\ell}_{L^{1}_{T}(\dot{B}^{\frac d2-1}_{2,1})}\lesssim \|\nabla K_3(a^\varepsilon)\|_{L^{2}_{T}(\dot{B}^{\frac dp-1}_{p,1})}\|\theta^{\varepsilon}\|_{L^{2}_{T}(\dot{B}^{\frac dp}_{p,1})}\lesssim (X^{\varepsilon})^2,
\\&\label{R-E19}
\|T_{\theta^{\varepsilon}}\nabla K_3(a^{\varepsilon})^{\ell}\|^{\ell}_{L^{1}_{T}(\dot{B}^{\frac d2-1}_{2,1})}\lesssim \|\theta^{\varepsilon}\|_{L^{2}_{T}(\dot{B}^{\frac dp}_{p,1})}\|\nabla K_3(a^{\varepsilon})^{\ell}\|_{L^{2}_{T}(\dot{B}^{\frac d2-1}_{2,1})}.
\end{align}
Using that $\nabla K_3(a^{\varepsilon})=K_{3}'(0)\nabla a^{\varepsilon}+\widetilde{K}_{3}(a^{\varepsilon})\nabla a^{\varepsilon}$ for some smooth function $\widetilde{K}_{3}$ vanishing at zero, we have
\begin{eqnarray}\label{R-E21}
\|\widetilde{K}_{3}(a^{\varepsilon})\nabla a^{\varepsilon}\|^{\ell}_{L^{2}_{T}(\dot{B}^{\frac d2-1}_{2,1})}&\lesssim&\|a^{\varepsilon}\|_{L^{2}_{T}(\dot{B}^{\frac dp}_{p,1})}\|a^{\varepsilon}\|_{L^{\infty}_{T}(\dot{B}^{\frac dp}_{p,1})}
\nonumber\\&&+\|a^{\varepsilon}\|_{L^{\infty}_{T}(\dot{B}^{\frac dp}_{p,1})}\|a^{\ell,\varepsilon}\|_{L^{2}_{T}(\dot{B}^{\frac dp}_{p,1})}+\|a^{\varepsilon}\|_{L^{\infty}_{T}(\dot{B}^{\frac dp-1}_{p,1})}\|a^{h}\|_{L^{2}_{T}(\dot{B}^{\frac dp}_{p,1})}\nonumber\\&\lesssim&(X^{\varepsilon})^2+X^{\varepsilon}X^{\ell,\varepsilon}+(X^{\ell,\varepsilon}+X^{m,\varepsilon}+\varepsilon X^{h,\varepsilon})(X^{m,\varepsilon}+X^{h,\varepsilon}).
\end{eqnarray}
The final term $T_{\theta^{\varepsilon}}\nabla K_3(a^{\varepsilon})^{\widetilde{h}}$ can be similarly estimated  as follows
\begin{eqnarray}\label{R-E22}
\|T_{\theta^{\varepsilon}}\nabla K_3(a^{\varepsilon})^{\widetilde{h}}\|^{\ell}_{L^{1}_{T}(\dot{B}^{\frac d2-1}_{2,1})}&\lesssim &
\|\theta^{\varepsilon}\|_{L^{\infty}_{T}(\dot{B}^{\frac dp-1}_{p,1})} \|\nabla K_3(a^{\varepsilon})^{\widetilde{h}}\|_{L^{1}_{T}(\dot{B}^{\frac dp-1}_{p,1})}
\nonumber\\&\lesssim & X^{\varepsilon}(X^{m,\varepsilon}+X^{h,\varepsilon}+(X^{\varepsilon})^2).
\end{eqnarray}
For the nonlinear term $v^{\varepsilon}\cdot\nabla \theta^{\varepsilon}$, we employ Bony's para-product decomposition:
$v^{\varepsilon}\cdot\nabla \theta^{\varepsilon}\triangleq T_{\nabla \theta^{\varepsilon}}v^{\varepsilon}+R(v^{\varepsilon}\cdot\nabla \theta^{\varepsilon})+T_{v^{\varepsilon}}\nabla \theta^{\ell,\varepsilon}+T_{v^{\varepsilon}}\nabla \theta^{h,\varepsilon}$. We have
$$\|T_{\nabla \theta^{\varepsilon}}v^{\varepsilon}+R(v^{\varepsilon}\cdot\nabla \theta^{\varepsilon})\|^{\ell}_{L^{1}_{T}(\dot{B}^{\frac d2-1}_{2,1})}
\lesssim \|\theta^{\varepsilon}\|_{L^{2}_{T}(\dot{B}^{\frac dp}_{p,1})}\|v^{\varepsilon}\|_{L^{2}_{T}(\dot{B}^{\frac dp}_{p,1})}\lesssim (X^{\varepsilon})^2,
$$
$$\|T_{v^{\varepsilon}}\nabla \theta^{\ell,\varepsilon}\|^{\ell}_{L^{1}_{T}(\dot{B}^{\frac d2-1}_{2,1})}\lesssim \|v^{\varepsilon}\|_{L^{2}_{T}(L^\infty)}\|\theta^{\ell,\varepsilon}\|_{L^{2}_{T}(\dot{B}^{\frac d2-1}_{2,1})}\lesssim X^{\varepsilon}X^{\ell,\varepsilon},$$
and
$$
\|T_{v^{\varepsilon}}\nabla \theta^{h,\varepsilon}\|^{\ell}_{L^{1}_{T}(\dot{B}^{\frac d2-1}_{2,1})}\lesssim \|v^{\varepsilon}\|_{L^{\infty}_{T}(\dot{B}^{\frac dp-1}_{p,1})}\|\theta^{h,\varepsilon}\|_{L^{1}_{T}(\dot{B}^{\frac dp}_{p,1})}\lesssim (X^{\ell,\varepsilon} +X^{m,\varepsilon} + \varepsilon X^{h,\varepsilon})(X^{m,\varepsilon}+\varepsilon X^{h,\varepsilon}).
$$
For the term $I(a^{\varepsilon})\Delta \theta^{\varepsilon}\triangleq I(a^{\varepsilon})\Delta \theta^{\ell, \varepsilon}+I(a^{\varepsilon})\Delta \theta^{\widetilde{h}, \varepsilon}$, we obtain
$$\|I(a^{\varepsilon})\Delta \theta^{\ell, \varepsilon}\|^{\ell}_{L^{1}_{T}(\dot{B}^{\frac d2-1}_{2,1})}\lesssim \|I(a^{\varepsilon})\|_{L^{\infty}_{T}(L^\infty)}\|\Delta \theta^{\ell, \varepsilon}\|_{L^{1}_{T}(\dot{B}^{\frac d2-1}_{2,1})}\lesssim\|a^{\varepsilon}\|_{L^{\infty}_{T}(\dot{B}^{\frac dp}_{p,1})}\|\theta^{\ell, \varepsilon}\|_{L^{1}_{T}(\dot{B}^{\frac d2+1}_{2,1})}\lesssim X^{\varepsilon}X^{\ell,\varepsilon}.$$
For the second term, we write
$I(a^{\varepsilon})\Delta \theta^{\widetilde{h}, \varepsilon}=T_{I(a^{\varepsilon})}\Delta \theta^{\widetilde{h}}+R(I(a^{\varepsilon}), \Delta \theta^{\widetilde{h}})$ and
we use that $R$ and $T$ map $$\dot{B}^{\frac dp-2}_{p,1}\times\dot{B}^{\frac dp}_{p,1}\rightarrow \dot{B}^{\frac d2-2}_{2,1}, $$ for $p<d$ and $d\geq3$. This leads to
\begin{eqnarray}\label{R-E23}
\|T_{I(a^{\varepsilon})}\Delta \theta^{\widetilde{h}}+R(I(a^{\varepsilon}), \Delta \theta^{\widetilde{h}})\|^{\ell}_{L^{1}_{T}(\dot{B}^{\frac d2-1}_{2,1})}&\lesssim & K\|T_{I(a^{\varepsilon})}\Delta \theta^{\widetilde{h}}+R(I(a^{\varepsilon}), \Delta \theta^{\widetilde{h}})\|\nonumber\\ &\lesssim &\|I(a^{\varepsilon})\|_{L^{\infty}_{T}(\dot{B}^{\frac dp}_{p,1})}
\|\Delta \theta^{\widetilde{h}}\|_{L^{1}_{T}(\dot{B}^{\frac dp-2}_{p,1})}\\&\lesssim& \|a^{\varepsilon}\|_{L^{\infty}_{T}(\dot{B}^{\frac dp}_{p,1})}\|\theta^{\widetilde{h}}\|_{L^{1}_{T}(\dot{B}^{\frac dp}_{p,1})}\nonumber\\ &\lesssim& X^{\varepsilon}(X^{m,\varepsilon}+\varepsilon X^{h,\varepsilon}).
\end{eqnarray}
Next, we focus on $\dfrac{N(\nabla v^\varepsilon,\nabla v^\varepsilon)}{1+a^\varepsilon}\triangleq (1+J(a))N(\nabla v^\varepsilon,\nabla v^\varepsilon)$. Using the continuity of the para-product and remainder operators, we obtain
\begin{align}\label{R-E24}
&\|T_{J(a)}\nabla u\otimes \nabla u\|_{L^{1}_{T}(\dot{B}^{\frac d2-2}_{2,1})}\lesssim \|J(a)\|_{L^{\infty}_{T}(L^\infty)}\|\nabla v^{\varepsilon}\otimes \nabla v^{\varepsilon}\|_{L^{1}_{T}(\dot{B}^{\frac d2-2}_{2,1})},
\\&\label{R-E25}
\|R(J(a),\nabla v^{\varepsilon}\otimes \nabla v^{\varepsilon})\|_{L^{1}_{T}(\dot{B}^{\frac d2-2}_{2,1})}\lesssim \|J(a)\|_{L^{\infty}_{T}(L^\infty)}\|\nabla v^{\varepsilon}\otimes \nabla v^{\varepsilon}\|_{L^{1}_{T}(\dot{B}^{\frac d2-2}_{2,1})}
\end{align}
and
\begin{align}\label{R-E26}
\|T_{\nabla v^{\varepsilon}\otimes \nabla v^{\varepsilon}}J(a)\|_{L^{1}_{T}(\dot{B}^{\frac d2-2}_{2,1})}&\lesssim \|\nabla v^{\varepsilon}\otimes \nabla v^{\varepsilon}\|_{L^{1}_{T}(\dot{B}^{\frac {d}{p^{*}}-2}_{p^{*},1})}\|J(a)\|_{L^{\infty}_{T}(\dot{B}^{\frac dp}_{p,1})}\nonumber\\&\lesssim\|\nabla v^{\varepsilon}\otimes \nabla v^{\varepsilon}\|_{L^{1}_{T}(\dot{B}^{\frac d2-2}_{2,1})}\|J(a)\|_{L^{\infty}_{T}(\dot{B}^{\frac dp}_{p,1})}.
\end{align}
It follows from the mapping
$$\dot{B}^{\frac dp-1}_{p,1}\times\dot{B}^{\frac dp-1}_{p,1}\rightarrow \dot{B}^{\frac d2-2}_{2,1},$$
for $2\leq p\leq 2d/d-2$, $p<d$ and $d\geq3$, where $1/p+1/p^{*}=1/2$, that
\begin{align}\label{R-E27}
\|\dfrac{N(\nabla v^\varepsilon,\nabla v^\varepsilon)}{1+a^\varepsilon}\|^{\ell}_{L^{1}_{T}(\dot{B}^{\frac d2-1}_{2,1})}&\lesssim K\|\dfrac{N(\nabla v^\varepsilon,\nabla v^\varepsilon)}{1+a^\varepsilon}\|^{\ell}_{L^{1}_{T}(\dot{B}^{\frac d2-2}_{2,1})}\nonumber\\&\lesssim (1+\|a\|_{L^{\infty}_{T}(\dot{B}^{\frac dp}_{p,1})})\|\nabla v^{\varepsilon}\|^2_{L^{2}_{T}(\dot{B}^{\frac dp-1}_{p,1})}\lesssim (1+X^{\varepsilon})(X^{\varepsilon})^2.
\end{align}
Now we bound $H_1(a^\varepsilon)\theta^\varepsilon \div v^\varepsilon$. We write $H_1(a^\varepsilon)\theta^\varepsilon \div v^\varepsilon=T_{H_1(a^\varepsilon)\div v^\varepsilon}\theta^\varepsilon+R(H_1(a^\varepsilon)\div v^\varepsilon,\theta^\varepsilon)+T_{\theta^\varepsilon}(H_1(a^\varepsilon)\div v^\varepsilon)^{\ell}+T_{\theta^\varepsilon}(H_1(a^\varepsilon)\div v^\varepsilon)^{\widetilde{h}}$. Similarly, one gets
\begin{eqnarray}\label{R-E28}
\|T_{H_1(a^\varepsilon)\div v^\varepsilon}\theta^\varepsilon+R(H_1(a^\varepsilon)\div v^\varepsilon,\theta^\varepsilon)\|^{\ell}_{L^{1}_{T}(\dot{B}^{\frac d2-1}_{2,1})}&\lesssim &\|H_1(a^\varepsilon)\div v^\varepsilon\|_{L^{2}_{T}(\dot{B}^{\frac dp-1}_{p,1})}\|\theta^\varepsilon\|_{L^{2}_{T}(\dot{B}^{\frac dp}_{p,1})}\nonumber\\&\lesssim& \|a^\varepsilon\|_{L^{\infty}_{T}(\dot{B}^{\frac dp}_{p,1})}\|v^{\varepsilon}\|_{L^{2}_{T}(\dot{B}^{\frac dp}_{p,1})}
\|\theta^{\varepsilon}\|_{L^{2}_{T}(\dot{B}^{\frac dp}_{p,1})}\lesssim (X^{\varepsilon})^3.
\end{eqnarray}
By applying a similar procedure that led to \eqref{R-E22}, we obtain
\begin{eqnarray}\label{R-E29}
\|H_1(a^{\varepsilon})\div v^\varepsilon\|^{\ell}_{L^{2}_{T}(\dot{B}^{\frac d2-1}_{2,1})}&\lesssim&\|\div v^\varepsilon\|_{L^{2}_{T}(\dot{B}^{\frac dp-1}_{p,1})}\|a^{\varepsilon}\|_{L^{\infty}_{T}(\dot{B}^{\frac dp}_{p,1})}
\nonumber\\&&+\|a^{\varepsilon}\|_{L^{\infty}_{T}(\dot{B}^{\frac dp}_{p,1})}\|\div v^{\ell,\varepsilon}\|_{L^{2}_{T}(\dot{B}^{\frac d2-1}_{2,1})}+\|a^{\varepsilon}\|_{L^{\infty}_{T}(\dot{B}^{\frac dp-1}_{p,1})}\|\div v^{h}\|_{L^{2}_{T}(\dot{B}^{\frac dp-1}_{p,1})}\nonumber\\&\lesssim&(X^{\varepsilon})^2+X^{\varepsilon}X^{\ell,\varepsilon}+(X^{\ell,\varepsilon}+X^{m,\varepsilon}+\varepsilon X^{h,\varepsilon})(X^{m,\varepsilon}+X^{h,\varepsilon}),
\end{eqnarray}
which yields
\begin{eqnarray}\label{R-E30}
\|T_{\theta^\varepsilon}(H_1(a^\varepsilon)\div v^\varepsilon)^{\ell}\|^{\ell}_{L^{1}_{T}(\dot{B}^{\frac d2-1}_{2,1})}\lesssim \|\theta^\varepsilon\|_{L^{2}_{T}(\dot{B}^{\frac dp}_{p,1})}\|H_1(a^{\varepsilon})\div v^\varepsilon\|^{\ell}_{L^{2}_{T}(\dot{B}^{\frac d2-1}_{2,1})}\lesssim (X^{\varepsilon})^3.
\end{eqnarray}
It follows that
\begin{eqnarray}\label{R-E300}
\|T_{\theta^\varepsilon}(H_1(a^\varepsilon)\div v^\varepsilon)^{\widetilde{h}}\|^{\ell}_{L^{1}_{T}(\dot{B}^{\frac d2-1}_{2,1})}&\lesssim& \|\theta^\varepsilon\|_{L^{\infty}_{T}(\dot{B}^{\frac dp-1}_{p,1})}\|(H_1(a^\varepsilon)\div v^\varepsilon)^{\widetilde{h}}\|_{L^{1}_{T}(\dot{B}^{\frac dp-1}_{p,1})}
\nonumber\\&\lesssim &\|\theta^\varepsilon\|_{L^{\infty}_{T}(\dot{B}^{\frac dp-1}_{p,1})}\|a^\varepsilon\|_{L^{\infty}_{T}(\dot{B}^{\frac dp}_{p,1})}\|v^\varepsilon\|_{L^{1}_{T}(\dot{B}^{\frac dp+1}_{p,1})}\lesssim (X^{\varepsilon})^3,
\end{eqnarray}
which concludes the proof of Lemma \ref{lem:NLlow}.
     \end{proof}

\subsection{Medium frequencies: nonlinear analysis}
In this section, we show the following lemma.
\begin{Lemme}\label{lem:NLmed} Let $(a^\varepsilon,v^\varepsilon,\theta^\varepsilon,q^\varepsilon)$ be a smooth solution of \eqref{LinearEffective}, we have
$$\|F_1^\varepsilon\|^{m,\var}_{L^1_T(\dot{B}^{\frac{d}{p}}_{p,1})}+\|G^\varepsilon\|^{m,\var}_{L^1_T\cap L^2_T(\dot{B}^{\frac{d}{p}-1}_{p,1})}+\varepsilon\|I^\varepsilon\|^{m,\var}_{L^1_T( \dot{B}^{\frac{d}{p}-1}_{p,1})}
     + \| H^\varepsilon\|^{m,\var}_{L^1_T( \dot{B}^{\frac{d}{p}-1}_{p,1})}\leq X^\varepsilon(t)^2.$$
     \end{Lemme}
     \begin{proof}
First, we handle with the nonlinear terms in $I^{\varepsilon}$. It follows from standard product law that
\begin{eqnarray}\label{R-E31}
\|v^{\varepsilon}\cdot \nabla q^{\varepsilon}\|^{m}_{L^1_T( \dot{B}^{\frac{d}{p}-1}_{p,1})}\lesssim \|v^{\varepsilon}\|_{L^{\infty}_T( \dot{B}^{\frac{d}{p}-1}_{p,1})}\|\nabla q^{\varepsilon}\|_{L^1_T( \dot{B}^{\frac{d}{p}-1}_{p,1})}\lesssim X^\varepsilon(t)^2
\end{eqnarray}
and
\begin{eqnarray}\label{R-E32}
\|q^{\varepsilon}\cdot\nabla v^{\varepsilon}\|^{m}_{L^1_T( \dot{B}^{\frac{d}{p}-1}_{p,1})}\lesssim \|q^{\varepsilon}\|_{L^1_T( \dot{B}^{\frac{d}{p}}_{p,1})}\|\nabla v^{\varepsilon}\|_{L^{\infty}_T( \dot{B}^{\frac{d}{p}-1}_{p,1})}\lesssim X^\varepsilon(t)^2.
\end{eqnarray}
Similarly,
\begin{eqnarray}\label{R-E33}
\|q^{\varepsilon}\div v^{\varepsilon}\|^{m}_{L^1_T( \dot{B}^{\frac{d}{p}-1}_{p,1})}\lesssim X^\varepsilon(t)^2.
 \end{eqnarray}
Regarding $F_{1}^{\varepsilon}=a^{\varepsilon}\div v^{\varepsilon}$, it is easy to see that
\begin{eqnarray}\label{R-E34}
\|a^{\varepsilon}\div v^{\varepsilon}\|^{m}_{L^1_T( \dot{B}^{\frac{d}{p}}_{p,1})} \lesssim \|a^{\varepsilon}\|_{L^{\infty}_T( \dot{B}^{\frac{d}{p}}_{p,1})}\|\div v^{\varepsilon}\|_{L^1_T( \dot{B}^{\frac{d}{p}}_{p,1})}\lesssim X^\varepsilon(t)^2.
 \end{eqnarray}
Secondly, we bound the terms of $H^{\varepsilon}$ in turn. We have
    \begin{eqnarray}\label{R-E35}
\|v^{\varepsilon}\cdot\nabla \theta^{\varepsilon}\|^{m}_{L^1_T( \dot{B}^{\frac{d}{p}-1}_{p,1})} \lesssim \|v^{\varepsilon}\|_{L^2_T( \dot{B}^{\frac{d}{p}}_{p,1})}\|\nabla \theta^{\varepsilon}\|_{L^2_T( \dot{B}^{\frac{d}{p}-1}_{p,1})}\lesssim X^\varepsilon(t)^2.
 \end{eqnarray}
It follows from standard product laws and the composition Proposition \ref{prop:comphf} that
    \begin{eqnarray}\label{R-E36}
\|J(a^{\varepsilon})\div q^{\varepsilon}\|^{m}_{L^1_T( \dot{B}^{\frac{d}{p}-1}_{p,1})}\lesssim \|a^{\varepsilon}\|_{L^{\infty}_T( \dot{B}^{\frac{d}{p}}_{p,1})}\|q^{\varepsilon}\|_{L^1_T(\dot{B}^{\frac{d}{p}}_{p,1})}\lesssim X^\varepsilon(t)^2.
 \end{eqnarray}
 Similarly, we have
    \begin{align}\label{R-E37}
\|\dfrac{N(\nabla v^\varepsilon,\nabla v^\varepsilon)}{1+a^\varepsilon}\|^{m}_{L^1_T( \dot{B}^{\frac{d}{p}-1}_{p,1})} &\lesssim  (1+\|a^{\varepsilon}\|_{L^{\infty}_T( \dot{B}^{\frac{d}{p}}_{p,1})})\|\nabla v^\varepsilon\|_{L^{\infty}_T( \dot{B}^{\frac{d}{p}-1}_{p,1})}
\|\nabla v^\varepsilon\|_{L^{1}_T( \dot{B}^{\frac{d}{p}}_{p,1})}
\\&\lesssim (1+ X^\varepsilon(t))X^\varepsilon(t)^2.\nonumber
 \end{align}
and 
\begin{eqnarray}\label{R-E38}
\|\wt H_1(a^\varepsilon)\theta^\varepsilon \div v^\varepsilon\|^{m}_{L^1_T( \dot{B}^{\frac{d}{p}-1}_{p,1})}\lesssim \|a^{\varepsilon}\|_{L^{\infty}_T( \dot{B}^{\frac{d}{p}}_{p,1})}\|\theta^\varepsilon\|_{L^2_T( \dot{B}^{\frac{d}{p}}_{p,1})}\|\div v^\varepsilon\|_{L^2_T( \dot{B}^{\frac{d}{p}-1}_{p,1})}
\lesssim X^\varepsilon(t)^3.
 \end{eqnarray}
Finally, w estimate the nonlinear terms in $G^{\varepsilon}$. Precisely, we have
\begin{align}\label{R-E39}
&\|v^\varepsilon\cdot\nabla v^\varepsilon\|^{m}_{L^1_T( \dot{B}^{\frac{d}{p}-1}_{p,1})}\lesssim \|v^\varepsilon\|^2_{L^2_T( \dot{B}^{\frac{d}{p}}_{p,1})}\lesssim X^\varepsilon(t)^2,
\\&\label{R-E40}
\|J(a^{\varepsilon})\mathcal{A}v^\varepsilon\|^{m}_{L^1_T( \dot{B}^{\frac{d}{p}-1}_{p,1})}\lesssim \|a^{\varepsilon}\|_{L^{\infty}_T( \dot{B}^{\frac{d}{p}}_{p,1})}\|v^\varepsilon\|_{L^1_T( \dot{B}^{\frac{d}{p}+1}_{p,1})}\lesssim X^\varepsilon(t)^2,
\\&\label{R-E41}
\|K_1(a^{\varepsilon})\nabla a^{\varepsilon}\|^{m}_{L^1_T( \dot{B}^{\frac{d}{p}-1}_{p,1})}\lesssim\|a^{\varepsilon}\|^2_{L^2_T( \dot{B}^{\frac{d}{p}}_{p,1})}\lesssim X^\varepsilon(t)^2,
\\&\label{R-E42}
\|K_2(a^{\varepsilon})\nabla \theta^{\varepsilon}\|^{m}_{L^1_T( \dot{B}^{\frac{d}{p}-1}_{p,1})}\lesssim\|a^{\varepsilon}\|_{L^2_T( \dot{B}^{\frac{d}{p}}_{p,1})}\|\theta^{\varepsilon}\|_{L^2_T( \dot{B}^{\frac{d}{p}}_{p,1})}\lesssim X^\varepsilon(t)^2,
\\&\label{R-E43}
\|\theta^{\varepsilon}\nabla K_3(a^{\varepsilon})\|^{m}_{L^1_T( \dot{B}^{\frac{d}{p}-1}_{p,1})}\lesssim\|\theta^{\varepsilon}\|_{L^2_T( \dot{B}^{\frac{d}{p}}_{p,1})}\|a^{\varepsilon}\|_{L^2_T( \dot{B}^{\frac{d}{p}}_{p,1})}\lesssim X^\varepsilon(t)^2.
\end{align}
In a similar way, one can get the corresponding estimates in the norm $L^2_T( \dot{B}^{\frac{d}{p}-1}_{p,1})$. We obtain
\begin{align}\label{R-E44}
&\|v^\varepsilon\cdot\nabla v^\varepsilon\|^{m}_{L^2_T( \dot{B}^{\frac{d}{p}-1}_{p,1})}\lesssim \|v^\varepsilon\|_{L^{\infty}_T( \dot{B}^{\frac{d}{p}}_{p,1})}\|v^\varepsilon\|_{L^{2}_T( \dot{B}^{\frac{d}{p}}_{p,1})}\lesssim X^\varepsilon(t)^2,
\\&\label{R-E45}\|J(a^{\varepsilon})\mathcal{A}v^\varepsilon\|^{m}_{L^2_T( \dot{B}^{\frac{d}{p}-1}_{p,1})}\lesssim \|a^{\varepsilon}\|_{L^{\infty}_T( \dot{B}^{\frac{d}{p}}_{p,1})}\|v^\varepsilon\|_{L^2_T( \dot{B}^{\frac{d}{p}+1}_{p,1})}
\lesssim X^\varepsilon(t)^2,
\\&\label{R-E46}
\|K_1(a^{\varepsilon})\nabla a^{\varepsilon}\|^{m}_{L^2_T( \dot{B}^{\frac{d}{p}-1}_{p,1})}\lesssim\|a^{\varepsilon}\|_{L^{\infty}_T( \dot{B}^{\frac{d}{p}}_{p,1})}\|a^{\varepsilon}\|_{L^2_T( \dot{B}^{\frac{d}{p}}_{p,1})}\lesssim X^\varepsilon(t)^2,
\\&\label{R-E47}
\|K_2(a^{\varepsilon})\nabla \theta^{\varepsilon}\|^{m}_{L^2_T( \dot{B}^{\frac{d}{p}-1}_{p,1})}\lesssim\|a^{\varepsilon}\|_{L^{\infty}_T( \dot{B}^{\frac{d}{p}}_{p,1})}\|\theta^{\varepsilon}\|_{L^2_T( \dot{B}^{\frac{d}{p}}_{p,1})}\lesssim X^\varepsilon(t)^2,
\\&\label{R-E48}
\|\theta^{\varepsilon}\nabla K_3(a^{\varepsilon})\|^{m}_{L^2_T( \dot{B}^{\frac{d}{p}-1}_{p,1})}\lesssim\|\theta^{\varepsilon}\|_{L^2_T( \dot{B}^{\frac{d}{p}}_{p,1})}\|a^{\varepsilon}\|_{L^{\infty}_T( \dot{B}^{\frac{d}{p}}_{p,1})}\lesssim X^\varepsilon(t)^2.
\end{align}
The proof of Lemma \ref{lem:NLmed} is concluded.
\end{proof}

\subsection{High frequencies: nonlinear analysis}
We show the following lemma.
\begin{Lemme}\label{lem5.3}
Let $(a^\varepsilon,v^\varepsilon,\theta^\varepsilon,q^\varepsilon)$ be a smooth solution of \eqref{LinearEffective}, we have
\begin{align}
&\varepsilon\|(G^\varepsilon,H^\varepsilon)\|^{h,\var}_{L^1_T(\dot{B}^{\frac{d}{2}}_{2,1})}+\varepsilon\|G^\varepsilon\|^{h,\var}_{L^2_T(\dot{B}^{\frac{d}{2}}_{2,1})}+\|(\varepsilon F_1^\varepsilon,\varepsilon^2H_1^\varepsilon, \varepsilon^3I^\varepsilon,\varepsilon^2 R_1^\varepsilon)\|^{h,\var}_{L^1_T(\dot{B}^{\frac{d}{2}+1}_{2,1})} \\&\quad +\varepsilon\int_0^T\|\d_tJ(a^\varepsilon)\|_{\dot{B}^{\frac{d}{2}}_{2,1}} \|\var^2 q^\varepsilon\|^{h,\var}_{\dot{B}^{\frac{d}{2}+1}_{2,1}} \lesssim X^\varepsilon(t)^2.\nonumber
\end{align}
\end{Lemme}
\begin{proof}
Using commutator estimates from \cite{CBD3}, i.e. Lemma \ref{CP}, we obtain
\begin{align}\label{commutator12}
    \varepsilon^2\|R_1^\varepsilon\|^{h,\var}_{L^1_T(\dot{B}^{\frac{d}{2}+1}_{2,1})}& \lesssim \varepsilon^2\norme{\nabla J(a^\varepsilon)}_{L^\infty_T( B^{\frac dp}_{p,1})}\norme{q^\varepsilon}^{h,\varepsilon}_{L^1_T({B}^{\frac{d}{2}+1}_{2,1})}
+ \varepsilon^2 \norme{q^\varepsilon}^{\ell,\varepsilon}_{L^1_T( B^{\frac{d}{p}}_{p,1})}\norme{\nabla J(a^\varepsilon)}^{\ell,\varepsilon}_{L^\infty_T(\dot{B}^{\frac{d}{2}+1}_{p,1})}\\&+\varepsilon^2\norme{q^\varepsilon}_{L^1_T(B^{\frac dp}_{p,1})}\norme{\nabla J(a^\varepsilon)}^h_{L^\infty_T( B^{\frac{d}{2}+1}_{2,1})}
+  \varepsilon^2\norme{q^\varepsilon}^{\ell,\varepsilon}_{ L^1_T(B^{\frac dp}_{p,1})}\norme{\nabla J(a^\varepsilon)}^{\ell,\varepsilon}_{L^\infty_T( B^{\frac{d}{p}}_{p,1})} \nonumber
    \\&\lesssim X^\varepsilon(t)^2,\nonumber
\end{align}
since, for $\alpha\geq0$, we have $\varepsilon^\alpha\|a^\varepsilon\|^{m,\varepsilon}_{\dot{B}^{\frac{d}{2}+\alpha}_{2,1}} \lesssim \|a^\varepsilon\|^{m,\varepsilon}_{\dot{B}^{\frac{d}{2}}_{2,1}}$ and similarly for $q^\varepsilon$. Moreover, we have
\begin{align}\label{toupdateP2}
\int_0^T\varepsilon^2\|\d_tJ(a^\varepsilon)\|^{h,\var}_{\dot{B}^{\frac{d}{2}}_{2,1}}\|\varepsilon q^\varepsilon\|^{h,\var}_{\dot{B}^{\frac{d}{2}+1}_{2,1}}& \lesssim \varepsilon^2\|\div v^\varepsilon\|^{h,\varepsilon}_{L^2_T(\dot{B}^{\frac{d}{2}}_{2,1})}  \varepsilon\|q\|^{h,\var}_{L^2_T(\dot{B}^{\frac{d}{2}+1}_{2,1})} 
    \\& + \varepsilon^2\|v^\varepsilon\cdot \nabla a^\varepsilon\|^{h,\varepsilon}_{L^\infty_T(\dot{B}^{\frac{d}{2}}_{2,1})}  \varepsilon\|q\|^{h,\var}_{L^1_T(\dot{B}^{\frac{d}{2}+1}_{2,1})}.\nonumber
\end{align}
Using Proposition \ref{LPP}, we have
\begin{align}
\|v^\varepsilon\cdot \nabla a^\varepsilon\|^{h,\varepsilon}_{L^\infty_T(\dot{B}^{\frac{d}{2}}_{2,1})}&\lesssim\varepsilon\|v^\varepsilon\|_{L^\infty_T(\dot{B}^{\frac dp}_{p,1})}\|\nabla a^\varepsilon\|^{h}_{L^\infty_T(\dot{B}^{\frac d2}_{2,1})}+\varepsilon\|\nabla a^\varepsilon\|_{L^\infty_T(\dot{B}^{\frac dp}_{p,1})} \|v^\varepsilon\|^{h}_{L^\infty_T(\dot{B}^{\frac d2}_{2,1})}\nonumber\\&+ \varepsilon\|v^\varepsilon\|^{\ell,\varepsilon}_{L^\infty_T(\dot{B}^{\frac dp}_{p,1})}\|\nabla a^\varepsilon\|^{\ell,\varepsilon}_{L^\infty_T(\dot{B}^{\frac dp}_{p,1})},
\end{align}
which yields
\begin{align}\label{toupdateP3}
\int_0^T\varepsilon^2\|\d_tJ(a^\varepsilon)\|^{h,\var}_{\dot{B}^{\frac{d}{2}}_{2,1}}\|\varepsilon q^\varepsilon\|^{h,\var}_{\dot{B}^{\frac{d}{2}+1}_{2,1}}\lesssim X^\varepsilon(t)^2+X^\varepsilon(t)^3.
\end{align}
For the remaining terms, we rely on Proposition \ref{LPP}. Concerning $I^\varepsilon_{1}$, we have
   \begin{align}
        \label{eq:Ihigh1}
     \varepsilon^3\| q^\varepsilon\cdot\nabla v^\varepsilon\|^h_{L^1_T(\dot{B}^{\frac d2+1}_{2,1})}\nonumber&\lesssim   \varepsilon^3\|q^{\varepsilon}\|_{L^2_T(\dot{B}^{\frac dp}_{p,1})}\| \nabla v^\varepsilon\|^{h}_{L^2_T(\dot{B}^{\frac d2+1}_{2,1})}+ \varepsilon^3\|\nabla v^\varepsilon\|_{L^2_T(\dot{B}^{\frac dp}_{p,1})} \|q^{\varepsilon}\|^{h}_{L^2_T(\dot{B}^{\frac d2+1}_{2,1})}
      \\&+ \varepsilon^3\|q^{\varepsilon}\|^{\ell,\varepsilon}_{L^2_T(\dot{B}^{\frac dp+1}_{p,1})} \|\nabla v^\varepsilon\|^{\ell,\varepsilon}_{L^2_T(\dot{B}^{\frac dp}_{p,1})}+ \varepsilon^3\|q^{\varepsilon}\|^{\ell,\varepsilon}_{L^2_T(\dot{B}^{\frac dp}_{p,1})} \|\nabla v^\varepsilon\|^{\ell,\varepsilon}_{L^2_T(\dot{B}^{\frac dp+1}_{p,1})}
      \nonumber\\&\lesssim X^\varepsilon(t)^2,
    \end{align}
where we recall that the notation $\|\cdot\|^{\ell,\varepsilon}$ refers to the sum of low-frequency and medium-frequency norms.
Handling $q^\varepsilon\div v^\varepsilon$ is the same as handling $q^\varepsilon\cdot\nabla v^\varepsilon$ and we obtain $\|\varepsilon^3 I_1^\varepsilon\|^{h,\var}_{L^1_T(\dot{B}^{\frac{d}{2}+1}_{2,1})}\lesssim X^\varepsilon(t)^2.$ 
Regarding $F_1^\varepsilon=a^\varepsilon\div v^\varepsilon$, we have
\begin{align}\label{R-E49}
\varepsilon\|a^\varepsilon\div v^\varepsilon\|^{h,\varepsilon}_{L^1_T(\dot{B}^{\frac d2+1}_{2,1})}\nonumber&\lesssim\varepsilon\|a^\varepsilon\|_{L^\infty_T(\dot{B}^{\frac dp}_{p,1})}\|\div v^\varepsilon\|^{h,\varepsilon}_{L^1_T(\dot{B}^{\frac d2+1}_{2,1})}+\varepsilon\|\div v^\varepsilon\|_{L^1_T(\dot{B}^{\frac dp}_{p,1})}\|a^\varepsilon\|^{h,\varepsilon}_{L^{\infty}_T(\dot{B}^{\frac d2+1}_{2,1})}\nonumber\\&+
\varepsilon\|a^\varepsilon\|^{\ell,\varepsilon}_{L^{\infty}_T(\dot{B}^{\frac dp}_{p,1})}\|\div v^\varepsilon\|^{\ell,\varepsilon}_{L^1_T(\dot{B}^{\frac dp+1}_{p,1})}+\varepsilon\|a^\varepsilon\|^{\ell,\varepsilon}_{L^{\infty}_T(\dot{B}^{\frac dp+1}_{p,1})}\|\div v^\varepsilon\|^{\ell,\varepsilon}_{L^1_T(\dot{B}^{\frac dp}_{p,1})}\nonumber\\&\lesssim X^\varepsilon(t)X^{h,\varepsilon}(t)+(X^{m,\varepsilon}(t)+X^{\ell,\varepsilon}(t))^2\lesssim X^\varepsilon(t)^2.
\end{align}
Concerning $H^\varepsilon$, we have
\begin{align}\label{R-E51}
\varepsilon\|J(a^\varepsilon)\div q^\varepsilon\|^{h}_{L^1_T( \dot{B}^{\frac{d}{2}}_{2,1})}&\lesssim\varepsilon\|J(a^\varepsilon)\|_{L^2_T(\dot{B}^{\frac dp}_{p,1})}\|\div q^\varepsilon\|^{h}_{L^2_T(\dot{B}^{\frac d2}_{2,1})}+\varepsilon\|\div q^\varepsilon\|_{L^1_T(\dot{B}^{\frac dp}_{p,1})} \|J(a^\varepsilon)\|^{h}_{L^\infty_T(\dot{B}^{\frac d2}_{2,1})}\nonumber\\&+ \varepsilon\|J(a^\varepsilon)|^{\ell,\varepsilon}_{L^\infty_T(\dot{B}^{\frac dp}_{p,1})}\|\div q^\varepsilon\|^{\ell,\varepsilon}_{L^1_T(\dot{B}^{\frac dp}_{p,1})}\nonumber
\\&\lesssim X^\varepsilon(t)^2.
\end{align}
Using that $1/a^\varepsilon=1+I(a^\varepsilon)$, we have
\begin{align}\label{R-E510}
&\varepsilon\|\dfrac{N(\nabla v^\varepsilon,\nabla v^\varepsilon)}{1+a^\varepsilon}\|^{h}_{L^1_T( \dot{B}^{\frac{d}{2}}_{2,1})}\\&\lesssim \varepsilon\|(1+I(a^\varepsilon))\nabla v^\varepsilon\|_{L^2_T(\dot{B}^{\frac dp}_{p,1})}\|\nabla v^\varepsilon\|^{h}_{L^2_T(\dot{B}^{\frac d2}_{2,1})}+\varepsilon\|\nabla v^\varepsilon\|_{L^2_T(\dot{B}^{\frac dp}_{p,1})} \|(1+I(a^\varepsilon))\nabla v^\varepsilon\|^{h}_{L^2_T(\dot{B}^{\frac d2}_{2,1})}\nonumber\\&+ \varepsilon\|(1+I(a^\varepsilon))\nabla v^\varepsilon\|^{\ell,\varepsilon}_{L^2_T(\dot{B}^{\frac dp}_{p,1})}\|\nabla v^\varepsilon\|^{\ell,\varepsilon}_{L^2_T(\dot{B}^{\frac dp}_{p,1})}\nonumber\\&\lesssim \varepsilon(1+\|a^\varepsilon\|_{L^{\infty}_T(\dot{B}^{\frac dp}_{p,1})})\|v^\varepsilon\|_{L^2_T(\dot{B}^{\frac {d}{p}+1}_{p,1})}\|v^\varepsilon\|^{h}_{L^2_T(\dot{B}^{\frac {d}{2}+1}_{2,1})}+\varepsilon\|v^\varepsilon\|_{L^2_T(\dot{B}^{\frac {d}{p}+1}_{p,1})}\|(1+I(a^\varepsilon))\nabla v^\varepsilon\|^{h}_{L^2_T(\dot{B}^{\frac d2}_{2,1})}\nonumber\\&+
\varepsilon(1+\|a^\varepsilon\|_{L^{\infty}_T(\dot{B}^{\frac dp}_{p,1})})\|v^\varepsilon\|_{L^2_T(\dot{B}^{\frac {d}{p}+1}_{p,1})}\|v^\varepsilon\|^{\ell,\varepsilon}_{L^2_T(\dot{B}^{\frac {d}{p}+1}_{p,1})}. \nonumber
\end{align}
Then, we have
$$\|(1+I(a^\varepsilon))\nabla v^\varepsilon\|^{h}_{L^2_T(\dot{B}^{\frac d2}_{2,1})}\leq \|v^\varepsilon\|^{h}_{L^2_T(\dot{B}^{\frac d2+1}_{2,1})}+\|I(a^\varepsilon)\nabla v^\varepsilon\|^{h}_{L^2_T(\dot{B}^{\frac d2}_{2,1})},$$
and from a composition law
\begin{eqnarray}\label{R-E52}
\|I(a)\|^{h}_{L^{\infty}_T(\dot{B}^{\frac d2}_{2,1})}&\lesssim& (1+\|a\|^{\ell,\varepsilon}_{L^{\infty}_T(\dot{B}^{\frac {d}{p}}_{p,1})}+\varepsilon\|a\|^{h,\varepsilon}_{L^{\infty}_T(\dot{B}^{\frac {d}{2}}_{2,1})})(\|a\|^{\ell,\varepsilon}_{L^{\infty}_T(\dot{B}^{\frac {d}{p}}_{p,1})}+\|a\|^{h,\varepsilon}_{L^{\infty}_T(\dot{B}^{\frac {d}{2}}_{2,1})})\nonumber\\&\lesssim& (1+X^\varepsilon(t))X^\varepsilon(t).
\end{eqnarray}
Employing Proposition \ref{LPP} implies that 
\begin{eqnarray}\label{R-E53}
\|I(a^\varepsilon)\nabla v^\varepsilon\|^{h}_{L^2_T(\dot{B}^{\frac d2}_{2,1})}&\lesssim& \|I(a^\varepsilon)\|_{L^2_T(\dot{B}^{\frac dp}_{p,1})}\|\nabla v^\varepsilon\|^{h}_{L^{\infty}_T(\dot{B}^{\frac d2}_{2,1})}+\|\nabla v^\varepsilon\|_{L^2_T(\dot{B}^{\frac dp}_{p,1})}\|I(a)\|^{h}_{L^{\infty}_T(\dot{B}^{\frac dp}_{p,1})}\nonumber\\&& +\|I(a)\|^{\ell,\varepsilon}_{L^{\infty}_T(\dot{B}^{\frac dp}_{p,1})}\|\nabla v^\varepsilon\|^{\ell,\varepsilon}_{L^{2}_T(\dot{B}^{\frac dp}_{p,1})}\nonumber\\&\lesssim& X^\varepsilon(t)X^{h,\varepsilon}(t)+(1+X^\varepsilon(t))X^\varepsilon(t)^2+X^\varepsilon(t)(X^{m,\varepsilon}(t)+X^{\ell,\varepsilon}(t))\nonumber\\&\lesssim& (1+X^\varepsilon(t))X^\varepsilon(t)^2.
\end{eqnarray}
Gathering \eqref{R-E51}, \eqref{R-E510} \eqref{R-E52} and \eqref{R-E53}, we arrive at 
\begin{eqnarray}\label{R-E533}
\varepsilon\|\dfrac{N(\nabla v^\varepsilon,\nabla v^\varepsilon)}{1+a^\varepsilon}\|^{h}_{L^1_T( \dot{B}^{\frac{d}{2}}_{2,1})}\lesssim(1+X^\varepsilon(t)+X^\varepsilon(t)^2)X^\varepsilon(t)^2.
\end{eqnarray}
Similarly, we have
\begin{align}\label{R-E54}
\varepsilon\|\wt H_1(a^\varepsilon)\theta^\varepsilon \div v^\varepsilon\|^{h}_{L^1_T(\dot{B}^{\frac d2}_{2,1})}\nonumber&\lesssim \varepsilon\|\wt H_1(a^\varepsilon)\theta^\varepsilon\|_{L^2_T(\dot{B}^{\frac dp}_{p,1})}\|\div v^\varepsilon\|^{h}_{L^2_T(\dot{B}^{\frac d2}_{2,1})}\nonumber\\&+\varepsilon\|\div v^\varepsilon\|_{L^2_T(\dot{B}^{\frac dp}_{p,1})}\|\wt H_1(a^\varepsilon)\theta^\varepsilon\|^{h}_{L^2_T(\dot{B}^{\frac d2}_{2,1})}+\varepsilon\|\wt H_1(a^\varepsilon)\theta^\varepsilon\|^{\ell,\varepsilon}_{L^2_T(\dot{B}^{\frac dp}_{p,1})}\|\div v^\varepsilon\|^{\ell,\varepsilon}_{L^2_T(\dot{B}^{\frac dp}_{p,1})}
\nonumber\\&\lesssim\varepsilon\|a\|_{L^{\infty}_T(\dot{B}^{\frac dp}_{p,1})}\|\theta^\varepsilon\|_{L^{2}_T(\dot{B}^{\frac dp}_{p,1})}\|v^\varepsilon\|^{h}_{L^2_T(\dot{B}^{\frac d2+1}_{2,1})}
+\varepsilon\|v^\varepsilon\|_{L^2_T(\dot{B}^{\frac dp+1}_{p,1})}\|\wt H_1(a^\varepsilon)\theta^\varepsilon\|^{h}_{L^2_T(\dot{B}^{\frac d2}_{2,1})}\nonumber\\&+\varepsilon\|a\|_{L^{\infty}_T(\dot{B}^{\frac dp}_{p,1})}\|\theta^\varepsilon\|_{L^{2}_T(\dot{B}^{\frac dp}_{p,1})}\|v^\varepsilon\|^{\ell,\varepsilon}_{L^2_T(\dot{B}^{\frac dp+1}_{p,1})},
\end{align}
where \begin{align}\label{R-E55}&\|\wt H_1(a^\varepsilon)\theta^\varepsilon\|^{h}_{L^2_T(\dot{B}^{\frac d2}_{2,1})}\nonumber\\&\lesssim  \|a^\varepsilon\|_{L^{\infty}_T(\dot{B}^{\frac dp}_{p,1})}\|\theta^\varepsilon\|^{h}_{L^{2}_T(\dot{B}^{\frac d2}_{2,1})}+\|\theta^\varepsilon\|_{L^{2}_T(\dot{B}^{\frac dp}_{p,1})}\|H_1(a^\varepsilon)\|^{h}_{L^{\infty}_T(\dot{B}^{\frac d2}_{2,1})}+\|a^\varepsilon\|^{\ell,\varepsilon}_{L^{\infty}_T(\dot{B}^{\frac dp}_{p,1})}\|\theta^\varepsilon\|^{\ell,\varepsilon}_{L^{2}_T(\dot{B}^{\frac dp}_{p,1})}\nonumber\\ \hspace{5mm} &\lesssim  (1+X^\varepsilon(t))X^\varepsilon(t)^2. \end{align}
Hence, we deduce that 
\begin{eqnarray}\label{R-E56}
\varepsilon\|\wt H_1(a^\varepsilon)\theta^\varepsilon \div v^\varepsilon\|^{h}_{L^1_T(\dot{B}^{\frac d2}_{2,1})}\lesssim (1+X^\varepsilon(t))X^\varepsilon(t)^3.
\end{eqnarray}
Bounding $H_1$ follows from similar considerations and we have $\|\varepsilon^2H_1^\varepsilon\|^{h,\var}_{L^1_T(\dot{B}^{\frac{d}{2}+1}_{2,1})}\leq X^\varepsilon(t)^2.$

Finally, we bound the nonlinear terms in $G^{\varepsilon}$. We have
\begin{align}\label{R-E57}
\varepsilon\|v^{\varepsilon}\cdot\nabla v^{\varepsilon}\|^{h}_{L^1_T(\dot{B}^{\frac d2}_{2,1})}\nonumber&\lesssim\varepsilon\|v^{\varepsilon}\|_{L^{2}_T(\dot{B}^{\frac dp}_{p,1})}\|\nabla v^{\varepsilon}\|^{h}_{L^2_T(\dot{B}^{\frac d2}_{2,1})}+\varepsilon\|\nabla v^{\varepsilon}\|_{L^1_T(\dot{B}^{\frac dp}_{p,1})}\|v^{\varepsilon}\|^{h}_{L^{\infty}_T(\dot{B}^{\frac d2}_{2,1})}\nonumber\\&+ \varepsilon\|v^{\varepsilon}\|^{\ell,\varepsilon}_{L^{2}_T(\dot{B}^{\frac dp}_{p,1})}\|\nabla v^{\varepsilon}\|^{\ell,\varepsilon}_{L^{2}_T(\dot{B}^{\frac dp}_{p,1})}\nonumber\\&\lesssim X^\varepsilon(t)^2.
\end{align}
Using \eqref{eq:prod433} together with composition estimates, we obtain
\begin{align}\label{R-E58} \nonumber
\varepsilon\|J(a^{\varepsilon})\mathcal{A}v^\varepsilon\|^{h}_{L^1_T(\dot{B}^{\frac d2}_{2,1})}&\lesssim \varepsilon\|a^{\varepsilon}\|_{L^{\infty}_T( \dot{B}^{\frac{d}{p}}_{p,1})}\|(v^{h,\varepsilon},v^{\ell})\|_{L^1_T( \dot{B}^{\frac{d}{2}+2}_{2,1})} +\varepsilon\|v\|^{m,\varepsilon}_{L^{1}_T( \dot{B}^{\frac{d}{p}+2}_{p,1})}\|a^{\varepsilon}\|_{L^\infty_T( \dot{B}^{\frac{d}{p}}_{p,1})}\\&+\varepsilon \|a\|^{\ell,\varepsilon}_{L^2_T(\dot{B}^{\frac{d}{p}}_{p,1})}\|v\|^{m,\varepsilon}_{L^2_T(\dot{B}^{\frac{d}{p}+2}_{p,1})}\\&\lesssim X^\varepsilon(t)^2.\nonumber 
\end{align}
For the terms $K_1$, using \eqref{eq:prod433}, we obtain
\begin{align}\label{R-E59} \|K_1(a^{\varepsilon})\nabla a^{\varepsilon}\|^{h}_{L^1_T( \dot{B}^{\frac{d}{2}}_{2,1})}\nonumber&\lesssim \varepsilon\|a^{\varepsilon}\|_{L^{\infty}_T( \dot{B}^{\frac{d}{p}}_{p,1})}\|(a^{h,\varepsilon},a^{\ell})\|_{L^1_T( \dot{B}^{\frac{d}{2}+1}_{2,1})} +\varepsilon\|a\|^{m,\varepsilon}_{L^{2}_T( \dot{B}^{\frac{d}{p}}_{p,1})}\|a^{\varepsilon}\|_{L^2_T( \dot{B}^{\frac{d}{p}}_{p,1})}\\&\quad+ \varepsilon\|a\|^{\ell,\varepsilon}_{L^2_T(\dot{B}^{\frac{d}{p}}_{p,1})}\|a\|^{m,\varepsilon}_{L^2_T(\dot{B}^{\frac{d}{p}}_{p,1})} \lesssim X^\varepsilon(t)^2.
\end{align}
Using \eqref{eq:prod422}, we have
\begin{align}\label{R-E60}\varepsilon\|K_2(a^{\varepsilon})\nabla \theta^{\varepsilon}\|^{h}_{L^1_T( \dot{B}^{\frac{d}{2}}_{2,1})}\nonumber&\lesssim \varepsilon\|K_2(a^{\varepsilon})\|_{L^2_T( \dot{B}^{\frac{d}{p}}_{p,1})}
\|\nabla \theta^{\varepsilon}\|^{h}_{L^2_T(\dot{B}^{\frac{d}{2}}_{2,1})}+\varepsilon\|\nabla \theta^{\varepsilon}\|_{L^1_T( \dot{B}^{\frac{d}{p}}_{p,1})}\|K_2(a^{\varepsilon})\|^{h}_{L^\infty_T( \dot{B}^{\frac{d}{2}}_{2,1})}\nonumber\\&+\varepsilon\|K_2(a^{\varepsilon})\|^{\ell,\varepsilon}_{L^\infty_T( \dot{B}^{\frac{d}{p}}_{p,1})}\|\nabla \theta^{\varepsilon}\|^{\ell,\varepsilon}_{L^1_T( \dot{B}^{\frac{d}{p}}_{p,1})}
\nonumber\\&\lesssim(1+X^\varepsilon(t))X^\varepsilon(t)^2
\end{align}
and using \eqref{eq:prod433}
\begin{align}\label{R-E61}
\varepsilon\|\theta^{\varepsilon}\nabla K_3(a^{\varepsilon})\|^{h}_{L^1_T( \dot{B}^{\frac{d}{2}}_{2,1})}\nonumber&\lesssim \varepsilon\|\theta^{\varepsilon}\|_{L^{2}_T( \dot{B}^{\frac{d}{p}}_{p,1})}\|(a^{h,\varepsilon},a^{\ell})\|_{L^2_T( \dot{B}^{\frac{d}{2}+1}_{2,1})} +\varepsilon\|a^\varepsilon\|^{m,\varepsilon}_{L^{2}_T( \dot{B}^{\frac{d}{p}}_{p,1})}\|\theta^{\varepsilon}\|_{L^2_T( \dot{B}^{\frac{d}{p}}_{p,1})}\\&\quad+ \varepsilon\|\theta\|^{\ell,\varepsilon}_{L^2_T(\dot{B}^{\frac{d}{p}}_{p,1})}\|a\|^{m,\varepsilon}_{L^2_T(\dot{B}^{\frac{d}{p}+1}_{p,1})} \lesssim X^\varepsilon(t)^2.
\end{align}
We now control $\|G^\varepsilon\|^{h,\var}_{L^2_T(\dot{B}^{\frac{d}{2}}_{2,1})}$.
We only treat the term $J(a^{\varepsilon})\mathcal{A}v^\varepsilon$, the other terms can be treated in the same way as in the medium-frequency regime. Using \eqref{eq:prod433}, we have
\begin{align}
\nonumber\varepsilon\|J(a^{\varepsilon})\mathcal{A}v^\varepsilon\|^{h,\varepsilon}_{L^2_T( \dot{B}^{\frac{d}{2}}_{2,1})}&\lesssim \varepsilon\|a^{\varepsilon}\|_{L^{\infty}_T( \dot{B}^{\frac{d}{p}}_{p,1})}\|(v^{h,\varepsilon},v^{\ell})\|_{L^2_T( \dot{B}^{\frac{d}{2}+2}_{2,1})} \\&+\varepsilon\|v\|^{m,\varepsilon}_{L^{2}_T( \dot{B}^{\frac{d}{p}+2}_{p,1})}\|J(a^{\varepsilon})\|_{L^\infty_T( \dot{B}^{\frac{d}{p}}_{p,1})}+ \varepsilon\|a\|^{\ell,\varepsilon}_{L^\infty_T(\dot{B}^{\frac{d}{p}}_{p,1})}\|v\|^{m,\varepsilon}_{L^2_T(\dot{B}^{\frac{d}{p}+2}_{p,1})} \\\nonumber &\lesssim X^\varepsilon(t)^2.
\end{align}
\end{proof}

\subsection{Concluding the proof of Theorem \ref{thm:exist}}

Gathering the estimates from the previous section, we obtain
\begin{align}\label{est:ccl}
X^\varepsilon(t)\leq X^\varepsilon_0+X^\varepsilon(t)^2.
\end{align}
From here, a classical bootstrap argument applied to the local-in-time solution constructed in \cite{FA2023-2}, similar to the one used in \cite{HeDanchin}, allows us to conclude the existence of global-in-time solutions for the system \ref{LinearHypNSCzero}.
We omit here the estimates to justify the uniqueness of the solutions as it follows closely the stability estimates that are established in Section \ref{sec:relaxation}. This concludes the proof Theorem \ref{thm:exist}.

\qed
\eqref{LinearHypNSCzero}
\section{Optimal time-decay estimates: Proof of Theorem \ref{Thm:decay}} \label{sec:decay}
In this section, we follow the Lyapunov-type energy argument developed by Xin and Xu in \cite{XX} to derive large-time decay estimates. The main difference is that we need to perform the analysis uniformly with respect to the relaxation parameter $\var$ in each frequency domain.
\subsection{Linear estimates}
For low frequencies, $j\leq J_0$, it follows from \eqref{feiasgn} that
\begin{eqnarray}\label{R-E62}
\frac{d}{dt}\mathcal{L}_{j}+c2^{2j}\|( a_j^\varepsilon,v_j^\varepsilon,\theta_j^\varepsilon)\|_{L^2}\lesssim K\|Q_j^\varepsilon\|_{L^2}+\|(F_j^\varepsilon,G_j^\varepsilon,H_j^\varepsilon)\|_{L^2}.
\end{eqnarray}
On the other hand, we have
\begin{eqnarray}\label{R-E63}
\frac{\varepsilon}{2}\frac{d}{dt}\|Q_{j}^{\varepsilon}\|^2_{L^2}+\frac{1}{\varepsilon}\|Q_{j}^{\varepsilon}\|^2_{L^2}\lesssim \varepsilon\|f_{1j}^{\varepsilon}+I^{\varepsilon}_{j}+\kappa \nabla H^{\varepsilon}_{j}\|_{L^2}\|Q_{j}^{\varepsilon}\|_{L^2},
\end{eqnarray}
which leads to
\begin{eqnarray}\label{R-E64}
\frac{d}{dt}\|\varepsilon Q_{j}^{\varepsilon}\|_{L^2}+\frac{1}{\varepsilon}\|Q_{j}^{\varepsilon}\|_{L^2}\lesssim \varepsilon(\|f_{1j}^{\varepsilon}+I^{\varepsilon}_{j}+\kappa \nabla H^{\varepsilon}_{j}\|_{L^2}).
\end{eqnarray}
Adding \eqref{R-E62} to \eqref{R-E64}, we have
\begin{eqnarray}\label{R-E65712}
\frac{d}{dt}(\mathcal{L}_{j}+\|\varepsilon Q_{j}^{\varepsilon}\|_{L^2})+2^{2j}\|( a_j^\varepsilon,v_j^\varepsilon,\theta_j^\varepsilon)\|_{L^2}+\frac{1}{2\varepsilon}\|Q_{j}^{\varepsilon}\|_{L^2}\lesssim  \varepsilon 2^{2j}\|q^{\varepsilon}_{j}\|_{L^2}+\|(F_j^\varepsilon,G_j^\varepsilon,H_j^\varepsilon,\varepsilon I^{\varepsilon})\|_{L^2}.
\end{eqnarray}
Note that $Q^{\varepsilon}=q^{\varepsilon}+\nabla\theta^{\varepsilon}$, thus choosing $\varepsilon$ sufficiently small, we obtain
\begin{align}\label{R-E66}
\frac{d}{dt}\|(a^\varepsilon,v^\varepsilon,\theta^\varepsilon,\varepsilon q^{\varepsilon})^{\ell}\|_{\dot{B}^{\frac {d}{2}-1}_{2,1}}+\|( a^\varepsilon,v^\varepsilon,\theta^\varepsilon)^{\ell}\|_{\dot{B}^{\frac {d}{2}+1}_{2,1}}+\frac{1}{\varepsilon}\|Q^{\varepsilon,\ell}\|_{\dot{B}^{\frac {d}{2}-1}_{2,1}}&+\|q^{\varepsilon,\ell}\|_{\dot{B}^{\frac {d}{2}}_{2,1}}\nonumber\\ &\lesssim \|(F^\varepsilon,G^\varepsilon,H^\varepsilon,\varepsilon I^{\varepsilon})\|^{\ell}_{\dot{B}^{\frac {d}{2}-1}_{2,1}}.
\end{align}

For medium frequencies, $J_0\leq j\leq J_\varepsilon$, applying the operator $\partial_{i}\dot{\Delta}_{j}$  to the first equation in \eqref{LinearHypNSCMed} and denoting $R_{j}^{i}\triangleq[u^{\varepsilon}\cdot\nabla,\partial_{i}\dot{\Delta}_{j}]a^{\varepsilon}$
gives
\begin{equation}\label{R-E67}
 \partial_{t}\partial_{i}a_j+u^{\varepsilon}\cdot\nabla\partial_{i}a_j+\partial_{i}a_j=   -\partial_{i}\dot{\Delta}_{j}(a\div u)-\partial_{i}\div w_j+R_{j}^{i},\quad i=1,\cdots,d.\end{equation}
Multiplying  by $|\partial_{i}a_j^{\varepsilon}|^{p-2}\partial_{i}a_j^{\varepsilon},$   integrating on $\R^d,$ and performing an integration by
parts in the second term of  \eqref{R-E67}, we get
$$\displaylines{
\frac1p\frac{d}{dt}\|\partial_{i}a^{\varepsilon}_j\|_{L^p}^p+\|\partial_{i}a_j^{\varepsilon}\|_{L^p}^p=\frac1p\int\div u^{\varepsilon}\:|\partial_{i}a_j|^p\,dx
+\int\bigl(R_j^i -\partial_{i}\dot{\Delta}_{j}(a^{\varepsilon}\div u^{\varepsilon})-\partial_{i}\div w_j^{\varepsilon})|\partial_{i}a^{\varepsilon}_j|^{p-2}\partial_{i}a^{\varepsilon}_j\,dx.}
$$
Summing up on $i=1,\cdots,d,$ and applying H\"older  and Bernstein inequalities implies that
\begin{multline}\label{R-E68}
\frac1p\frac{d}{dt}\|\nabla a_j^{\varepsilon}\|_{L^p}+\|\nabla a_j^{\varepsilon}\|_{L^p}\leq\frac{1}{p}\|\mathrm{div}u^{\varepsilon}\|_{L^\infty}\|\nabla a_j^{\varepsilon}\|_{L^p}+\|\nabla\ddj(a^{\varepsilon}\mathrm{div}u^{\varepsilon})\|_{L^p}+2^{2j}\|w_j^{\varepsilon}\|_{L^p}+\|R_j^{\varepsilon}\|_{L^p}.
\end{multline}
Similarly, we have
\begin{eqnarray}\label{R-E69}
\frac{d}{dt}\|\varepsilon\Lambda^{-1}Q_j^{\varepsilon}\|_{L^p}+\frac{1}{\varepsilon}\|\Lambda^{-1}Q_j^{\varepsilon}\|_{L^p}\leq \var \|\Lambda^{-1}(f_{1j}^\varepsilon+\nabla a^\varepsilon_{j}+ I^\varepsilon_{j}+\kappa \nabla H^\varepsilon_{j})\|_{L^p}.
\end{eqnarray}
Applying $\ddj$ to the second and third equations yields for all $j\in \mathbb{Z}$, with the help of Lemma \ref{maximalL1L2}, we obtain
\begin{eqnarray}\label{R-E70}
\frac{d}{dt}\|w^{\varepsilon}_{j}\|_{L^p}+2^{2j}\|w^{\varepsilon}_{j}\|_{L^p}\leq \|w^{\varepsilon}_{j}\|_{L^p}+2^{-2j}\|\nabla a^{\varepsilon}_{j}\|_{L^p}+\|\nabla \theta^{\varepsilon}_{j}\|_{L^p}+\|G^{\varepsilon}_{j}\|_{L^p}+2^{-2j}\|\nabla F^{\varepsilon}_{j}\|_{L^p}
\end{eqnarray}
and
\begin{eqnarray}\label{R-E71}
\frac{d}{dt}\|\Lambda^{-1}\theta_j^{\varepsilon}\|_{L^p}+2^{2j}\|\Lambda^{-1}\theta_j^{\varepsilon}\|_{L^p}\leq\|w^{\varepsilon}_{j}\|_{L^p}
+\|\Lambda^{-1}a^{\varepsilon}_{j}\|_{L^p}+\|\Lambda^{-1}\div Q^{\varepsilon}_{j}\|_{L^p}+\|\Lambda^{-1}H^{\varepsilon}_{j}\|_{L^p}.
\end{eqnarray}
Since $J_0 \leq j \leq J_{\varepsilon}$, we have $2^{-2j}\leq 1/K^2$ and $2^{J_{\varepsilon}}\leq k/\varepsilon$. Choosing $K$ suitably large and $k$ suitably small in the inequalities \eqref{R-E68}- \eqref{R-E71},  we get
\begin{align}\label{R-E72}
\frac{d}{dt}(\|\nabla a_j^{\varepsilon}\|_{L^p}+\|w^{\varepsilon}_{j}\|_{L^p}&+\|\varepsilon\Lambda^{-1}Q_j^{\varepsilon}\|_{L^p}+\|\Lambda^{-1}\theta_j^{\varepsilon}\|_{L^p})\\&+\|\nabla a_j^{\varepsilon}\|_{L^p}+2^{2j}\|w^{\varepsilon}_{j}\|_{L^p}+\frac{1}{\varepsilon}\|\Lambda^{-1}Q_j^{\varepsilon}\|_{L^p}+2^{2j}\|\Lambda^{-1}\theta_j^{\varepsilon}\|_{L^p}\nonumber\\
&\lesssim \|G^{\varepsilon}_{j}\|_{L^p}+2^{-2j}\|\nabla F^{\varepsilon}_{j}\|_{L^p}+\|\mathrm{div}u^{\varepsilon}\|_{L^\infty}\|\nabla a_j^{\varepsilon}\|_{L^p}+\|\nabla\ddj(a^{\varepsilon}\mathrm{div}u^{\varepsilon})\|_{L^p}\nonumber\\&\quad+\|R_j^{\varepsilon}\|_{L^p}+\var \|\Lambda^{-1}(f_{1j}^\varepsilon+ I^\varepsilon_{j}+\kappa \nabla H^\varepsilon_{j})\|_{L^p}. \nonumber
\end{align}
Recalling that $Q^{\var}=q^{\var}+\kappa \nabla \theta^{\var}$ and $ w^{\var}=v^{\var}+(-\Delta)^{-1}\nabla a^{\varepsilon},$
we have
\begin{eqnarray}\label{R-E73}
\var\|\nabla\div \Lambda^{-1}q^{\var}_{j}\|_{L^p}\lesssim \frac{k^2}{\var}\|\Lambda^{-1}Q_j^{\varepsilon}\|_{L^p}+k2^{j}\|\theta_j^{\varepsilon}\|_{L^p}
\end{eqnarray}
and
\begin{eqnarray}\label{R-E74}
\var\|\nabla\div \Lambda^{-1}v^{\var}_{j}\|_{L^p}\lesssim \frac{\var}{K}(2^{2j}\|w^{\varepsilon}_{j}\|_{L^p})+\frac{\var}{K}\|\nabla a^{\varepsilon}_{j}\|_{L^p}.
\end{eqnarray}
Multiplying by $2^{j(\frac{d}{p}-1)}$ on both sides of \eqref{R-E72} and summing over $j\in [J_0,J_{\var}]$, it follows from \eqref{R-E73}-\eqref{R-E74} that
\begin{multline}\label{R-E75}
\frac{d}{dt}(\|a^{\var}\|^{m}_{\dot{B}^{d/p}_{p,1}}+\|w^{\var}\|^{m}_{\dot{B}^{d/p-1}_{p,1}}+\|\var Q^{\var}\|^{m}_{\dot{B}^{d/p-2}_{p,1}}+\|\theta^{\var}\|^{m}_{\dot{B}^{d/p-2}_{p,1}})\\+\|a^{\var}\|^{m}_{\dot{B}^{d/p}_{p,1}}
+\|w^{\var}\|^{m}_{\dot{B}^{d/p+1}_{p,1}}+\frac{1}{\varepsilon}\|Q^{\var}\|^{m}_{\dot{B}^{d/p-2}_{p,1}}+\|\theta^{\var}\|^{m}_{\dot{B}^{d/p}_{p,1}}
\\ \lesssim \|G^{\var}\|^{m}_{\dot{B}^{d/p-1}_{p,1}}+\|F^{\var}\|^{m}_{\dot{B}^{d/p-2}_{p,1}}+\|\nabla v^{\var}\|_{\dot{B}^{d/p}_{p,1}}\|a^{\var}\|_{\dot{B}^{d/p}_{p,1}}+\varepsilon\|I^{\varepsilon}\|^{m}_{\dot{B}^{d/p-2}_{p,1}}
+\|H^{\varepsilon}\|^{m}_{\dot{B}^{d/p-2}_{p,1}}.
\end{multline}

In the high-frequency regime,  we have $j\geq J_0$, so it follows from the previous lines that
\begin{eqnarray*}\label{R-E76}
\frac{d}{dt}(\|\nabla a^{\var}_{j}\|_{L^p}+\|w^{\var}_{j}\|_{L^p})+\|\nabla a^{\var}_{j}\|_{L^p}+2^{2j}\|w^{\var}_{j}\|_{L^p}\leq \|\nabla \theta_{j}\|_{L^p}+\|G^{\var}_{j}\|_{L^p}+2^{-2j}\|\nabla F^{\var}_{1j}\|_{L^p}\nonumber\\+\|\div v^{\var}\|_{L^\infty}\|\nabla a^{\var}_{j}\|_{L^p}+\|R_{j}\|_{L^p}.
\end{eqnarray*}
On the other hand, 
\begin{align}\label{R-E77}
\frac 12\frac{d}{dt}\|(\sqrt{\kappa}\theta^{\var}_{j},\varepsilon q^{\var}_{j})\|^2_{L^2}+&\frac{1}{\varepsilon^2}\|\varepsilon q^{\var}_{j}\|^2_{L^2} \nonumber\\&\lesssim \int \div w^{\var}_{j}\theta^{\var}_{j}+\int H^{\var}_{j}\theta^{\var}_{j}+\varepsilon^2\int I^{\var}_{1j}q^{\var}_{j}+\varepsilon^2\int\Delta_{j}(v^{\var}\cdot\nabla q^{\var})q^{\var}_{j} \\ \nonumber &\lesssim
\|\div w^{\var}_{j}\|_{L^2}\|\theta^{\var}_{j}\|_{L^2}+\|(H^{\var}_{j},\varepsilon I^{\var}_{1j})\|_{L^2}\|(\theta^{\var}_{j},\varepsilon q^{\var}_{j})\|_{L^2}+\|\varepsilon\Delta_{j}(v^{\var}\cdot\nabla q^{\var})\|_{L^2}\|\varepsilon q^{\var}_{j}\|_{L^2}.
\end{align}
\begin{multline}\label{R-E78}
\frac{d}{dt}\int 2^{-2j} q^{\var}_{j}\cdot\nabla\theta^{\var}_{j}+\frac{1}{\var^2}\|\theta^{\var}_{j}\|^2_{L^2}\\ \leq \|q^{\var}_{j}\|^2_{L^2}+\frac{2^{-2j}}{\var^2}\|\theta^{\var}_{j}\|_{L^2}\|\nabla\theta^{\var}_{j}\|_{L^2}+\|w^{\var}_{j}\|_{L^2}\|q^{\var}_{j}\|_{L^2}
+\|H^{\var}_{j}\|_{L^2}\|\var q^{\var}_{j}\|_{L^2}+\|\var I^{\var}_{1j}\|_{L^2}\|\theta^{\var}_{j}\|_{L^2}.
\end{multline}
Combining \eqref{R-E77} \eqref{R-E78}, we get
\begin{multline}\label{R-E79}
\frac{d}{dt}\mathcal{L}^{h}_{j}+\frac{1}{\var^2}\|(\theta^{\var}_{j},\varepsilon q^{\var}_{j})\|^2_{L^2}\lesssim \|\div w^{\var}_{j}\|_{L^2}\|\theta^{\var}_{j}\|_{L^2}+\|(H^{\var}_{j},\varepsilon I^{\var}_{1j})\|_{L^2}\|(\theta^{\var}_{j},\varepsilon q^{\var}_{j})\|_{L^2}+\|\varepsilon\Delta_{j}(v^{\var}\cdot\nabla q^{\var})\|_{L^2}.
\end{multline}
Using that $\mathcal{L}^{h}_{j}\approx \|(\theta^{\var}_{j},\varepsilon q_{j}^{\var})\|^2_{L^2}$ leads to
\begin{eqnarray}\label{R-E80}
\frac{d}{dt}\|(\var\theta^{\var}_{j},\varepsilon^2 q_{j}^{\var})\|_{L^2}+\frac{1}{\var}\|(\theta^{\var}_{j},\varepsilon q_{j}^{\var})\|_{L^2}\lesssim \var
\|\div w_{j}^{\var}\|_{L^2}+\|(\var H^{\var}_{j},\varepsilon^2 I^{\var}_{1j})\|_{L^2}.
\end{eqnarray}
Therefore, we obtain
\begin{align}\label{R-E81}
\frac{d}{dt}\Big(\|\varepsilon a^{\var}\|^{h}_{\dot{B}^{d/2+1}_{2,1}}+\|\varepsilon w^{\var}\|^{h}_{\dot{B}^{d/2}_{2,1}}&+\|\varepsilon^2 \theta^{\var}\|^{h}_{\dot{B}^{d/2+1}_{2,1}}+\|\varepsilon^3q^{\var}\|^{h}_{\dot{B}^{d/2+1}_{2,1}}\Big)\\&+\|\varepsilon a^{\var}\|^{h}_{\dot{B}^{d/2+1}_{2,1}}+\|\varepsilon w^{\var}\|^{h}_{\dot{B}^{d/2+2}_{2,1}}+\|(\theta^{\var},\var q^{\var})\|^{h}_{\dot{B}^{d/2+1}_{2,1}} \\ &\lesssim \varepsilon\|G^{\var}\|^{h}_{\dot{B}^{d/2}_{2,1}}+\varepsilon\|F^{\var}_1\|^{h}_{\dot{B}^{d/2-1}_{2,1}}+\varepsilon\|\nabla v^{\var}\|_{\dot{B}^{d/2}_{2,1}}\|a^{\var}\|^{h}_{\dot{B}^{d/2+1}_{2,1}}\\&+\varepsilon\|\nabla v^{\var}\|_{\dot{B}^{d/2}_{2,1}}\|a^{\var}\|_{\dot{B}^{d/2}_{2,1}}+\|(\var^2 H^{\var},\varepsilon^3 I^{\var}_{1})\|^{h}_{\dot{B}^{d/2+1}_{2,1}}.
\end{align}
\subsection{Nonlinear analysis}
In this second step, we bound the nonlinear terms. First, we prove the estimates for $\|(F^\varepsilon,G^\varepsilon,H^\varepsilon,\varepsilon
I^{\varepsilon})\|^{\ell}_{\dot{B}^{\frac {d}{2}-1}_{2,1}}$. We have
\begin{align}\label{R-E82}
\|v^{\var}\cdot&\nabla q^{\var}\|^{\ell}_{\dot{B}^{d/2-1}_{2,1}}\\&\lesssim (\|v^{\var}\|^{\ell}_{\dot{B}^{d/2-1}_{2,1}}+\|v^{\var}\|^{m}_{\dot{B}^{d/p}_{2,1}\cap\dot{B}^{d/p-1}_{p,1}}+\|v^{\var}\|^{h}_{\dot{B}^{d/2}_{2,1}})
(\|q^{\var}\|^{\ell}_{\dot{B}^{d/2}_{2,1}}+\frac{1}{\var}\|q^{\var}\|^{m}_{\dot{B}^{d/p}_{p,1}}+\|q^{\var}\|^{h}_{\dot{B}^{d/2}_{2,1}})\\
&\lesssim X^{\var}(t)(\|q^{\var}\|^{\ell}_{\dot{B}^{d/2}_{2,1}}+\frac{1}{\var}\|q^{\var}\|^{m}_{\dot{B}^{d/p-1}_{p,1}}+\|q^{\var}\|^{h}_{\dot{B}^{d/2}_{2,1}})
\end{align}
and
\begin{eqnarray}\label{R-E83}
\|q^{\var}\cdot\nabla v^{\var}\|^{\ell}_{\dot{B}^{d/2-1}_{2,1}}\lesssim \|\nabla v^{\var}\|_{\dot{B}^{d/p-1}_{p,1}}\|q^{\var}\|_{\dot{B}^{d/p}_{p,1}}+
\|q^{\var}\|_{\dot{B}^{d/p}_{p,1}}\|v^{\ell,\var}\|_{\dot{B}^{d/2-1}_{2,1}}+\|q^{\var}\|_{\dot{B}^{d/p}_{p,1}}\|v^{\tilde{h},\var}\|_{\dot{B}^{d/p}_{p,1}}.
\end{eqnarray}
The term $q^{\var}\div v^{\var}$ can be treated similarly as in \eqref{R-E83}. Next we bound $F^{\var}$.
\begin{multline}\label{R-E84}
\|F^{\var}\|^{\ell}_{\dot{B}^{d/2-1}_{2,1}}\lesssim \|a^{\var}v^{\var}\|^{\ell}_{\dot{B}^{d/2}_{2,1}}\lesssim \|a^{\var}\|_{\dot{B}^{d/p-1}_{p,1}}\|v^{\var}\|_{\dot{B}^{d/p+1}_{p,1}}+\|v^{\var}\|_{\dot{B}^{d/p-1}_{p,1}}(\|a^{\ell,\varepsilon}\|_{\dot{B}^{d/p+1}_{p,1}}+
\|a^{\tilde{h},\var}\|_{\dot{B}^{d/p}_{p,1}}).
\end{multline}
Regarding $G^\varepsilon$, we have
\begin{eqnarray}\label{R-E85}
\|v^{\var}\cdot\nabla v^{\var}\|^{\ell}_{\dot{B}^{d/2-1}_{2,1}}\lesssim \|v^{\var}\|_{\dot{B}^{d/p-1}_{p,1}}\|\nabla v^{\var}\|_{\dot{B}^{d/p}_{p,1}}+
\|v^{\var}\|_{\dot{B}^{d/p-1}_{p,1}}(\|\nabla v^{\ell,\var}\|_{\dot{B}^{d/p}_{p,1}}+\|\nabla v^{\tilde{h},\var}\|_{\dot{B}^{d/p-1}_{p,1}}),
\end{eqnarray}
\begin{eqnarray}\label{R-E86}
\|J(a^{\var})\mathcal{A}v^{\var}\|^{\ell}_{\dot{B}^{d/2-1}_{2,1}}\lesssim\|a^{\var}\|_{\dot{B}^{d/p}_{p,1}}\|\nabla v^{\var}\|_{\dot{B}^{d/p}_{p,1}}+
\|a^{\var}\|_{\dot{B}^{d/p}_{p,1}}(\|\mathcal{A}v^{\ell,\var}\|_{\dot{B}^{d/2-1}_{2,1}}+\|\mathcal{A}v^{\tilde{h},\var}\|_{\dot{B}^{d/p-1}_{p,1}}),
\end{eqnarray}
\begin{eqnarray}\label{R-E87}
\|H_{1}(a^{\var})\nabla a^{\var}\|^{\ell}_{\dot{B}^{d/2-1}_{2,1}}\nonumber\lesssim\|a^{\ell,\var}\|_{\dot{B}^{d/2-1}_{p,1}}\|a^{\ell,\var}\|_{\dot{B}^{d/2+1}_{p,1}}+\|a^{\tilde{h},\var}\|^2_{\dot{B}^{d/p}_{p,1}}+\|a^{\var}\|_{\dot{B}^{d/p-1}_{p,1}}
(\|a^{\var}\|^{\ell}_{\dot{B}^{d/2+1}_{2,1}}+\|a^{\var}\|^{\tilde{h}}_{\dot{B}^{d/p}_{p,1}}),
\end{eqnarray}
and
\begin{eqnarray}\label{R-E88}
\|H_2(a^{\var})\nabla \theta^{\var}\|^{\ell}_{\dot{B}^{d/2-1}_{2,1}}\nonumber\lesssim \|\nabla \theta^{\var}\|_{\dot{B}^{d/p-1}_{p,1}}\|a^{\var}\|_{\dot{B}^{d/p}_{p,1}}+\|a^{\var}\|_{\dot{B}^{d/p-1}_{p,1}}(\|\theta^{\ell,\var}\|_{\dot{B}^{d/2+1}_{2,1}}
+\|\theta^{m,\var}\|_{\dot{B}^{d/p}_{p,1}}+\frac 1\var \|\theta^{h,\var}\|_{\dot{B}^{d/2}_{2,1}}).
\end{eqnarray}
Then, using an interpolation inequality, we get
\begin{eqnarray*}\|\nabla \theta^{\var}\|_{\dot{B}^{d/p-1}_{p,1}}\|a^{\var}\|_{\dot{B}^{d/p}_{p,1}}&\lesssim & \|(a^{\var},\theta^{\var})\|^{\ell}_{\dot{B}^{d/2-1}_{2,1}}\|(a^{\var},\theta^{\var})\|^{\ell}_{\dot{B}^{d/2+1}_{2,1}}
+(\|a^{\var}\|^{m}_{\dot{B}^{d/p}_{p,1}}+\|\theta^{\var}\|^{m}_{\dot{B}^{d/p-1}_{p,1}})\\&&+\|(a^{\var},\nabla\theta^{\var})\|^{m}_{\dot{B}^{d/p}_{p,1}}+
\|(\varepsilon a^{\var},\varepsilon^2\theta^{\var})\|^{h}_{\dot{B}^{d/2+1}_{2,1}}\|(\varepsilon a^{\var},\theta^{\var})\|^{h}_{\dot{B}^{d/2+1}_{2,1}}.
\end{eqnarray*}
Moreover, we have
\begin{eqnarray}\label{R-E89}
\|\theta\nabla H_{3}(a^{\var})\|^{\ell}_{\dot{B}^{d/2-1}_{2,1}}\lesssim \|\theta^{\var}\|_{\dot{B}^{d/p}_{p,1}}\|a^{\var}\|_{\dot{B}^{d/p}_{p,1}}+\|\theta^{\var}\|_{\dot{B}^{d/p-1}_{p,1}}(\|a^{\var}\|^{\ell}_{\dot{B}^{d/2+1}_{2,1}}+\|a^{\var}\|^{\tilde{h}}_{\dot{B}^{d/p}_{p,1}}).
\end{eqnarray}
Regarding $H^{\var}$, we have 
\begin{align}\label{R-E90}
\|v^{\var}\cdot\nabla\theta^{\var}\|^{\ell}_{\dot{B}^{d/2-1}_{2,1}}&\lesssim \|\nabla \theta^{\var}\|_{\dot{B}^{d/p-1}_{p,1}}\|v^{\var}\|_{\dot{B}^{d/p}_{p,1}}\\&+\|v^{\var}\|_{\dot{B}^{d/p-1}_{p,1}}(\|\theta^{\ell,\var}\|_{\dot{B}^{d/2+1}_{2,1}}
+\|\theta^{m,\var}\|_{\dot{B}^{d/p}_{p,1}}+\frac 1\var \|\theta^{h,\var}\|_{\dot{B}^{d/2}_{2,1}}),
\end{align}
\begin{align}\label{R-E91}
\|J(a^{\var})\div q^{\var}\|^{\ell}_{\dot{B}^{d/2-1}_{2,1}}&\lesssim \|q^{\var}\|_{\dot{B}^{d/p}_{p,1}}\|a^{\var}\|_{\dot{B}^{d/p}_{p,1}}\\&+\|a^{\var}\|_{\dot{B}^{d/p}_{p,1}}\|q^{\var}\|_{\dot{B}^{d/2}_{2,1}}
+\|a^{\var}\|_{\dot{B}^{d/p-1}_{p,1}}(\|q^{\var}\|^{m}_{\dot{B}^{d/p}_{p,1}}+\|\var q^{\var}\|^{h}_{\dot{B}^{d/2+1}_{p,1}}),
\end{align}
and
\begin{eqnarray}\label{R-E910}
\|\frac{N(\nabla v^{\var},\nabla v^{\var})}{1+a^{\var}}\|^{\ell}_{\dot{B}^{d/2-1}_{2,1}}=\|(1+J(a^{\var}))N(\nabla v^{\var},\nabla v^{\var})\|^{\ell}_{\dot{B}^{d/2-1}_{2,1}}&\lesssim &\|(1+J(a^{\var}))N(\nabla v^{\var},\nabla v^{\var})\|^{\ell}_{\dot{B}^{d/2-2}_{2,1}}\nonumber\\&\lesssim & (1+\|a^{\var}\|_{\dot{B}^{d/p}_{p,1}})\|\nabla v^{\var}\|_{\dot{B}^{d/p^{*}-1}_{p^{*},1}}
\|\nabla v^{\var}\|_{\dot{B}^{d/p-1}_{p,1}}\nonumber\\&\lesssim & (1+\|a^{\var}\|_{\dot{B}^{d/p}_{p,1}})\|v^{\var}\|^2_{\dot{B}^{d/p}_{p,1}},
\end{eqnarray}
where $1/p^{*}+1/p=1/2$. Then,
\begin{multline}\label{R-E93}
\|H_1(a^{\var})\theta^{\var}\div v^{\var}\|^{\ell}_{\dot{B}^{d/2-1}_{2,1}}\lesssim \|a^{\var}\|_{\dot{B}^{d/p}_{p,1}} \|v^{\var}\|_{\dot{B}^{d/p}_{p,1}}  \|\theta^{\var}\|_{\dot{B}^{d/p}_{p,1}\cap\dot{B}^{d/p-1}_{p,1}}+\|\theta^{\var}\|_{\dot{B}^{d/p}_{p,1}}\|H_1(a^{\var})\div v^{\var}\|^{\ell}_{\dot{B}^{d/2-1}_{2,1}},
\end{multline}
and
\begin{eqnarray*}
\|H_1(a^{\var})\div v^{\var}\|^{\ell}_{\dot{B}^{d/2-1}_{2,1}}\|\div v^{\var}\|_{\dot{B}^{d/p}_{p,1}}\|a^{\var}\|_{\dot{B}^{d/p}_{p,1}}+\|a^{\var}\|_{\dot{B}^{d/p-1}_{p,1}}(\|v^{\var}\|^{m}_{\dot{B}^{d/p}_{p,1}}
+\varepsilon\|v^{\var}\|^{h}_{\dot{B}^{d/2+1}_{2,1}}).
\end{eqnarray*}

We now provide the estimates for $\|F^{\var}_{1}\|^{m}_{\dot{B}^{d/p-2}_{p,1}}+\|G^{\var}\|^{m}_{\dot{B}^{d/p-1}_{p,1}}+\varepsilon\|I^{\varepsilon}\|^{m}_{\dot{B}^{d/p-2}_{p,1}}+\|H^{\varepsilon}\|^{m}_{\dot{B}^{d/p-2}_{p,1}}$. We have
\begin{eqnarray}\label{R-E94}
\|F^{\var}_{1}\|^{m}_{\dot{B}^{d/p-2}_{p,1}}\lesssim \|F^{\var}_{1}\|^{m}_{\dot{B}^{d/p}_{p,1}}\lesssim \|a^{\var}\|_{\dot{B}^{d/p}_{p,1}}\|\div v^{\var}\|_{\dot{B}^{d/p}_{p,1}}.
\end{eqnarray}
For $\|G^{\var}\|^{m}_{\dot{B}^{d/p-1}_{p,1}}$, we have
\begin{eqnarray}\label{R-E95}
\|v^{\var}\cdot\nabla v^{\var}\|^{m}_{\dot{B}^{d/p-1}_{p,1}}\lesssim \|v^{\var}\|^2_{\dot{B}^{d/p}_{p,1}},
\end{eqnarray}
\begin{eqnarray}\label{R-E96}
\|J(a^{\var})\mathcal{A}v^{\var}\|^{m}_{\dot{B}^{d/p-1}_{p,1}}\lesssim \|a^{\var}\|_{\dot{B}^{d/p}_{p,1}}\|v^{\var}\|_{\dot{B}^{d/p+1}_{p,1}},
\end{eqnarray}
\begin{eqnarray}\label{R-E97}
\|H_{1}(a^{\var})\nabla a^{\var}\|^{m}_{\dot{B}^{d/p-1}_{p,1}}\lesssim \|a^{\var}\|^2_{\dot{B}^{d/p}_{p,1}},
\end{eqnarray}
\begin{eqnarray}\label{R-E98}
\|H_{2}(a^{\var})\nabla \theta^{\var}\|^{m}_{\dot{B}^{d/p-1}_{p,1}}\lesssim \|a^{\var}\|_{\dot{B}^{d/p}_{p,1}}\|\theta^{\var}\|_{\dot{B}^{d/p}_{p,1}},
\end{eqnarray}
\begin{eqnarray}\label{R-E99}
\|\theta^{\var}\nabla H_{3}(a^{\var})\|^{m}_{\dot{B}^{d/p-1}_{p,1}}\lesssim \|\theta^{\var}\|_{\dot{B}^{d/p}_{p,1}}\|a^{\var}\|_{\dot{B}^{d/p}_{p,1}},
\end{eqnarray}
Using Proposition \ref{prop:Bernstein} gives
\begin{eqnarray}\label{R-E100}
\varepsilon\|I^{\varepsilon}\|^{m}_{\dot{B}^{d/p-2}_{p,1}}\lesssim \var \|I^{\varepsilon}\|^{m}_{\dot{B}^{d/p-1}_{p,1}}.
\end{eqnarray}
We have
\begin{eqnarray}\label{R-E101}
\|v^{\var}\cdot \nabla q^{\var}\|^{m}_{\dot{B}^{d/p-1}_{p,1}}\lesssim \|v^{\var}\|_{\dot{B}^{d/p-1}_{p,1}}\|q^{\var}\|_{\dot{B}^{d/p}_{p,1}},
\end{eqnarray}
\begin{eqnarray}\label{R-E102}
\varepsilon\|q^{\var}\cdot \nabla v^{\var}\|^{m}_{\dot{B}^{d/p-1}_{p,1}}\lesssim (\|\var q^{\ell,\var}\|_{\dot{B}^{d/2-1}_{2,1}}+\|\var q^{m,\var}\|_{\dot{B}^{d/p-1}_{p,1}}+\|\var^3q^{h,\var}\|_{\dot{B}^{d/2+1}_{2,1}})\|v^{\var}\|_{\dot{B}^{d/p+1}_{p,1}},
\end{eqnarray}
and
\begin{eqnarray}\label{R-E103}
\varepsilon\|q^{\var}\div v^{\var}\|^{m}_{\dot{B}^{d/p-1}_{p,1}}\lesssim (\|\var q^{\ell,\var}\|_{\dot{B}^{d/2-1}_{2,1}}+\|\var q^{m,\var}\|_{\dot{B}^{d/p-1}_{p,1}}+\|\var^3q^{h,\var}\|_{\dot{B}^{d/2+1}_{2,1}})\|v^{\var}\|_{\dot{B}^{d/p+1}_{p,1}}.
\end{eqnarray}
Employing Proposition \ref{prop:Bernstein}, we get
\begin{eqnarray}\label{R-E104}
\|H^{\varepsilon}\|^{m}_{\dot{B}^{d/p-2}_{p,1}}\lesssim \|H^{\varepsilon}\|^{m}_{\dot{B}^{d/p-1}_{p,1}}.
\end{eqnarray}
Furthermore, we have
\begin{eqnarray}\label{R-E105}
\|v^{\var}\cdot\nabla \theta^{\var}\|^{m}_{\dot{B}^{d/p-1}_{p,1}}\lesssim \|v^{\var}\|_{\dot{B}^{d/p-1}_{p,1}}\|\theta^{\var}\|_{\dot{B}^{d/p+1}_{p,1}},
\end{eqnarray}
\begin{eqnarray}\label{R-E106}
\|J(a^{\var})\div q^{\var}\|^{m}_{\dot{B}^{d/p-1}_{p,1}}\lesssim \|a^{\var}\|_{\dot{B}^{d/p}_{p,1}}\|q^{\var}\|_{\dot{B}^{d/p}_{p,1}},
\end{eqnarray}
\begin{eqnarray}\label{R-E107}
\|\frac{N(\nabla v^{\var},\nabla v^{\var})}{1+a^{\var}}\|^{m}_{\dot{B}^{d/p-1}_{p,1}}\lesssim (1+\|a^{\var}\|_{\dot{B}^{d/p}_{p,1}}) \|\nabla v^{\var}\|_{\dot{B}^{d/p-1}_{p,1}}\|\nabla v^{\var}\|_{\dot{B}^{d/p}_{p,1}}
\end{eqnarray}
and
\begin{align}\label{R-E108}
\|H_1(a^{\var})\theta^{\var}\div v^{\var}\|^{m}_{\dot{B}^{d/p-1}_{p,1}}&\lesssim \|a^{\var}\|_{\dot{B}^{d/p}_{p,1}}\|\theta^{\var}\|_{\dot{B}^{d/p-1}_{p,1}}\|v^{\var}\|_{\dot{B}^{d/p+1}_{p,1}}\\&\lesssim \|a^{\var}\|_{\dot{B}^{d/p}_{p,1}}(\|\theta^{\var}\|^{\ell}_{\dot{B}^{d/2-1}_{2,1}}+\|\theta^{\var}\|^{m}_{\dot{B}^{d/p-1}_{p,1}}
+\|\var^2\theta^{\var}\|^{h}_{\dot{B}^{d/2+1}_{2,1}})\|v^{\var}\|_{\dot{B}^{d/p+1}_{p,1}}.
\end{align}
Finally, we provide the estimates for $\varepsilon\|G^{\var}\|^{h}_{\dot{B}^{d/2}_{2,1}}+
\varepsilon\|F^{\var}_1\|^{h}_{\dot{B}^{d/2-1}_{2,1}}+\|(\var^2 H^{\var},\varepsilon^3 I^{\var}_{1})\|^{h}_{\dot{B}^{d/2+1}_{2,1}}
$. First, we pay attention to $\|\varepsilon^3 I^{\var}_{1}\|^{h}_{\dot{B}^{d/2+1}_{2,1}}$. It follows from Proposition \ref{LPP} that 
\begin{eqnarray}\label{R-E109}
\|q^{\var}\cdot \nabla v^{\var}\|^{h}_{\dot{B}^{d/2+1}_{2,1}}\lesssim \|q^{\var}\|_{\dot{B}^{d/p}_{p,1}}\|\nabla v^{\var}\|^{h}_{\dot{B}^{d/2+1}_{2,1}}+\|\nabla v^{\var}\|_{\dot{B}^{d/p}_{p,1}}\|q^{\var}\|^{h}_{\dot{B}^{d/2+1}_{2,1}}+\|q^{\var}\|^{\tilde{\ell}}_{\dot{B}^{d/p}_{p,1}}\|\nabla v^{\var}\|^{\tilde{\ell}}_{\dot{B}^{d/p+1}_{p,1}},
\end{eqnarray}
which leads to
\begin{eqnarray}\label{R-E110}
\varepsilon^3\|q^{\var}\cdot \nabla v^{\var}\|^{h}_{\dot{B}^{d/2+1}_{2,1}}\lesssim X^{\varepsilon}(t)(\|\varepsilon v^{\var}\|^{h}_{\dot{B}^{d/2+2}_{2,1}}
+\|\varepsilon q^{\var}\|^{h}_{\dot{B}^{d/2+1}_{2,1}}+\|v^{\var}\|^{\ell}_{\dot{B}^{d/2+1}_{2,1}}+\|v^{\var}\|^{m}_{\dot{B}^{d/p+1}_{p,1}}).
\end{eqnarray}
Bounding the term $\varepsilon^3q^{\var}\div v^{\var}$ is similar to bounding $\varepsilon^3q^{\var}\cdot \nabla v^{\var}$. We obtain
\begin{eqnarray}\label{R-E111}
\|F^{\var}_1\|^{h}_{\dot{B}^{d/2}_{2,1}}&\lesssim& \|a^{\var}\|_{\dot{B}^{d/p}_{p,1}}\|\div v^{\var}\|^{h}_{\dot{B}^{d/2}_{2,1}}+\|\div v^{\var}\|_{\dot{B}^{d/p}_{p,1}}\|a^{\var}\|^{h}_{\dot{B}^{d/2}_{2,1}}+\|a^{\var}\|^{\tilde{\ell}}_{\dot{B}^{d/p}_{p,1}}
\|\div v^{\var}\|^{\tilde{\ell}}_{\dot{B}^{d/p}_{p,1}}\nonumber\\ &\lesssim& X^{\varepsilon}(t)(\|v^{\var}\|^{\ell}_{\dot{B}^{d/2+1}_{2,1}}+\|v^{\var}\|^{m}_{\dot{B}^{d/p+1}_{p,1}}+\varepsilon\|v^{\var}\|^{h}_{\dot{B}^{d/2+2}_{2,1}}).
\end{eqnarray}
Next, we handle $\|\var^2 H^{\var}\|^{h}_{\dot{B}^{d/2+1}_{2,1}}$. We claim that
$\|f(a^{\var})\|^{h}_{\dot{B}^{d/2}_{2,1}}\lesssim (1+X^{\varepsilon}(t))X^{\varepsilon}(t)$ for some smooth function satisfying $f(0)=0$. Indeed, it follows from \eqref{prop:comphf} that
\begin{eqnarray}\label{R-E112}
\|f(a^{\var})\|^{h}_{\dot{B}^{d/2}_{2,1}}\lesssim (1+\|a^{\var}\|^{\tilde{\ell}}_{\dot{B}^{d/p}_{p,1}}+\var\|a^{\var}\|^{h}_{\dot{B}^{d/2}_{2,1}})(\|a^{\var}\|^{\tilde{\ell}}_{\dot{B}^{d/p}_{p,1}}
+\|a^{\var}\|^{h}_{\dot{B}^{d/2}_{2,1}})\lesssim (1+X^{\varepsilon}(t))X^{\varepsilon}(t).
\end{eqnarray}
Therefore, we have
\begin{multline}\label{R-E113}
\|J(a^{\var})\div q^{\var}\|^{h}_{\dot{B}^{d/2}_{2,1}}\lesssim \|a^{\var}\|_{\dot{B}^{d/p}_{p,1}}\|\div q^{\var}\|^{h}_{\dot{B}^{d/2}_{2,1}}+\|\div q^{\var}\|_{\dot{B}^{d/p}_{p,1}}\|J(a^{\var})\|^{h}_{\dot{B}^{d/2}_{2,1}}+\|J(a^{\var})\|^{\tilde{\ell}}_{\dot{B}^{d/p}_{p,1}}\|\div q^{\var}\|^{\tilde{\ell}}_{\dot{B}^{d/p}_{p,1}}
\end{multline}
and
\begin{eqnarray}\label{R-E114}
\|\frac{N(\nabla v^{\var},\nabla v^{\var})}{1+a^{\var}}\|^{h}_{\dot{B}^{d/2}_{2,1}}\lesssim (1+\|a^{\var}\|_{\dot{B}^{d/p}_{p,1}})\|v^{\var}\|_{\dot{B}^{d/p+1}_{p,1}}\|\nabla v^{\var}\|_{\dot{B}^{d/2}_{2,1}}+\|\nabla v^{\var}\|_{\dot{B}^{d/p}_{p,1}}\|J(a^{\var})\nabla v^{\var}\|^{h}_{\dot{B}^{d/2}_{2,1}},
\end{eqnarray}
where 
$$\|J(a^{\var})\nabla v^{\var}\|^{h}_{\dot{B}^{d/2}_{2,1}}\lesssim \|a^{\var}\|_{\dot{B}^{d/p}_{p,1}}\|\nabla v^{\var}\|^{h}_{\dot{B}^{d/2}_{2,1}}+
\|\nabla v^{\var}\|_{\dot{B}^{d/p}_{p,1}}\|a^{\var}\|_{\dot{B}^{d/p}_{p,1}}.$$
Similarly, we have
\begin{multline}\label{R-E115}
\|H_1(a^{\var})\theta^{\var}\div v^{\var}\|^{h}_{\dot{B}^{d/2}_{2,1}}\lesssim \|a^{\var}\|_{\dot{B}^{d/p}_{p,1}}\|\theta^{\var}\|_{\dot{B}^{d/p}_{p,1}}
(\|v^{\var}\|^{\ell}_{\dot{B}^{d/2+1}_{2,1}}+\|v^{\var}\|^{m}_{\dot{B}^{d/p+1}_{p,1}}+\varepsilon\|v^{\var}\|^{h}_{\dot{B}^{d/2+2}_{2,1}})\\+\|\nabla v^{\var}\|_{\dot{B}^{d/p}_{p,1}}\|H_1(a^{\var})\theta^{\var}\|^{h}_{\dot{B}^{d/2}_{2,1}},
\end{multline}
where
$$\|H_1(a^{\var})\theta^{\var}\|^{h}_{\dot{B}^{d/2}_{2,1}}\lesssim \|a^{\var}\|_{\dot{B}^{d/p}_{p,1}} \|\theta^{\var}\|^{h}_{\dot{B}^{d/2}_{2,1}}+\|\theta^{\var}\|_{\dot{B}^{d/p}_{p,1}}\|H_1(a^{\var})\|^{h}_{\dot{B}^{d/2}_{2,1}}
+\|a^{\var}\|^{\tilde{\ell}}_{\dot{B}^{d/p}_{p,1}} \|\theta^{\var}\|^{\tilde{\ell}}_{\dot{B}^{d/p}_{p,1}}.$$
Finally, we bound the nonlinear term $G^{\var}$. Precisely,
\begin{eqnarray}\label{R-E116}
\|v^{\var}\cdot\nabla v^{\var}\|^{h}_{\dot{B}^{d/2}_{2,1}}&\lesssim& \|v^{\var}\|_{\dot{B}^{d/p}_{p,1}} \|\nabla v^{\var}\|^{h}_{\dot{B}^{d/2}_{2,1}}+ \|\nabla v^{\var}\|_{\dot{B}^{d/p}_{p,1}} \|v^{\var}\|^{h}_{\dot{B}^{d/2}_{2,1}}+ \|v^{\var}\|^{\tilde{\ell}}_{\dot{B}^{d/p}_{p,1}}\|\nabla v^{\var}\|^{\tilde{\ell}}_{\dot{B}^{d/p}_{p,1}}\nonumber\\ &\lesssim&
\|v^{\var}\|_{\dot{B}^{d/p}_{p,1}} \|\nabla v^{\var}\|^{h}_{\dot{B}^{d/2}_{2,1}}+(\|v^{\var}\|^{\ell}_{\dot{B}^{d/2-1}_{2,1}}+\frac{1}{\varepsilon}\|v^{\var}\|^{m}_{\dot{B}^{d/p-1}_{p,1}}
+\varepsilon\|v^{\var}\|^{h}_{\dot{B}^{d/2+1}_{2,1}})\|v^{\var}\|_{\dot{B}^{d/p+1}_{p,1}},
\end{eqnarray}
\begin{eqnarray}\label{R-E117}
\|J(a^{\var})\mathcal{A}v^{\var}\|^{h}_{\dot{B}^{d/2}_{2,1}}&\lesssim& \|a^{\var}\|_{\dot{B}^{d/p}_{p,1}}\|\mathcal{A} v^{\var}\|^{h}_{\dot{B}^{d/2}_{2,1}}+\|\mathcal{A} v^{\var}\|_{\dot{B}^{d/p}_{p,1}}\|J(a^{\var})\|^{h}_{\dot{B}^{d/2}_{2,1}}+
\|J(a^{\var})\|^{\tilde{\ell}}_{\dot{B}^{d/p}_{p,1}}\|\mathcal{A} v^{\var}\|^{\tilde{\ell}}_{\dot{B}^{d/p}_{p,1}},
\end{eqnarray}
\begin{multline}\label{R-E118}
\|H_{1}(a^{\var})\nabla a^{\var}\|^{h}_{\dot{B}^{d/2}_{2,1}}\lesssim \|a^{\var}\|_{\dot{B}^{d/p}_{p,1}} \|\nabla a^{\var}\|^{h}_{\dot{B}^{d/2}_{2,1}}+\|\nabla a^{\var}\|^{h}_{\dot{B}^{d/p}_{p,1}}\|H_{1}(a^{\var})\|^{h}_{\dot{B}^{d/2}_{2,1}}+\|\nabla a^{\var}\|^{\tilde{\ell}}_{\dot{B}^{d/p}_{p,1}}\|H_{1}(a^{\var})\|^{\tilde{\ell}}_{\dot{B}^{d/p}_{p,1}},
\end{multline}
\begin{multline}\label{R-E119}
\|H_{2}(a^{\var})\nabla \theta^{\var}\|^{h}_{\dot{B}^{d/2}_{2,1}}\lesssim \|a^{\var}\|_{\dot{B}^{d/p}_{p,1}} \|\nabla \theta^{\var}\|^{h}_{\dot{B}^{d/2}_{2,1}}+\|\nabla \theta^{\var}\|^{h}_{\dot{B}^{d/p}_{p,1}}\|H_{2}(a^{\var})\|^{h}_{\dot{B}^{d/2}_{2,1}}+\|\nabla \theta^{\var}\|^{\tilde{\ell}}_{\dot{B}^{d/p}_{p,1}}\|H_{1}(a^{\var})\|^{\tilde{\ell}}_{\dot{B}^{d/p}_{p,1}},
\end{multline}
and
\begin{align}\label{R-E120}
\|\theta^{\var}\nabla H_{3}(a^{\var})\|^{h}_{\dot{B}^{d/2}_{2,1}}&\lesssim \|\theta^{\var}\|_{\dot{B}^{d/p}_{p,1}} \|\nabla H_{3}(a^{\var})\|^{h}_{\dot{B}^{d/2}_{2,1}}+\|\nabla H_{3}(a^{\var})\|_{\dot{B}^{d/p-1}_{p,1}}\|\theta^{\var}\|^{h}_{\dot{B}^{d/2+1}_{2,1}}
\\&+\|\theta^{\var}\|^{\tilde{\ell}}_{\dot{B}^{d/p}_{p,1}}\|\nabla  H_{3}(a^{\var})\|^{\tilde{\ell}}_{\dot{B}^{d/p}_{p,1}}. \nonumber
\end{align}

\subsection{The regularity evolution of negative Besov norm}
In this section, we establish the regularity evolution of negative Besov norm in low frequencies, which is the key part in deriving the decay estimates. We have the following lemma.
\begin{Lemme}\label{LemmaDecay}Let 
 $(a^\var,v^\var,\theta^\var,q^\var)$ be the solution of \eqref{LinearHypNSCzero} given by Theorem \ref{thm:exist}. If $(a^{\var}_0,v^{\var}_0,\theta^{\var}_0,\var Q^{\var}_0)^{\ell}\in{\dot{B}^{-\sigma_1}_{2,\infty}}$, we have
\begin{align}\label{R-E1250}
\|(a^{\var},v^{\var},\theta^{\var},\var Q^{\var})\|^{\ell}_{\dot{B}^{-\sigma_1}_{2,\infty}}\lesssim \|(a^{\var}_0,v^{\var}_0,\theta^{\var}_0,\var Q^{\var}_0)\|^{\ell}_{\dot{B}^{-\sigma_1}_{2,\infty}}+X^\var(t)^2.
\end{align}
\end{Lemme}
\begin{proof}[Proof of Lemma \ref{LemmaDecay}]
Set $\omega^{\var}=\Lambda^{-1}\div v^{\var}$ and $\Omega^{\var}=\Lambda^{-1} \mathrm{curl}\,{v^{\var}}$.
The first three equations in \eqref{LinearHypNSC} can be written as
\begin{equation}
\left\{
\begin{array}
[c]{l}%
\d_ta^\varepsilon+\Lambda \omega^\varepsilon=F^{\var},\\
\d_t \omega^\varepsilon-\Delta\omega^\varepsilon-\Lambda a^\varepsilon-\Lambda\theta^\varepsilon=\Lambda^{-1}\div G^{\var},\\
\d_t \Omega^{\var}-\Delta \Omega^{\var}=\mathrm{curl} G^{\var},\\
\d_t \theta^\varepsilon-\Delta\theta^{\var}+\Lambda\omega^\varepsilon=-\div Q^{\var}+H^\varepsilon.\\
\end{array}
\right.
\label{R-E121}
\end{equation}
Energy estimates give
\begin{multline}\label{R-E122}
\frac{1}{2}\frac{d}{dt}(\|a^{\var}_j\|^2_{L^2}+\|\omega^{\var}_j\|^2_{L^2}+\|\Omega^{\var}_j\|^2_{L^2}+\|\theta^{\var}_j\|^2_{L^2})+\|\Lambda v^{\var}_j\|^2_{L^2}+\|\Lambda \theta^{\var}_j\|^2_{L^2}\\ \leq \|F^{\var}_{j}\|_{L^2}\|a^{\var}_j\|_{L^2}+\|\Lambda^{-1}\div G^{\var}_{j}\|_{L^2}\|\omega^{\var}_j\|_{L^2}+\|\Lambda^{-1}\mathrm{curl} G^{\var}_{j}\|_{L^2}\|\Omega^{\var}_j\|_{L^2}\\+\|-\div Q^{\var}_j+H^\varepsilon_j\|_{L^2}\|\theta^{\var}_j\|_{L^2}
\end{multline}
and
\begin{eqnarray}\label{R-E123}
\frac{1}{2}\frac{d}{dt}\|\var^2Q^{\var}_j\|^2_{L^2}+\|Q^{\var}_j\|^2_{L^2}\leq \|\var (f^{\var}_j+I^{\var}_j+\kappa \nabla H^{\var}_j)\|_{L^2}\|\var Q^{\var}_j\|_{L^2}.
\end{eqnarray}
It follows from \eqref{R-E122} and \eqref{R-E123} that
\begin{align}\label{R-E124}
\frac{d}{dt}(\|a^{\var}_j\|^2_{L^2}+\|\omega^{\var}_j\|^2_{L^2}+\|\Omega^{\var}_j\|^2_{L^2}&+\|\theta^{\var}_j\|^2_{L^2}+\|\var Q^{\var}_j\|^2_{L^2})\\ &\lesssim
(\|F^{\var}_{j}\|_{L^2}+\|\Lambda^{-1}\div G^{\var}_{j}\|_{L^2}+\|\Lambda^{-1}\mathrm{curl} G^{\var}_{j}\|_{L^2}+\|H^\varepsilon_j\|_{L^2}\nonumber \\&+\|\var (I^{\var}_j+\kappa \nabla H^{\var}_j)\|_{L^2})\|(a^{\var}_j,\omega^{\var}_j,\Omega^{\var}_j,\theta^{\var}_j,\var Q^{\var}_j)\|_{L^2}. \nonumber
\end{align}
A standard procedure leads to
\begin{multline}\label{R-E125}
\Big(\|(a^{\var},v^{\var},\theta^{\var},\var Q^{\var})\|^{\ell}_{\dot{B}^{-\sigma_1}_{2,\infty}}\Big)^2\lesssim \Big(\|(a^{\var}_0,v^{\var}_0,\theta^{\var}_0,\var Q^{\var}_0)\|^{\ell}_{\dot{B}^{-\sigma_1}_{2,\infty}}\Big)^2\\+\int^{t}_{0}\|(F^{\var},G^{\var},H^{\var},\varepsilon I^{\var})\|^{\ell}_{\dot{B}^{-\sigma_1}_{2,\infty}}\|(a^{\var},v^{\var},\theta^{\var},\var Q^{\var})\|^{\ell}_{\dot{B}^{-\sigma_1}_{2,\infty}}d\tau.
\end{multline}
In what follows, we focus on bounding the nonlinear term $\|(F^{\var},G^{\var},H^{\var},\varepsilon I^{\var})\|^{\ell}_{\dot{B}^{-\sigma_1}_{2,\infty}}$. It is convenient to decompose them into low-frequency and high-frequency parts. Precisely, 
$$F^{\ell,\var}=-a^{\var}\div v^{\ell,\var}-v^{\var}\cdot \nabla a^{\ell,\var}, \quad  F^{\tilde{h},\var}=-a^{\var}\div v^{\tilde{h},\var}-v^{\var}\cdot \nabla a^{\tilde{h},\var},$$
$$G^{\ell,\var}=-v^\varepsilon\cdot\nabla v^{\ell,\var} -J(a^\varepsilon)\mathcal{A}v^{\ell,\var}/\nu-H_1(a^\varepsilon)\nabla a^{\ell,\var}-H_2(a^\varepsilon)\nabla \theta^{\ell,\var}-\theta^{\ell,\var}\nabla H_3(a^\varepsilon),$$
$$G^{\tilde{h},\var}=-v^\varepsilon\cdot\nabla v^{\tilde{h},\var} -J(a^\varepsilon)\mathcal{A}v^{\tilde{h},\var}/\nu-H_1(a^\varepsilon)\nabla a^{\tilde{h},\var}-H_2(a^\varepsilon)\nabla \theta^{\tilde{h},\var}-\theta^{\tilde{h},\var}\nabla H_3(a^\varepsilon),$$
$$H^{\ell,\var}=-v^\varepsilon\cdot\nabla\theta^{\ell,\var}+J(a^\varepsilon)\div q^{\ell,\var}+\dfrac{N(\nabla v^\varepsilon,\nabla v^{\ell,\var})}{1+a^\varepsilon}-\wt H_1(a^\varepsilon)\theta^\varepsilon \div v^{\ell,\var},$$
$$H^{\tilde{h},\var}=-v^\varepsilon\cdot\nabla\theta^{\tilde{h},\var}+J(a^\varepsilon)\div q^{\tilde{h},\var}+\dfrac{N(\nabla v^\varepsilon,\nabla v^{\tilde{h},\var})}{1+a^\varepsilon}-\wt H_1(a^\varepsilon)\theta^\varepsilon \div v^{\tilde{h},\var},$$
$$I^{\ell,\var}=v^\varepsilon\cdot\nabla q^{\ell,\var}-q^\varepsilon\cdot\nabla v^{\ell,\var}+q^\varepsilon\div v^{\ell,\var},$$
$$I^{\tilde{h},\var}=v^\varepsilon\cdot\nabla q^{\tilde{h},\var}-q^\varepsilon\cdot\nabla v^{\tilde{h},\var}+q^\varepsilon\div v^{\tilde{h},\var}.$$
Using standard product laws, we have
\begin{equation}\label{R-E126}
\|a^{\ell,\var}\div v^{\ell,\var}\|^{\ell}_{\dot{B}^{-\sigma_1}_{2,\infty}}\lesssim \|\div v^{\ell,\var}\|_{\dot{B}^{d/p}_{p,1}}\|a^{\ell,\var}\|_{\dot{B}^{-\sigma_1}_{2,\infty}}\lesssim \|v^{\var}\|^{\ell}_{\dot{B}^{d/2+1}_{2,1}}\|a^{\var}\|^{\ell}_{\dot{B}^{-\sigma_1}_{2,\infty}},
\end{equation}
\begin{equation}\label{R-E127}
\|a^{\tilde{h},\var}\div v^{\ell,\var}\|^{\ell}_{\dot{B}^{-\sigma_1}_{2,\infty}}\lesssim \|a^{\tilde{h},\var}\|_{\dot{B}^{d/p}_{p,1}}\|\div v^{\ell,\var}\|_{\dot{B}^{-\sigma_1}_{2,\infty}}\lesssim (\|a^{\var}\|^{m}_{\dot{B}^{d/p}_{p,1}}+\|a^{\var}\|^{h}_{\dot{B}^{d/2}_{2,1}})\|v^{\var}\|^{\ell}_{\dot{B}^{-\sigma_1}_{2,\infty}},
\end{equation}
\begin{equation}\label{R-E128}
\|v^{\ell,\var}\cdot \nabla a^{\ell,\var}\|^{\ell}_{\dot{B}^{-\sigma_1}_{2,\infty}}\lesssim \|a^{\var}\|^{\ell}_{\dot{B}^{d/2+1}_{2,1}}\|v^{\var}\|^{\ell}_{\dot{B}^{-\sigma_1}_{2,\infty}}
\end{equation}
and
\begin{equation}\label{R-E129}
\|v^{\tilde{h},\var}\cdot \nabla a^{\ell,\var}\|^{\ell}_{\dot{B}^{-\sigma_1}_{2,\infty}}\lesssim (\|v^{\var}\|^{m}_{\dot{B}^{d/p+1}_{p,1}}+\varepsilon^2\|v^{\var}\|^{h}_{\dot{B}^{d/2+2}_{2,1}})\|a^{\var}\|^{\ell}_{\dot{B}^{-\sigma_1}_{2,\infty}}.
\end{equation}
Similarly, 
\begin{equation}\label{R-E130}
\|v^{\var}\cdot \nabla v^{\ell,\var}\|^{\ell}_{\dot{B}^{-\sigma_1}_{2,\infty}}\lesssim (\|v^{\var}\|^{\ell}_{\dot{B}^{d/2+1}_{2,1}}+\|v^{\var}\|^{m}_{\dot{B}^{d/p+1}_{p,1}}+\varepsilon^2\|v^{\var}\|^{h}_{\dot{B}^{d/2+2}_{2,1}})\|v^{\var}\|^{\ell}_{\dot{B}^{-\sigma_1}_{2,\infty}}.
\end{equation}
Since $J(a^{\var})=J'(0)a^{\var}+\tilde{J}(a^{\var})$ for some smooth function satisfying $\tilde{J}=0$. It follows that
\begin{equation}\label{R-E131}
\|a^{\ell,\var}\mathcal{A}v^{\var}\|^{\ell}_{\dot{B}^{-\sigma_1}_{2,\infty}}\lesssim \|v^{\var}\|^{\ell}_{\dot{B}^{d/2+1}_{2,1}}\|a^{\var}\|^{\ell}_{\dot{B}^{-\sigma_1}_{2,\infty}},
\end{equation}
\begin{equation}\label{R-E132}
\|a^{\tilde{h},\var}\mathcal{A}v^{\ell,\var}\|^{\ell}_{\dot{B}^{-\sigma_1}_{2,\infty}}\lesssim
\|a^{\tilde{h},\var}\|_{\dot{B}^{d/2+1}_{2,1}}\|\mathcal{A}v^{\ell,\var}\|_{\dot{B}^{-\sigma_1}_{2,\infty}}\lesssim (\|a^{\var}\|^{m}_{\dot{B}^{d/p}_{p,1}}+\|a^{\var}\|^{h}_{\dot{B}^{d/2}_{2,1}})\|v^{\var}\|^{\ell}_{\dot{B}^{-\sigma_1}_{2,\infty}},
\end{equation}
\begin{equation}\label{R-E133}
\|\tilde{J}(a^{\var})a^{\var}\mathcal{A}v^{\ell,\var}\|^{\ell}_{\dot{B}^{-\sigma_1}_{2,\infty}}\lesssim\|\tilde{J}(a^{\var})a^{\var}\|_{\dot{B}^{d/p}_{p,1}}
\|\mathcal{A}v^{\ell,\var}\|_{\dot{B}^{-\sigma_1}_{2,\infty}}\lesssim\|a^{\var}\|^2_{\dot{B}^{d/p}_{p,1}}\|v^{\var}\|^{\ell}_{\dot{B}^{-\sigma_1}_{2,\infty}},
\end{equation}
\begin{equation}\label{R-E134}
\|H_1(a^\varepsilon)\nabla a^{\ell,\var}\|^{\ell}_{\dot{B}^{-\sigma_1}_{2,\infty}}\lesssim (\|a^{\var}\|^{\ell}_{\dot{B}^{d/2+1}_{2,1}}+\|a^{\var}\|^{m}_{\dot{B}^{d/p}_{p,1}}+\|a^{\var}\|^{h}_{\dot{B}^{d/2}_{2,1}})\|a^{\var}\|^{\ell}_{\dot{B}^{-\sigma_1}_{2,\infty}}
\end{equation}
and
\begin{equation}\label{R-E135}
\|H_2(a^\varepsilon)\nabla \theta^{\ell,\var}\|^{\ell}_{\dot{B}^{-\sigma_1}_{2,\infty}}\lesssim (\|a^{\var}\|^{\ell}_{\dot{B}^{d/2+1}_{2,1}}+\|a^{\var}\|^{m}_{\dot{B}^{d/p}_{p,1}}+\|a^{\var}\|^{h}_{\dot{B}^{d/2}_{2,1}})\|\theta^{\var}\|^{\ell}_{\dot{B}^{-\sigma_1}_{2,\infty}}.
\end{equation}
Notice that $\nabla H_3(a^\varepsilon)=H'_3(0)\nabla a^\varepsilon+\nabla(\tilde{H}_3(a^\varepsilon)a^\varepsilon)$, we have
\begin{equation}\label{R-E136}
\|\theta^{\ell,\var}\nabla a^{\ell,\var}\|^{\ell}_{\dot{B}^{-\sigma_1}_{2,\infty}}\lesssim \|\nabla a^{\ell,\var}\|_{\dot{B}^{d/p}_{p,1}}\|\theta^{\ell,\var}\|_{\dot{B}^{-\sigma_1}_{2,\infty}}\lesssim\|a^{\var}\|^{\ell}_{\dot{B}^{d/2+1}_{2,1}}
\|\theta^{\var}\|^{\ell}_{\dot{B}^{-\sigma_1}_{2,\infty}},
\end{equation}
and
\begin{equation}\label{R-E137}
\|\theta^{\ell,\var}\nabla a^{\tilde{h},\var}\|^{\ell}_{\dot{B}^{-\sigma_1}_{2,\infty}}\lesssim \|\theta^{\ell,\var}\nabla a^{\tilde{h},\var}\|^{\ell}_{\dot{B}^{\frac dp-\frac d2-\sigma_1}_{2,\infty}}\lesssim 
(\|a^{\var}\|^{m}_{\dot{B}^{d/p}_{p,1}}+\|a^{\var}\|^{h}_{\dot{B}^{d/2}_{2,1}}))
\|\theta^{\ell,\var}\|_{\dot{B}^{-\sigma_1}_{2,\infty}},
\end{equation}
where we used the fact that $p\leq d^{*}$, which implies that $d/p-d/2-\sigma_{1}+1\geq-\sigma_{1}$.
Concerning, $H^\varepsilon$, we have
\begin{align}\label{R-E138}
    \|v^{\ell}\cdot\nabla \theta^{\ell}\|_{\dot{B}^{-\sigma_1}_{2,\infty}}^\ell\lesssim  \|\theta^{\ell}\|_{\dot{B}^{\frac d2+1}_{2,1}} \|v^{\ell}\|_{\dot{B}^{-\sigma_1}_{2,\infty}}
\end{align}
and
\begin{align}\label{R-E139}
    \|v^{h}\cdot\nabla \theta^{\ell}\|_{\dot{B}^{-\sigma_1}_{2,\infty}}^\ell\lesssim \left(\|v^{h}\|_{\dot{B}^{\frac d2}_{2,1}} +\|v^{m}\|_{\dot{B}^{\frac dp+1}_{p,1}} \right)\|\theta^{\ell}\|_{\dot{B}^{-\sigma_1}_{2,\infty}}.
\end{align}
Using that $J(a^\var)=J'(0)a^\var+\wt J(a^\varepsilon)a^\varepsilon$, we have \begin{align}\label{R-E140}
    \|J(a^\varepsilon)\div q^{\ell}\|_{\dot{B}^{-\sigma_1}_{2,\infty}}^\ell&\lesssim \|q^\var\|^\ell_{\dot{B}^{\frac d2}_{2,1}} \|a^{\ell}\|_{\dot{B}^{-\sigma_1}_{2,\infty}}+\left(\|a^\var\|^m_{\dot{B}^{\frac dp}_{p,1}}+\|a^\var\|^h_{\dot{B}^{\frac d2}_{2,1}}\right)\|q^{\ell}\|_{\dot{B}^{-\sigma_1}_{2,\infty}}+ \|a^\var\|^2_{\dot{B}^{\frac dp}_{p,1}} \|q\|^{\ell}_{\dot{B}^{-\sigma_1}_{2,\infty}},
\end{align}
\begin{align}\label{R-E141}
    \|\dfrac{N(\nabla v^\varepsilon,\nabla v^{\ell,\varepsilon})}{1+a^\varepsilon}\|_{\dot{B}^{-\sigma_1}_{2,\infty}}^\ell&\lesssim (1+\|a^\var\|^\ell_{\dot{B}^{\frac dp}_{p,1}}) \|N(\nabla v^{\varepsilon},\nabla v^{\ell,\varepsilon})\|_{\dot{B}^{-\sigma_1}_{2,\infty}}
    \\ \nonumber  & \lesssim (1+\|a^\var\|^\ell_{\dot{B}^{\frac dp}_{p,1}})\left(\|v^\var\|^\ell_{\dot{B}^{\frac d2+1}_{2,1}}+\|v^\var\|^{m,\varepsilon}_{\dot{B}^{\frac dp+1}_{p,1}}+\|v^\var\|^{h,\varepsilon}_{\dot{B}^{\frac d2+1}_{2,1}}\right)\|v^{\varepsilon}\|^\ell_{\dot{B}^{-\sigma_1}_{2,\infty}}
\end{align}
and
\begin{align}\label{R-E142}
    \|H_1(a^\varepsilon)\theta^\varepsilon\div v^{\varepsilon,\ell}\|_{\dot{B}^{-\sigma_1}_{2,\infty}}^\ell&\lesssim \|a^{\varepsilon}\|_{\dot{B}^{\frac dp}_{p,1}} \|\theta^\varepsilon\div v^{\varepsilon,\ell}\|_{\dot{B}^{\frac dp+1}_{p,1}}
    \\ \nonumber&\lesssim \|a^{\varepsilon}\|_{\dot{B}^{\frac dp}_{p,1}} \|v^\varepsilon\|^\ell_{\dot{B}^{\frac d2+1}_{2,1}}\|\theta^\varepsilon\|^\ell_{\dot{B}^{-\sigma_1}_{2,\infty}}\\&+ \|a^{\varepsilon}\|_{\dot{B}^{\frac dp}_{p,1}} \left(\|\theta^\varepsilon\|^{m,\varepsilon}_{\dot{B}^{\frac dp}_{p,1}}+\dfrac{1}{\varepsilon}\|\theta^\varepsilon\|^{h,\varepsilon}_{\dot{B}^{\frac d2}_{2,1}}\right)\|v^\varepsilon\|^\ell_{\dot{B}^{-\sigma_1}_{2,\infty}}.
\end{align}
Using that $\nabla(K_3(a^\varepsilon)a^\varepsilon)\theta^\varepsilon=\nabla(K_3(a^\varepsilon)a^\varepsilon)\theta^{\ell}+\nabla(K_3(a^\varepsilon)a^\varepsilon)\theta^{h}$ and $p\leq d^*=\frac{2d}{d-2} \Leftrightarrow \frac dp-\frac d2-\sigma_1+1\geq -\sigma_1$, we have
\begin{align}\label{R-E15003}
\|\nabla(K_3(a^\varepsilon)a^\varepsilon)\theta^{\ell}\|^{\ell}_{\dot{B}^{-\sigma_1}_{2,\infty}}&\lesssim \|K(a^\varepsilon)a^\varepsilon\|_{\dot{B}^{\frac dp}_{p,1}}\|\theta^\varepsilon\|^{\ell}_{\dot{B}^{\frac dp-\frac d2-\sigma_1+1}_{2,\infty}}
\\&\lesssim \|a^\varepsilon\|^2_{\dot{B}^{\frac dp}_{p,1}}\|\theta^\varepsilon\|^{\ell}_{\dot{B}^{\sigma_1}_{2,\infty}}.\nonumber
\end{align}
The estimates for $I^{\ell,\varepsilon}$ in low-frequency follow similar lines. 
We now focus on the high-frequency counterpart. For $F^{h}$, we have
\begin{equation}\label{R-E143}
\|a^{\varepsilon}\div v^{h}\|^{\ell}_{\dot{B}^{-\sigma_1}_{2,\infty}}\lesssim \left(\|a^\varepsilon\|^\ell_{\dot{B}^{\frac d2-1}_{2,1}}+\|a^\varepsilon\|^{m,\varepsilon}_{\dot{B}^{\frac dp}_{p,1}}+\varepsilon\|a^\varepsilon\|^{h,\varepsilon}_{\dot{B}^{\frac d2}_{2,1}}\right)\|\div v^{\varepsilon}\|^h_{\dot{B}^{\frac dp-1}_{p,1}},
\end{equation}

\begin{equation}\label{R-E144}
\|v^{\varepsilon}\nabla a^{h}\|^{\ell}_{\dot{B}^{-\sigma_1}_{2,\infty}}\lesssim \left(\|v^\varepsilon\|^\ell_{\dot{B}^{\frac d2-1}_{2,1}}+\|v^\varepsilon\|^{m,\varepsilon}_{\dot{B}^{\frac dp-1}_{p,1}}+\varepsilon\|v^\varepsilon\|^{h,\varepsilon}_{\dot{B}^{\frac d2}_{2,1}}\right)\|\nabla a^{\varepsilon}\|^h_{\dot{B}^{\frac dp-1}_{p,1}},
\end{equation}
\begin{equation}\label{R-E145}
\|v^{\varepsilon}\nabla v^{h}\|^{\ell}_{\dot{B}^{-\sigma_1}_{2,\infty}}\lesssim \left(\|v^\varepsilon\|^\ell_{\dot{B}^{\frac d2-1}_{2,1}}+\|v^\varepsilon\|^{m,\varepsilon}_{\dot{B}^{\frac dp-1}_{p,1}}+\varepsilon\|v^\varepsilon\|^{h,\varepsilon}_{\dot{B}^{\frac d2}_{2,1}}\right)\|\nabla v^{\varepsilon}\|^h_{\dot{B}^{\frac dp-1}_{p,1}},
\end{equation}

\begin{equation}\label{R-E146}
\|J(a^{\varepsilon})\mathcal{A}v^{h}\|^{\ell}_{\dot{B}^{-\sigma_1}_{2,\infty}}\lesssim \left(\|a^\varepsilon\|^\ell_{\dot{B}^{\frac d2-1}_{2,1}}+\|a^\varepsilon\|^{m,\varepsilon}_{\dot{B}^{\frac dp}_{p,1}}+\varepsilon\|a^\varepsilon\|^{h,\varepsilon}_{\dot{B}^{\frac d2}_{2,1}}\right)\|\mathcal{A}v^{\varepsilon}\|^h_{\dot{B}^{\frac dp-1}_{p,1}},
\end{equation}

\begin{equation}\label{R-E147}
\|H_1(a^{\varepsilon})\nabla a^{h}\|^{\ell}_{\dot{B}^{-\sigma_1}_{2,\infty}}\lesssim \left(\|a^\varepsilon\|^\ell_{\dot{B}^{\frac d2-1}_{2,1}}+\|a^\varepsilon\|^{m,\varepsilon}_{\dot{B}^{\frac dp}_{p,1}}+\varepsilon\|a^\varepsilon\|^{h,\varepsilon}_{\dot{B}^{\frac d2}_{2,1}}\right)\|\nabla a^{\varepsilon}\|^h_{\dot{B}^{\frac dp-1}_{p,1}},
\end{equation}

\begin{equation}\label{R-E148}
\|H_2(a^{\varepsilon})\nabla \theta^{h}\|^{\ell}_{\dot{B}^{-\sigma_1}_{2,\infty}}\lesssim \left(\|a^\varepsilon\|^\ell_{\dot{B}^{\frac d2-1}_{2,1}}+\|a^\varepsilon\|^{m,\varepsilon}_{\dot{B}^{\frac dp}_{p,1}}+\varepsilon\|a^\varepsilon\|^{h,\varepsilon}_{\dot{B}^{\frac d2}_{2,1}}\right)\|\nabla \theta^{\varepsilon}\|^h_{\dot{B}^{\frac dp-1}_{p,1}}
\end{equation}
and
\begin{align}\label{R-E149}
\|\theta^{\varepsilon}\nabla a^{h}\|^{\ell}_{\dot{B}^{-\sigma_1}_{2,\infty}}&\lesssim \left(\|\theta^\varepsilon\|^\ell_{\dot{B}^{\frac d2-1}_{2,1}}+\|\theta^\varepsilon\|^{m,\varepsilon}_{\dot{B}^{\frac dp-1}_{p,1}}+\|\theta^\varepsilon\|^{h,\varepsilon}_{\dot{B}^{\frac d2-1}_{2,1}}\right)\|\nabla a^{\varepsilon}\|^h_{\dot{B}^{\frac dp-1}_{p,1}}.
\end{align}
Concerning the other terms, we have
\begin{equation}\label{R-E150}
\|v^{\varepsilon}\nabla \theta^{h}\|^{\ell}_{\dot{B}^{-\sigma_1}_{2,\infty}}\lesssim \left(\|v^\varepsilon\|^\ell_{\dot{B}^{\frac d2-1}_{2,1}}+\|v^\varepsilon\|^{m,\varepsilon}_{\dot{B}^{\frac dp-1}_{p,1}}+\varepsilon\|v^\varepsilon\|^{h,\varepsilon}_{\dot{B}^{\frac d2}_{2,1}}\right)\|\nabla \theta^{\varepsilon}\|^h_{\dot{B}^{\frac dp-1}_{p,1}}
\end{equation}
and
\begin{equation}\label{R-E151}
\|J(a^{\varepsilon})\div q^{h}\|^{\ell}_{\dot{B}^{-\sigma_1}_{2,\infty}}\lesssim \left(\|a^\varepsilon\|^\ell_{\dot{B}^{\frac d2-1}_{2,1}}+\|a^\varepsilon\|^{m,\varepsilon}_{\dot{B}^{\frac dp}_{p,1}}+\varepsilon\|a^\varepsilon\|^{h,\varepsilon}_{\dot{B}^{\frac d2}_{2,1}}\right)\|\div q^{\varepsilon}\|^h_{\dot{B}^{\frac dp-1}_{p,1}}.
\end{equation}
Concerning $\nabla(K_3(a^\varepsilon)a^\varepsilon)\theta^{h}$, we split its analysis into two cases. For $\frac dp-\frac d2< \sigma_1\leq \sigma_0=\dfrac{2d}{p}-\dfrac{d}{2}$, $p<d$, using the embedding $L^p\hookrightarrow \dot{B}^{\frac{d}{2}-\frac{2d}{p}}_{2,\infty}$, we have
\begin{align}\label{R-E15004}
\|\nabla K_3(a^\varepsilon)\theta^{h}\|^{\ell}_{\dot{B}^{-\sigma_0}_{2,\infty}}&\lesssim \|K(a^\varepsilon)\|_{L^p}\|\theta^\varepsilon\|^{h}_{\dot{B}^{0}_{p,1}}
\end{align}
which leads to
\begin{align}\label{R-E15005}
\|\nabla K_3(a^\varepsilon)\|_{L^p}&\lesssim \|\nabla a^\varepsilon\|_{L^p}\lesssim \|\nabla a^\varepsilon\|^{\ell}_{\dot{B}^{\frac d2-\frac dp}_{2,1}}+\|\nabla a^\var\|^{h}_{\dot{B}^{0}_{p,1}}
\\\nonumber &\lesssim\|a^\varepsilon\|^{\ell}_{\dot{B}^{\frac d2-\frac dp}_{2,1}}+\|a^\var\|^{m,\var}_{\dot{B}^{\frac dp}_{p,1}}+\| a^\var\|^{h,\var}_{\dot{B}^{\frac d2}_{2,1}}.
\end{align}
Then, using an interpolation inequality, for $\theta_2=\frac{\sigma_1+\frac d2-\frac dp}{\sigma_1+\frac d2-1}$, we have
\begin{align}
\|a^\varepsilon\|^{\ell}_{\dot{B}^{\frac d2-\frac dp}_{2,1}} &\lesssim \left(\|a^\varepsilon\|^{\ell,}_{\dot{B}^{-\sigma_1}_{2,\infty}}\right)^{1-\theta_2}\left(\|a^\varepsilon\|^{\ell}_{\dot{B}^{\frac d2-1}_{2,1}}\right)^{\theta_2}.
\end{align}
Thus
\begin{align}\label{R-E150006}
\|\nabla K_3(a^\varepsilon)\theta^{h}\|^{\ell}_{\dot{B}^{-\sigma_0}_{2,\infty}}&\lesssim \left(\|a^\varepsilon\|^{\ell}_{\dot{B}^{\frac d2-1}_{2,1}}+\|a^\varepsilon\|^{\ell}_{\dot{B}^{-\sigma_1}_{2,\infty}}+\|a^\var\|^{m,\var}_{\dot{B}^{\frac dp}_{p,1}}+\| a^\var\|^{h,\var}_{\dot{B}^{\frac d2}_{2,1}}\right)\left(\|\theta^\varepsilon\|^{m,\varepsilon}_{\dot{B}^{\frac dp}_{p,1}}+\|\theta^\varepsilon\|^{h,\varepsilon}_{\dot{B}^{\frac d2}_{2,1}}\right).
\end{align}
In the case $1-\frac{d}{2}<\sigma_1\leq \frac{d}{p}-\frac d2\leq 0$, $p<d$, we have $\dot{B}^{\frac{d}{p}-1}_{p,1}\hookrightarrow L^d$ and thus 
\begin{align}\label{R-E150007}
\|\nabla K_3(a^\varepsilon)\theta^{h}\|^{\ell}_{\dot{B}^{-\sigma_1}_{2,\infty}}&\lesssim \|a^\varepsilon\|_{\dot{B}^{\frac dp}_{p,1}}\left(\|\theta^\varepsilon\|^{m,\varepsilon}_{\dot{B}^{\frac d2}_{p,1}}+\|\theta^\varepsilon\|^{h,\varepsilon}_{\dot{B}^{\frac dp}_{p,1}}\right).
\end{align}
Since $\dfrac{N(\nabla v^\var,\nabla v^{\var,h}))}{1+a^\varepsilon}=(1+I(a^\varepsilon)N(\nabla v^\var,\nabla v^{\var,h}))$, in the case $\frac{d}{p}-\frac{d}{2}<\sigma_1\leq \sigma_0$, we have 
\begin{align}
    \|\dfrac{N(\nabla v^\var,\nabla v^{\var,h}))}{1+a^\varepsilon}\|^{\ell}_{\dot{B}^{-\sigma_1}_{2,\infty}}&\lesssim \left(1+\|a^\varepsilon\|_{\dot{B}^{\frac dp}_{p,1}}\right)\|\nabla v^{\varepsilon}\|_{L^p}\|\nabla v^{\varepsilon,h}\|_{L^p},
    \\&\lesssim \left(1+\|a^\varepsilon\|_{\dot{B}^{\frac dp}_{p,1}}\right)\left(\|v^\varepsilon\|^{\ell}_{\dot{B}^{\frac d2-1}_{2,1}}+\|v^\varepsilon\|^{\ell}_{\dot{B}^{-\sigma_1}_{2,\infty}}+\|v^\var\|^{m,\var}_{\dot{B}^{\frac dp}_{p,1}}+\| v^\var\|^{h,\var}_{\dot{B}^{\frac d2}_{2,1}}\right)\nonumber\\&\times\left(\|v^\varepsilon\|^{m,\varepsilon}_{\dot{B}^{\frac dp}_{p,1}}+\|v^\varepsilon\|^{h,\varepsilon}_{\dot{B}^{\frac d2}_{2,1}}\right). \nonumber
\end{align}
In the case $1-\frac d2< \sigma_1\leq \frac dp-\frac d2\leq0$, we have 
\begin{align}
    \|\dfrac{N(\nabla v^\var,\nabla v^{\var,h}))}{1+a^\varepsilon}\|^{\ell}_{\dot{B}^{-\sigma_1}_{2,\infty}}&\lesssim \left(1+\|a^\varepsilon\|_{\dot{B}^{\frac dp}_{p,1}}\right)\|\nabla v^{\varepsilon}\|_{L^d}\|\nabla v^{\varepsilon,h}\|_{L^{d^*}},
    \\&\lesssim \left(1+\|a^\varepsilon\|_{\dot{B}^{\frac dp}_{p,1}}\right)\|v^\varepsilon\|_{\dot{B}^{\frac dp}_{p,1}}\|v^{\varepsilon}\|^h_{\dot{B}^{-\frac dp+1}_{p,1}},
\end{align}

\begin{equation}\label{R-E152}
\|v^{\varepsilon}\nabla q^{h}\|^{\ell}_{\dot{B}^{-\sigma_1}_{2,\infty}}\lesssim \left(\|v^\varepsilon\|^\ell_{\dot{B}^{\frac d2-1}_{2,1}}+\|v^\varepsilon\|^{m,\varepsilon}_{\dot{B}^{\frac dp-1}_{p,1}}+\varepsilon\|v^\varepsilon\|^{h,\varepsilon}_{\dot{B}^{\frac d2}_{2,1}}\right)\|\nabla q^{\varepsilon}\|^h_{\dot{B}^{\frac dp-1}_{p,1}},
\end{equation}

\begin{equation}\label{R-E1534}
\|q^{\varepsilon}\nabla v^{h}\|^{\ell}_{\dot{B}^{-\sigma_1}_{2,\infty}}\lesssim \left(\|q^\varepsilon\|^\ell_{\dot{B}^{\frac d2-1}_{2,1}}+\|q^\varepsilon\|^{m,\varepsilon}_{\dot{B}^{\frac dp-1}_{p,1}}+\varepsilon\|q^\varepsilon\|^{h,\varepsilon}_{\dot{B}^{\frac d2-1}_{2,1}}\right)\|\nabla v^{\varepsilon}\|^h_{\dot{B}^{\frac dp-1}_{p,1}}
\end{equation}
and
\begin{equation}\label{R-E153}
\|q^{\varepsilon}\div v^{h}\|^{\ell}_{\dot{B}^{-\sigma_1}_{2,\infty}}\lesssim \left(\|q^\varepsilon\|^\ell_{\dot{B}^{\frac d2-1}_{2,1}}+\|q^\varepsilon\|^{m,\varepsilon}_{\dot{B}^{\frac dp-1}_{p,1}}+\varepsilon\|q^\varepsilon\|^{h,\varepsilon}_{\dot{B}^{\frac d2-1}_{2,1}}\right)\|\div v^{\varepsilon}\|^h_{\dot{B}^{\frac dp-1}_{p,1}}.
\end{equation}
Concerning $H_1(a^\varepsilon)\theta^\varepsilon\div v^\varepsilon$, in the case $\frac dp-\frac d2< \sigma_1 \leq \sigma_0$, we have\begin{align}\label{R-E1423} \|H_1(a^\varepsilon)\theta^\varepsilon\div v^\varepsilon\|_{\dot{B}^{-\sigma_1}_{2,\infty}}^\ell&\lesssim \|a^{\varepsilon}\|_{\dot{B}^{\frac dp}_{p,1}}\left(\|\theta^\varepsilon\|^{\ell}_{\dot{B}^{\frac d2-1}_{2,1}}+\|\theta^\varepsilon\|^{\ell}_{\dot{B}^{-\sigma_1}_{2,\infty}}+\|\theta^\var\|^{m,\var}_{\dot{B}^{\frac dp-1}_{p,1}}+\| \theta^\var\|^{h,\var}_{\dot{B}^{\frac d2-1}_{2,1}}\right)\\&\times\left(\|v^\varepsilon\|^{m,\varepsilon}_{\dot{B}^{\frac dp}_{p,1}}+\|v^\varepsilon\|^{h,\varepsilon}_{\dot{B}^{\frac d2}_{2,1}}\right). \nonumber
\end{align}
and, in the case $1 - \frac d2 < \sigma_1 \leq \frac dp-\frac d2 \leq 0$, we have 
\begin{align}\label{R-E1424}
    \|H_1(a^\varepsilon)\theta^\varepsilon\div v^{\varepsilon,h}\|_{\dot{B}^{-\sigma_1}_{2,\infty}}^\ell&\lesssim \|H_1(a^\varepsilon)\theta^\varepsilon\div v^{\varepsilon,h}\|_{\dot{B}^{0}_{2,\infty}}^\ell
    \\&\lesssim \|a^{\varepsilon}\|_{\dot{B}^{\frac dp}_{p,1}}\|\theta^\varepsilon\|_{\dot{B}^{\frac dp-1}_{p,1}}\|v^\varepsilon\|^{h,\varepsilon}_{\dot{B}^{\frac dp+1}_{p,1}}. \nonumber
\end{align}
Gathering the estimates from this section and using that all the right-hand side can be bounded by $X^\varepsilon(t)^2$ concludes the proof of Lemma \ref{LemmaDecay}.
\end{proof}
\subsection{Conclusion of the proof of Theorem \ref{Thm:decay}}

Gathering the estimates from the previous sections and using that $X^\var(t)\lesssim X_0^\var \ll 1$, we obtain
\begin{align*}
&\dfrac{d}{dt}\|(a^\varepsilon,v^\varepsilon,\theta^\varepsilon,\varepsilon q^{\varepsilon})\|^{\ell}_{\dot{B}^{\frac {d}{2}-1}_{2,1}}+\dfrac{d}{dt}(\|a^{\var}\|^{m}_{\dot{B}^{d/p}_{p,1}}+\|w^{\var}\|^{m}_{\dot{B}^{d/p-1}_{p,1}}+\|\var Q^{\var}\|^{m}_{\dot{B}^{d/p-2}_{p,1}}+\|\theta^{\var}\|^{m}_{\dot{B}^{d/p-2}_{p,1}})
\\&+\frac{d}{dt}\Big(\|\varepsilon a^{\var}\|^{h}_{\dot{B}^{d/2+1}_{2,1}}+\|\varepsilon w^{\var}\|^{h}_{\dot{B}^{d/2}_{2,1}}+\|\varepsilon^2 \theta^{\var}\|^{h}_{\dot{B}^{d/2+1}_{2,1}}+\|\varepsilon^3q^{\var}\|^{h}_{\dot{B}^{d/2+1}_{2,1}}\Big)
\\&+\|( a^\varepsilon,v^\varepsilon,\theta^\varepsilon)^{\ell}\|_{\dot{B}^{\frac {d}{2}+1}_{2,1}}+\frac{1}{\varepsilon}\|Q^{\varepsilon,\ell}\|_{\dot{B}^{\frac {d}{2}-1}_{2,1}}+\|q^{\varepsilon,\ell}\|_{\dot{B}^{\frac {d}{2}}_{2,1}}
+\|a^{\var}\|^{m}_{\dot{B}^{d/p}_{p,1}}
+\|w^{\var}\|^{m}_{\dot{B}^{d/p+1}_{p,1}}\\&+\frac{1}{\varepsilon}\|Q^{\var}\|^{m}_{\dot{B}^{d/p-2}_{p,1}}+\|\theta^{\var}\|^{m}_{\dot{B}^{d/p}_{p,1}}+\|\varepsilon a^{\var}\|^{h}_{\dot{B}^{d/2+1}_{2,1}}+\|\varepsilon w^{\var}\|^{h}_{\dot{B}^{d/2+2}_{2,1}}+\|(\theta^{\var},\var q^{\var})\|^{h}_{\dot{B}^{d/2+1}_{2,1}} \leq 0.
\end{align*}
Then, we employ a classical interpolation argument to derive the time-decay estimates. Since $-\sigma_1<\frac d2-1\leq \frac dp< \frac d2+1$, we have
\begin{align}
    \|(a^\varepsilon,v^\varepsilon,\theta^\varepsilon)\|^{\ell}_{\dot{B}^{\frac {d}{2}-1}_{2,1}}\leq \left(\|(a^\varepsilon,v^\varepsilon,\theta^\varepsilon)\|^{\ell}\|_{\dot{B}^{-\sigma_1}_{2,\infty}}\right)^{\theta_1}\left(\|(a^\varepsilon,v^\varepsilon,\theta^\varepsilon)^{\ell}\|_{\dot{B}^{\frac {d}{2}+1}_{2,\infty}}\right)^{1-\theta_1}
\end{align}
where $\theta=\frac{2}{d/2+1+\sigma_1}$. Using the $\dot{B}^{-\sigma_1}_{2,\infty}$-boundedness obtained in Lemma \ref{LemmaDecay}, we deduce that 
\begin{align}
   \|(a^\varepsilon,v^\varepsilon,\theta^\varepsilon)\|^{\ell}_{\dot{B}^{\frac {d}{2}+1}_{2,1}} \geq c_0\left(\|(a^\varepsilon,v^\varepsilon,\theta^\varepsilon)\|^{\ell}_{\dot{B}^{\frac {d}{2}-1}_{2,1}}\right)^{\frac{1}{1-\theta_1}}.
\end{align}
In addition, it is easy to see that
\begin{align}   \|q^\varepsilon\|^{\ell}_{\dot{B}^{\frac {d}{2}}_{2,1}} \geq \left(\|\var q^\varepsilon\|^{\ell}_{\dot{B}^{\frac {d}{2}-1}_{2,1}}\right)^{\frac{1}{1-\theta_1}}.
\end{align}
Concerning the medium frequencies, using Bernstein inequality, we have
\begin{align}  \|a^\varepsilon\|^{m,\varepsilon}_{\dot{B}^{\frac {d}{p}}_{p,1}} \geq C\left(\| a^\varepsilon\|^{m,\var}_{\dot{B}^{\frac {d}{p}}_{p,1}}\right)^{\frac{1}{1-\theta_1}}, \; 
\|w^\varepsilon\|^{m,\varepsilon}_{\dot{B}^{\frac {d}{p}+1}_{p,1}} \geq C\left(\| w^\varepsilon\|^{m,\var}_{\dot{B}^{\frac {d}{p}-1}_{p,1}}\right)^{\frac{1}{1-\theta_1}}, \;  \|\theta^\varepsilon\|^{m,\varepsilon}_{\dot{B}^{\frac {d}{p}}_{p,1}} \geq C\left(\| \theta^\varepsilon\|^{m,\var}_{\dot{B}^{\frac {d}{p}-2}_{p,1}}\right)^{\frac{1}{1-\theta_1}}.
\end{align}
Moreover, it is clear that
\begin{align}   \frac{1}{\var}\|Q^\varepsilon\|^{m,\varepsilon}_{\dot{B}^{\frac {d}{p}-2}_{p,1}} \geq \left(\|\varepsilon Q^\varepsilon\|^{m,\var}_{\dot{B}^{\frac {d}{p}-2}_{p,1}}\right)^{\frac{1}{1-\theta_1}}.
\end{align}
Concerning the high frequencies, one has
\begin{align}  \|\varepsilon a^\varepsilon\|^{h,\varepsilon}_{\dot{B}^{\frac {d}{2}+1}_{2,1}} \geq \left(\| \varepsilon a^\varepsilon\|^{h,\var}_{\dot{B}^{\frac {d}{2}+1}_{2,1}}\right)^{\frac{1}{1-\theta_1}}, \; 
\varepsilon^2\|\varepsilon w^\varepsilon\|^{h,\varepsilon}_{\dot{B}^{\frac {d}{2}+2}_{2,1}} \geq \left(\| \varepsilon w^\varepsilon\|^{h,\var}_{\dot{B}^{\frac {d}{2}}_{2,1}}\right)^{\frac{1}{1-\theta_1}}
\end{align}
and
\begin{align}  \|(\theta^\varepsilon,\varepsilon q^\varepsilon)\|^{h,\varepsilon}_{\dot{B}^{\frac {d}{2}+1}_{2,1}} \geq \left(\| (\varepsilon^2\theta^\varepsilon,\varepsilon^3q^\varepsilon)\|^{h,\var}_{\dot{B}^{\frac {d}{2}+1}_{2,1}}\right)^{\frac{1}{1-\theta_1}}.
\end{align}
Therefore, defining 
\begin{align}
\mathcal{L}_1&=\|(a^\varepsilon,v^\varepsilon,\theta^\varepsilon,\varepsilon q^{\varepsilon})\|^{\ell}_{\dot{B}^{\frac {d}{2}-1}_{2,1}}+\|a^{\var}\|^{m}_{\dot{B}^{d/p}_{p,1}}+\|w^{\var}\|^{m}_{\dot{B}^{d/p-1}_{p,1}}+\|\var Q^{\var}\|^{m}_{\dot{B}^{d/p-2}_{p,1}}+\|\theta^{\var}\|^{m}_{\dot{B}^{d/p-2}_{p,1}}
\\&\quad +\|\varepsilon a^{\var}\|^{h}_{\dot{B}^{d/2+1}_{2,1}}+\|\varepsilon w^{\var}\|^{h}_{\dot{B}^{d/2}_{2,1}}+\|\varepsilon^2 \theta^{\var}\|^{h}_{\dot{B}^{d/2+1}_{2,1}}+\|\varepsilon^3q^{\var}\|^{h}_{\dot{B}^{d/2+1}_{2,1}},  
\end{align}
and gathering the previous estimates, we obtain
\begin{align}\label{est:lastdecay}
    \dfrac{d}{dt}\mathcal{L}_1+c_0\mathcal{L}_1^{1+\frac{2}{d/2-1+\sigma_1}} \leq 0.
\end{align}
Solving \eqref{est:lastdecay} and following the embedding arguments from \cite{XX} gives the desired decay estimate \eqref{est:thmdecay}, which concludes the proof of Theorem \ref{Thm:decay}.
\qed



\section{Proof of the relaxation Theorem \ref{thm:relax}}\label{sec:relaxation}
\subsection{Formulation of the error system}
Let $(a_0^\varepsilon,u^\var,\theta^\var,q^\var)$ and $(a,u,\theta)$ be the solutions of \eqref{LinearHypNSCzero} and \eqref{LinearHypNSFzero} from Theorem \ref{thm:exist} and \ref{Thm:ExistHeDanchin} associated to the initial data $(a^\varepsilon_0,v^\var_0,\theta^\var_0,q^\var_0)$ and $(a_0,v_0,\theta_0)$, respectively.
We will prove that as $\varepsilon\to0$, we have 
$(a^\varepsilon_0,v^\var_0,\theta^\var_0,q^\var_0)\to (a_0,v_0,\theta_0)$ strongly in some suitable homogeneous Besov norms. To that matter, we define the error unknowns $(\widetilde{a},\widetilde{v},\widetilde{\theta})$ as $$(\widetilde{a},\widetilde{v},\widetilde{\theta}):=(a^\varepsilon-a,v^\varepsilon-v,\theta^\varepsilon-\theta).$$ The couple $(\wt a, \wt v, \wt \theta)$ satisfies
\begin{equation} \label{relaxw}
\left\{
\begin{array}
[c]{l}%
\d_t \widetilde{a}+\div \widetilde{v}=-\wt F,\\
\d_t \widetilde v-\mathcal{A}\wt v+\nabla \widetilde a +\nabla \widetilde \theta=-\wt G,\\
\d_t \widetilde\theta- \Delta \widetilde\theta+\div \widetilde v=-\div Q-\wt H,\\
\end{array}
\right.
\end{equation}
where 
\begin{align*}
    &\wt F = v^\varepsilon\cdot \nabla \wt a+\wt v\cdot\nabla  a+ \wt a\: \div v^\varepsilon +  a\div \wt v,
    \\  &\wt G = \wt v\cdot \nabla v-v^\varepsilon\cdot \nabla \wt v+\left(J(a^\varepsilon)-J(a)\right)\cA v^\varepsilon/\nu+J(a)\cA\wt v/\nu+(H_1(a^\varepsilon)-H_1(a))\nabla a^\varepsilon+H_1(a))\nabla \wt a
    \\& \quad\quad  +(H_2(a^\varepsilon)-H_2(a))\nabla \theta^\varepsilon+H_2(a))\nabla \wt \theta+\wt \theta\nabla H_3(a^\varepsilon)+ \theta\nabla (H_3(a^\varepsilon)-H_3(a)),
    \\  &\wt H = \wt v\cdot\nabla\theta^\varepsilon +  v\cdot\nabla \wt\theta + (J(a^\varepsilon)-J(a))\div q^\varepsilon+J(a)(\div q^\varepsilon-\Delta \theta) -\wt R,
\end{align*}
where
$$\wt R= \dfrac{N(\nabla v^\varepsilon,\nabla v^\varepsilon)}{1+a^\varepsilon}-\dfrac{N(\nabla v,\nabla v)}{1+a}+\wt H_1(a^\varepsilon)\theta^\varepsilon \div v^\varepsilon-\wt H_1( a)\theta \div v.$$
The linear part of \eqref{relaxw} has a similar structure to the Navier-Stokes-Fourier system \eqref{FullNSC}, thus, as in \cite{DanchinFullNSC}, we analyze differently the low and high frequencies. In order to control the linear part of the source terms, we derive a priori estimates at different regularities in both frequency regime. To find the optimal regularity indexes to do so, we first check in which space we are able to extract a $\mathcal{O}(\var)$ bound for the linear source term $\div Q$. From Theorem \ref{thm:exist}, we have
\begin{align}\label{eq:Qepsilon}
\|Q^\varepsilon\|_{L^1_T(\dot{B}^{\frac{d}{2}-1}_{2,1})}^\ell+\|Q^\varepsilon\|_{L^1_T(\dot{B}^{\frac{d}{p}-2}_{p,1}\cap \dot{B}^{\frac{d}{p}-1}_{p,1})}^{m,\var}+\|Q^\varepsilon\|_{L^1_T(\dot{B}^{\frac{d}{2}-1}_{2,1})}^{h,\var}=\mathcal{O}(\varepsilon).    
\end{align}
This suggests us to work at the regularity index $d/2-2$ and $d/p-2$ in low frequencies and high frequencies, respectively, for the component $\theta^\varepsilon.$

\subsection{Error estimates: linear analysis}
We have the following proposition.
\begin{Prop} \label{prop:errorlin}
Let $( a^\var,u^\var,\theta^\var,q^\var)$ and $( a,u,\theta)$ be the solutions of the system \eqref{LinearHypNSCzero} and \eqref{LinearHypNSFzero} from Theorem \ref{thm:exist} and \ref{Thm:ExistHeDanchin} respectively, such that their initial data satisfy 
\begin{align}
\|(a_0^\varepsilon- a_0,v_0^\varepsilon-v_0,\theta_0^\varepsilon-\theta_0)\|^\ell_{\dot{B}^{\frac{d}{2}-2}_{2,1}}+\|a_0^\varepsilon- a_0\|^{h}_{\dot{B}^{\frac{d}{p}-1}_{p,1}}+\|(v_0^\varepsilon-v_0,\theta_0^\varepsilon-\theta_0)\|^{h}_{\dot{B}^{\frac{d}{p}-2}_{p,1}}\leq C\varepsilon.    
\end{align}
Then
\begin{align*}
\widetilde{X}(t):=&\|(\wt a,\wt v,\wt \theta)\|^\ell_{L^\infty_T(\dot{B}^{\frac{d}{2}-2}_{2,1})}+\|(\wt a,\wt v,\wt \theta)\|^\ell_{L^1_T(\dot{B}^{\frac{d}{2}}_{2,1})}
 \\ & +\|\wt a\|^{h}_{L^\infty_T\cap L^1_T(\dot{B}^{\frac{d}{p}-1}_{p,1})}+\|(\wt v,\wt \theta)\|^{h}_{L^\infty_T(\dot{B}^{\frac{d}{p}-2}_{p,1})}+\|(\wt v,\wt \theta)\|^{h}_{L^1_T(\dot{B}^{\frac{d}{p}}_{p,1})}
    \lesssim \varepsilon+\wt X(t)X^\varepsilon(t) + \mathcal{I}_1+ \mathcal{I}_2,
\end{align*}
where
 \begin{align*}\mathcal{I}_1=\|(\wt F,\wt G,\wt H)\|^\ell_{L^1_T(\dot{B}^{\frac{d}{2}-2}_{2,1})} \andf \mathcal{I}_2=\|\wt F-v^\varepsilon\cdot \nabla \wt a\|^{h}_{L^1_T(\dot{B}^{\frac{d}{p}-1}_{2,1})}+\|(\wt G,\wt H)\|^{h}_{L^1_T(\dot{B}^{\frac{d}{p}-2}_{p,1})}.
 \end{align*}
\end{Prop}
\begin{proof}
   \noindent\textbf{Step1: low-frequency estimates: $j\leq J_0$.}
Applying the localization operator $\ddj$ to \eqref{relaxw}, we obtain
\begin{equation} \label{relaxw2}
\left\{
\begin{array}
[c]{l}%
\d_t \widetilde{a}_j+\div \widetilde{v}+v^\varepsilon\cdot \nabla \wt a=-\wt F_j,\\
\d_t \widetilde v_j-\mathcal{A}\wt v_j+\nabla \widetilde a_j +\nabla \widetilde \theta_j=-\wt G_j,\\
\d_t \widetilde\theta_j- \Delta \widetilde\theta_j+\div \widetilde v_j=-\div Q_j-\wt H_j.
\end{array}
\right.
\end{equation}
Defining the perturbed energy functional
\begin{align}\label{LyaBF}
\mathcal{\widetilde{L}}^\ell_j=\|( \widetilde{a}_j,\widetilde{v}_j,\widetilde{\theta}_j)\|_{L^2}^2+\frac12\int_\R \widetilde{v}_j\nabla  \widetilde{a}_j \quad \text{for } j\leq J_0,
\end{align}
and following the exact same steps as in Section \ref{nonlinearLf}, we obtain
\begin{align}\label{est:QlinearBf2X}
    \|(\wt a,\wt v,\wt\theta)\|^\ell_{L^\infty_T(\dot{B}^{\frac{d}{2}-2}_{2,1})}+ \|(\wt a,\wt v,\wt\theta)\|^\ell_{L^1_T( \dot{B}^{\frac d2}_{2,1})}\leq& \:\|(\wt a_0,\wt v_0,\wt\theta_0)\|_{\dot{B}^{\frac d2-2}}^\ell+\|Q\|^\ell_{L^1_T(\dot{B}^{\frac d2-1})}\\ \nonumber & +\|(\wt F,\wt G,\wt H)\|^\ell_{L^1_T(\dot{B}^{\frac d2-2})}.
\end{align}
Using \eqref{eq:Qepsilon}, we have
\begin{align}\label{eq:absorbQX}
   \|Q^\varepsilon\|^\ell_{L^1_T(\dot{B}^{\frac d2-1}_{2,1})}\leq \varepsilon
\end{align}
and thus \begin{align}\label{est:QlinearBf7X}
    \|(\wt a,\wt v,\wt\theta)\|^\ell_{L^\infty_T(\dot{B}^{\frac{d}{2}-2}_{2,1})}+ \|(\wt a,\wt v,\wt\theta)\|^\ell_{L^1_T( \dot{B}^{\frac d2}_{2,1})}\leq& \:\|(\wt a_0,\wt v_0,\wt\theta_0)\|_{\dot{B}^{\frac d2-1}_{2,1}}^\ell+\varepsilon +\|(\wt F,\wt G,\wt H)\|^\ell_{L^1_T(\dot{B}^{\frac d2-2})}. 
\end{align}
   \textbf{Step 2: high-frequency estimates: $j\geq J_0$.} In high frequencies, we follow the computations done in the medium-frequency regime in Section \ref{sec:med}. We introduce the effective velocity $\wt w=\wt v+(-\Delta)^{-1}\nabla \wt a$\: to partially diagonalize the system, it reads as
\begin{equation}
\left\{
\begin{array}
[c]{l}%
\d_t\wt a+\wt a+v^\varepsilon\cdot\nabla \wt a=\div \wt w+\wt F,\\
\d_t \wt w-\Delta \wt w =\wt w-(-\Delta)^{-1}\nabla\wt a+\nabla \wt\theta+\wt G+(-\Delta)^{-1}\nabla \wt F,\\
\d_t\wt \theta-\Delta \wt\theta=-\div \wt w - \wt a+\div Q+\wt H,
\end{array}
\right.
\label{LinearHypNSCMed2}
\end{equation}
Then, standard estimates for damped and diffusive equations lead to
 \begin{align} \nonumber\|\wt a\|^{h,\varepsilon}_{L^\infty_T(\dot{B}^{\frac{d}{p}-1}_{p,1})}+  \|\wt a\|^{h,\varepsilon}_{L^1_T( \dot{B}^{\frac dp-1}_{p,1})}&\lesssim  \|\wt a_0\|^{h,\varepsilon}_{\dot{B}^{\frac{d}{p}-1}_{p,1}}+  \|\wt w\|^{h,\varepsilon}_{L^1_T(\dot{B}^{\frac{d}{p}}_{p,1})}+\|\nabla  v^\varepsilon\|_{L^1_T(\dot{B}^{d/p-1}_{p,1})}\|\wt a^\varepsilon\|_{L^\infty_T(\dot{B}^{d/p-1}_{p,1})}+\|\wt F_1\|^{h,\varepsilon}_{L^1_T( \dot{B}^{\frac{d}{p}-1}_{p,1})}
 \\&\lesssim\label{eq:medrhoX}    \|\wt a_0\|^{h,\varepsilon}_{\dot{B}^{\frac{d}{p}-1}_{p,1}}+  \|\wt w\|^{h,\varepsilon}_{L^1_T(\dot{B}^{\frac{d}{p}}_{p,1})}+\|\wt F_1\|^{h,\varepsilon}_{L^1_T(\dot{B}^{\frac{d}{p}-1}_{p,1})}+\|v^\varepsilon\|_{L^1_T(\dot{B}^{\frac{d}{p}+1}_{p,1})}\|\wt a\|_{L^\infty_T(\dot{B}^{\frac{d}{p}-1}_{p,1})}
\end{align}
and
\begin{align}\label{eq:medwX}  \|\wt w\|^{h,\varepsilon}_{L^\infty_T(\dot{B}^{\frac{d}{p}-2}_{p,1})}+ \|\wt w\|^{h,\varepsilon}_{L^1_T(\dot{B}^{\frac dp}_{p,1})}&\lesssim \|\wt w_0\|^{h,\varepsilon}_{\dot{B}^{\frac{d}{p}-2}_{p,1}}+ \|\wt w\|^{h,\varepsilon}_{L^1_T(\dot{B}^{\frac{d}{p}-2}_{p,1})}+ \|\wt a\|^{h,\varepsilon}_{L^1_T(\dot{B}^{\frac{d}{p}-3}_{p,1})}\\ &+ \|\wt \theta\|^{h,\varepsilon}_{L^1_T(\dot{B}^{\frac{d}{p}-1}_{p,1})} +\|\wt G\|^{h,\varepsilon}_{L^1_T( \dot{B}^{\frac{d}{p}-2}_{p,1})}+\|\wt F_1\|^{h,\varepsilon}_{L^1_T(\dot{B}^{\frac{d}{p}-3}_{p,1})} .\nonumber
 \end{align}
 Moreover, for $\widetilde\theta$, we have
\begin{align}\label{eq:medthetaX}
   \|\wt \theta\|^{h,\varepsilon}_{L^\infty_T(\dot{B}^{\frac{d}{p}-2}_{p,1})}+ \|\wt \theta\|^{h,\varepsilon}_{L^1_T(\dot{B}^{\frac dp}_{p,1})}&\lesssim \|\wt \theta_0\|^{h,\varepsilon}_{\dot{B}^{\frac{d}{p}-2}_{p,1}}+ \|\wt w\|^{h,\varepsilon}_{L^1_T(\dot{B}^{\frac{d}{p}-1}_{p,1})}+\|\wt a\|^{h,\varepsilon}_{L^1_T(\dot{B}^{\frac{d}{p}-2}_{p,1})}\\&+\|Q\|^{h,\varepsilon}_{L^1_T(\dot{B}^{\frac{d}{p}-1}_{p,1})}+\|\wt H\|^{h,\varepsilon}_{L^1_T( \dot{B}^{\frac{d}{p}-2}_{p,1})}. \nonumber
\end{align}
Gathering \eqref{eq:medrhoX}, \eqref{eq:medwX} and \eqref{eq:medthetaX}, and using Berstein inequalities from Proposition \ref{prop:Bernstein}, the linear source terms can be absorbed for $J_0$ large enough and we obtain the desired estimates.
\end{proof}
\subsection{Error estimates: nonlinear analysis}
We are now left with the estimation of the nonlinear terms.
Together with Proposition \ref{prop:errorlin}, the following proposition concludes the proof of Theorem \ref{thm:relax}.
\begin{Prop}\label{prop:relax}
    We have
    $$ \mathcal{I}_1+\mathcal{I}_2 \leq \varepsilon X^\varepsilon(t)+\wt X(t)(X(t)+X^\varepsilon(t)),$$
    where $X(t)$ is defined in \eqref{X:HeDanchin}, $X^\varepsilon(t)$ in \eqref{X} and we recall that
\begin{align}\label{tildeX0}
\wt X(t)&=
\|(\wt a,\wt v,\wt \theta)\|^\ell_{L^\infty_T(\dot{B}^{\frac{d}{2}-2}_{2,1})}+\|(\wt a,\wt v,\wt \theta)\|^\ell_{L^1_T(\dot{B}^{\frac{d}{2}}_{2,1})}
  \\&\nonumber\quad\quad+ \|\wt a\|^{h,\varepsilon}_{L^\infty_T\cap L^1_T(\dot{B}^{\frac{d}{p}-1}_{p,1})}+\|(\wt v,\wt \theta)\|^{h,\varepsilon}_{L^\infty_T(\dot{B}^{\frac{d}{p}-2}_{p,1})}+\|(\wt v,\wt \theta)\|^{h,\varepsilon}_{L^1_T(\dot{B}^{\frac{d}{p}}_{p,1})}.
  \end{align}
\end{Prop}

\begin{proof}
\textbf{Step 1: Analysis of $\mathcal{I}_1$.}
In this section we mainly rely on Proposition \ref{prop:PL-low} to control the nonlinearities.
First, we estimate the terms coming from $\wt F$. Using \eqref{PL:lowhigh1}, we have
\begin{align*}
\|v^\varepsilon\cdot\nabla \wt a\|^\ell_{L^1_T(\dot{B}^{\frac{d}{2}-2}_{2,1})} &\lesssim \|\nabla \wt a\|_{L^2_T(\dot{B}^{\frac{d}{p}-2}_{p,1})}\|v^\varepsilon\|_{L^2_T(\dot{B}^{\frac{d}{p}}_{p,1})}+\|\nabla \wt a\|^\ell_{L^1_T(\dot{B}^{\frac{d}{2}-1}_{2,1})}\|v^\varepsilon\|_{L^\infty_T(\dot{B}^{\frac{d}{p}-1}_{p,1})}\\&\quad+\|\nabla \wt a\|^{h,\varepsilon}_{L^1_T(\dot{B}^{\frac{d}{p}-2}_{p,1})}\|v^\varepsilon\|_{L^\infty_T(\dot{B}^{\frac{d}{p}-1}_{p,1})}
\\& \lesssim X^\varepsilon(t) \wt X(t).
\end{align*}
Thanks to \eqref{PL:lowhigh2}, we obtain
\begin{align*}
\|\wt v\cdot\nabla a\|^\ell_{L^1_T(\dot{B}^{\frac{d}{2}-2}_{2,1})} &\lesssim \|\wt v\|_{L^2_T(\dot{B}^{\frac{d}{p}-1}_{p,1})}\|\nabla a\|_{L^2_T(\dot{B}^{\frac{d}{p}-1}_{p,1})} 
\\& \lesssim \wt X(t)X(t).
\end{align*}
Employing \eqref{PL:lowhigh2}, we get
\begin{align*}
&\|\wt a\, \div v^\varepsilon\|^\ell_{L^1_T(\dot{B}^{\frac{d}{2}-2}_{2,1})}  \lesssim  \|\wt a\|_{L^2_T(\dot{B}^{\frac{d}{p}-1}_{p,1})} \|\div v^\varepsilon\|_{L^2_T(\dot{B}^{\frac{d}{p}-1}_{p,1})}
\lesssim \wt X(t)X^\varepsilon(t),
\\&
\| a\,\div \wt v\|^\ell_{L^1_T(\dot{B}^{\frac{d}{2}-2}_{p,1})}  \lesssim \|a\|_{L^\infty_T(\dot{B}^{\frac{d}{p}-1}_{p,1})}\|\div \wt v\|_{L^1_T(\dot{B}^{\frac{d}{p}-1}_{p,1})}
\lesssim \wt X(t)X(t).
\end{align*}
Gathering the above estimates, we obtain
\begin{align*}
\|\wt F\|_{L^1_T(\dot{B}^{\frac{d}{2}-2}_{2,1})}&\lesssim  \wt X(t)(X(t)+X^\varepsilon(t)).
\end{align*}
Next, we estimate the terms coming from $\wt G$. Using \eqref{PL:lowhigh2}, we have
\begin{align*}
&\|\wt v\cdot \nabla v\|_{L^1_T(\dot{B}^{\frac{d}{2}-2}_{2,1})}^\ell \lesssim \|\wt v\|_{L^2_T(\dot{B}^{\frac{d}{p}-1}_{p,1})}\|\nabla v\|_{L^2_T(\dot{B}^{\frac{d}{p}-1}_{p,1})}
\lesssim \wt X(t)X(t),
\\&\|v^\varepsilon\cdot \nabla \wt v\|_{L^1_T(\dot{B}^{\frac{d}{2}-2}_{2,1})}^\ell \lesssim \|v^\varepsilon\|_{L^\infty_T(\dot{B}^{\frac{d}{p}-1}_{p,1})}\|\nabla \wt v\|_{L^1_T(\dot{B}^{\frac{d}{p}-1}_{p,1})} \lesssim \wt X(t)X^\varepsilon(t),
\end{align*}
Applying composition estimates and \eqref{PL:lowhigh2}, we obtain
\begin{align*}
\|\left(J(a^\varepsilon)-J(a)\right)\cA v^\varepsilon/\nu\|_{L^1_T(\dot{B}^{\frac{d}{2}-2}_{2,1})}^\ell& \lesssim \|\wt a\|_{L^2_T(\dot{B}^{\frac{d}{p}-1}_{p,1})}\|\mathcal{A}v^\varepsilon\|_{L^2_T(\dot{B}^{\frac{d}{p}-1}_{p,1})}
\\ &\lesssim  \|\wt a\|_{L^2_T(\dot{B}^{\frac{d}{p}-1}_{p,1})}\|v^\varepsilon\|_{L^2_T(\dot{B}^{\frac{d}{p}+1}_{p,1})}
\\&\lesssim \wt X(t)X^\varepsilon(t)
\end{align*}
and
\begin{align*}
\|(H_1(a^\varepsilon)-H_1(a))\nabla a^\varepsilon\|_{L^1_T(\dot{B}^{\frac{d}{2}-2}_{2,1})}^\ell& \lesssim \|\wt a\|_{L^2_T(\dot{B}^{\frac{d}{p}-1}_{p,1})}\|\nabla a^\varepsilon\|_{L^2_T(\dot{B}^{\frac{d}{p}-1}_{p,1})}
\\ &\lesssim  \|\wt a\|_{L^2_T(\dot{B}^{\frac{d}{p}-1}_{p,1})}\|a^\varepsilon\|_{L^2_T(\dot{B}^{\frac{d}{p}}_{p,1})}
\\&\lesssim \wt X(t)X^\varepsilon(t).
\end{align*}
Using \eqref{PL:lowhigh2}, \eqref{PL:lowhigh3} and  \eqref{PL:lowhigh1}, we obtain
\begin{align*}
&\|H_1(a)\nabla \wt a^\ell\|^\ell_{L^1_T(\dot{B}^{\frac{d}{2}-2}_{2,1})}\leq \|a\|_{L^\infty_T(\dot{B}^{\frac{d}{p}-1}_{p,1})}\|\wt a\|^\ell_{L^1_T(\dot{B}^{\frac{d}{p}}_{p,1})}\lesssim \wt X(t) X(t),
\\
&\|H_1(a)\nabla \wt a^h\|^\ell_{L^1_T(\dot{B}^{\frac{d}{2}-2}_{2,1})}\leq \|\nabla \wt a\|_{L^2_T(\dot{B}^{\frac{d}{p}-2}_{p,1})}\|a\|_{L^2_T(\dot{B}^{\frac{d}{p}}_{p,1})}+\|\nabla \wt a\|^\ell_{L^1_T(\dot{B}^{\frac{d}{2}-1}_{2,1})}\|a\|_{L^\infty_T(\dot{B}^{\frac{d}{p}-1}_{p,1})}\\&\hspace{4.5cm}+\|\nabla \wt a\|^{h,\varepsilon}_{L^1_T(\dot{B}^{\frac{d}{p}-2}_{p,1})}\|a\|_{L^\infty_T(\dot{B}^{\frac{d}{p}-1}_{p,1})}\lesssim \wt X(t) X^\varepsilon(t),
\\&\|(H_2(a^\varepsilon)-H_2(a))\nabla \theta^\varepsilon\|_{L^1_T(\dot{B}^{\frac{d}{2}-2}_{2,1})}^\ell \lesssim \|\wt a\|_{L^2_T(\dot{B}^{\frac{d}{p}-1}_{p,1})}\|\nabla \theta^\varepsilon\|_{L^2_T(\dot{B}^{\frac{d}{p}-1}_{p,1})}\lesssim \wt X(t)X^\varepsilon(t),
\\&
\|H_2(a)\nabla \wt \theta\|_{L^1_T(\dot{B}^{\frac{d}{2}-2}_{2,1})}^\ell \lesssim \|a\|_{L^\infty_T(\dot{B}^{\frac{d}{p}-1}_{p,1})}\|\nabla \wt \theta\|_{L^1_T(\dot{B}^{\frac{d}{p}-1}_{p,1})}
\lesssim \wt X(t)X(t),
\\&
\|\wt \theta\nabla H_3(a^\varepsilon)\|_{L^1_T(\dot{B}^{\frac{d}{2}-2}_{2,1})}^\ell \lesssim \|\wt \theta\|_{L^2_T(\dot{B}^{\frac{d}{p}-1}_{p,1})}\|\nabla a^\varepsilon\|_{L^2_T(\dot{B}^{\frac{d}{p}-1}_{p,1})}\lesssim \wt X(t)X^\varepsilon(t).
\end{align*}
Using \eqref{PL:lowhigh3}, we have
\begin{align*}
\|\theta\nabla (H_3(a^\varepsilon)-H_3(a))\|_{L^1_T(\dot{B}^{\frac{d}{2}-2}_{2,1})}^\ell &\lesssim \|\nabla \wt a\|_{L^\infty_T(\dot{B}^{\frac{d}{p}-2}_{p,1})}\|\theta\|_{L^1_T(\dot{B}^{\frac{d}{p}}_{p,1})}+\|\nabla \wt a\|^\ell_{L^1_T(\dot{B}^{\frac{d}{2}-1}_{2,1})}\|\theta\|_{L^\infty_T(\dot{B}^{\frac{d}{p}-1}_{p,1})}\\&+\|\nabla \wt a\|^{h,\varepsilon}_{L^1_T(\dot{B}^{\frac{d}{p}-2}_{p,1})}\|\theta\|_{L^\infty_T(\dot{B}^{\frac{d}{p}-1}_{p,1})}
\\&\lesssim \wt X(t)X(t).
\end{align*}

Below, we estimate the terms coming from $\wt H$. Again, using \eqref{PL:lowhigh1} and \eqref{PL:lowhigh2}, we obtain
\begin{align*}
&\|\wt v\cdot\nabla\theta^\varepsilon\|_{L^1_T(\dot{B}^{\frac{d}{2}-2}_{2,1})}^\ell \lesssim \|\wt v\|_{L^2_T(\dot{B}^{\frac{d}{p}-1}_{p,1})}\|\nabla\theta^\varepsilon\|_{L^2_T(\dot{B}^{\frac{d}{p}-1}_{p,1})}\lesssim \wt X(t)X^\varepsilon(t),
\\
&\|v\cdot\nabla\wt \theta\|_{L^1_T(\dot{B}^{\frac{d}{2}-2}_{2,1})}^\ell \lesssim \|v\|_{L^\infty_T(\dot{B}^{\frac{d}{p}-1}_{p,1})}\|\nabla\wt\theta\|_{L^1_T(\dot{B}^{\frac{d}{p}-1}_{p,1})}\lesssim \wt X(t)X(t).
\end{align*}
The product law \eqref{PL:lowhigh2} together with composition estimates, we have
\begin{align*}
\|(J(a^\varepsilon)-J(a))\div q^\varepsilon\|_{L^1_T(\dot{B}^{\frac{d}{2}-2}_{2,1})}^\ell& \lesssim \|J(a^\varepsilon)-J(a)\|_{L^2_T(\dot{B}^{\frac{d}{p}-1}_{p,1})}\|\div q^\varepsilon\|_{L^2_T(\dot{B}^{\frac{d}{p}-1}_{p,1})}
\\ &\lesssim  \|\wt a\|_{L^2_T(\dot{B}^{\frac{d}{p}-1}_{p,1})}\|q^\varepsilon\|_{L^2_T(\dot{B}^{\frac{d}{p}}_{p,1})}
\\&\lesssim \wt X(t)X^\varepsilon(t).
\end{align*}
Since
\begin{align*}
\|J(a^\varepsilon)(\div q^\varepsilon+\Delta \theta)\|_{L^1_T(\dot{B}^{\frac{d}{2}-2}_{2,1})}^\ell& \lesssim  \|J(a^\varepsilon)\div Q^\varepsilon\|_{L^1_T(\dot{B}^{\frac{d}{p}-2}_{p,1})}^\ell+\|J(a^\varepsilon)\Delta \wt \theta\|_{L^1_T(\dot{B}^{\frac{d}{p}-2}_{p,1})}^\ell,
\end{align*}
we have
\begin{align*} \|J(a^\varepsilon)\div Q^\varepsilon\|_{L^1_T(\dot{B}^{\frac{d}{p}-2}_{p,1})}^\ell&\lesssim \|\div Q^\varepsilon\|_{L^1_T(\dot{B}^{\frac{d}{p}-2}_{p,1})}\|a^\varepsilon\|_{L^\infty_T(\dot{B}^{\frac{d}{p}}_{p,1})}+\|\div Q^\varepsilon\|^\ell_{L^1_T(\dot{B}^{\frac{d}{2}-1}_{2,1})}\|a^\varepsilon\|_{L^\infty_T(\dot{B}^{\frac{d}{p}-1}_{p,1})}\\&+\|\div Q^\varepsilon\|^{h,\varepsilon}_{L^1_T(\dot{B}^{\frac{d}{p}-2}_{p,1})}\|a^\varepsilon\|_{L^\infty_T(\dot{B}^{\frac{d}{p}-1}_{p,1})}
\\&\lesssim \varepsilon X^{\varepsilon}(t)
\end{align*}
and 
\begin{align*} \|J(a^\varepsilon)\Delta \wt\theta\|_{L^1_T(\dot{B}^{\frac{d}{p}-2}_{p,1})}^\ell&\lesssim \|\Delta \wt \theta\|_{L^1_T(\dot{B}^{\frac{d}{p}-2}_{p,1})}\|a^\varepsilon\|_{L^\infty_T(\dot{B}^{\frac{d}{p}}_{p,1})}+\|\Delta\wt \theta\|^\ell_{L^1_T(\dot{B}^{\frac{d}{2}-1}_{2,1})}\|a^\varepsilon\|_{L^\infty_T(\dot{B}^{\frac{d}{p}-1}_{p,1})}\\&+\|\Delta\wt \theta\|^{h,\varepsilon}_{L^1_T(\dot{B}^{\frac{d}{p}-2}_{p,1})}\|a^\varepsilon\|_{L^\infty_T(\dot{B}^{\frac{d}{p}-1}_{p,1})}
\\&\lesssim \wt X(t) X^{\varepsilon}(t).
\end{align*}
Finally, it is easy to show that
$\|\wt R\|_{L^1_T(\dot{B}^{\frac{d}{2}-2}_{2,1})}^\ell\lesssim \wt X(t)X^\varepsilon(t)$
where 
$\wt R= \dfrac{N(\nabla v^\varepsilon,\nabla v^\varepsilon)}{1+a^\varepsilon}-\dfrac{N(\nabla v,\nabla v)}{1+a}+\wt H_1(a^\varepsilon)\theta^\varepsilon \div v^\varepsilon-\wt H_1( a)\theta \div v.$

\bigbreak
\noindent\noindent\textbf{Step 2: Analysis of $\mathcal{I}_2$}.
First, we estimate the terms coming from $\wt F$. We have
\begin{align*}
&\|\wt v\cdot\nabla a\|_{L^1_T(\dot{B}^{\frac{d}{p}-1}_{p,1})} \lesssim \|\wt v\|_{L^1_T(\dot{B}^{\frac{d}{p}}_{p,1}\cap L^\infty)} \|\nabla a\|_{L^\infty_T(\dot{B}^{\frac{d}{p}-1}_{p,1})}\lesssim \wt X(t)X(t),
\\&
\|\wt a\, \div v^\varepsilon\|_{L^1_T(\dot{B}^{\frac{d}{p}-1}_{p,1})}  \lesssim  \|\wt a\|_{L^\infty_T(\dot{B}^{\frac{d}{p}-1}_{p,1})} \|\div v^\varepsilon\|_{L^1_T(\dot{B}^{\frac{d}{p}}_{p,1})}\lesssim X^\varepsilon(t)\wt X(t),
\\&\| a\,\div \wt v\|_{L^1_T(\dot{B}^{\frac{d}{p}-1}_{p,1})} \lesssim \|a\|_{L^\infty_T(\dot{B}^{\frac{d}{p}}_{p,1})}\|\div \wt v\|_{L^1_T(\dot{B}^{\frac{d}{p}-1}_{p,1})} \lesssim \wt X(t)X(t).
\end{align*}
Gathering the above estimates, we obtain
\begin{align*}
\|\wt F\|_{L^1_T(\dot{B}^{\frac{d}{p}-1}_{p,1})}&\lesssim  \wt X(t)(X(t)+X^\varepsilon(t)). 
\end{align*}

\medbreak

\noindent Then, we deal with $\wt G$.
We have
\begin{align*}
&\|\wt v\cdot \nabla v\|_{L^1_T(\dot{B}^{\frac{d}{p}-2}_{p,1})}^h \lesssim \|\wt v\|_{L^2_T(\dot{B}^{\frac{d}{p}-1}_{p,1})}\|\nabla v\|_{L^2_T(\dot{B}^{\frac{d}{p}-1}_{p,1})}\lesssim \wt X(t)X(t),
\\&\|v^\varepsilon\cdot \nabla \wt v\|_{L^1_T(\dot{B}^{\frac{d}{p}-2}_{p,1})}^h\lesssim \|v^\varepsilon\|_{L^\infty_T(\dot{B}^{\frac{d}{p}-1}_{p,1})}\|\nabla \wt v\|_{L^1_T(\dot{B}^{\frac{d}{p}-1}_{p,1})}
\lesssim \wt X(t)X^\varepsilon(t).
\end{align*}
Using product laws and the composition proposition \ref{prop:comphf}, we have
\begin{align*}
\|\left(J(a^\varepsilon)-J(a)\right)\cA v^\varepsilon/\nu\|_{L^1_T(\dot{B}^{\frac{d}{p}-2}_{p,1})}^h& \lesssim \|\wt a\|_{L^\infty_T(\dot{B}^{\frac{d}{p}-1}_{p,1})}\|\mathcal{A}v^\varepsilon\|_{L^1_T(\dot{B}^{\frac{d}{p}-1}_{p,1})}
\\ &\lesssim  \|\wt a\|_{L^\infty_T(\dot{B}^{\frac{d}{p}-1}_{p,1})}\|v^\varepsilon\|_{L^1_T(\dot{B}^{\frac{d}{p}+1}_{p,1})}
\\&\lesssim \wt X(t)X^\varepsilon(t).
\end{align*}
Similarly, 
\begin{align*}
\|J(a)\cA\wt v/\nu\|_{L^1_T(\dot{B}^{\frac{d}{p}-2}_{p,1})}^h& \lesssim \|a\|_{L^\infty_T(\dot{B}^{\frac{d}{p}}_{p,1})}\|\mathcal{A}\wt v\|_{L^1_T(\dot{B}^{\frac{d}{p}-2}_{p,1})}\lesssim \wt X(t)X(t),
\end{align*}
and
\begin{align*}
&\|(H_1(a^\varepsilon)-H_1(a))\nabla a^\varepsilon\|_{L^1_T(\dot{B}^{\frac{d}{p}-2}_{p,1})}^h \lesssim \|\wt a\|_{L^2_T(\dot{B}^{\frac{d}{p}-1}_{p,1})}\|\nabla a^\varepsilon\|_{L^2_T(\dot{B}^{\frac{d}{p}-1}_{p,1})}\lesssim \wt X(t)X^\varepsilon(t),
\\&\|H_1(a)\nabla \wt a\|_{L^1_T(\dot{B}^{\frac{d}{p}-2}_{p,1})}^h \lesssim \|a\|_{L^2_T(\dot{B}^{\frac{d}{p}}_{p,1})}\|\nabla \wt a\|_{L^2_T(\dot{B}^{\frac{d}{p}-2}_{p,1})}\lesssim \wt X(t)X(t),
\\&
\|(H_2(a^\varepsilon)-H_2(a))\nabla \theta^\varepsilon\|_{L^1_T(\dot{B}^{\frac{d}{p}-2}_{p,1})}^h \lesssim \|\wt a\|_{L^2_T(\dot{B}^{\frac{d}{p}-1}_{p,1})}\|\nabla \theta^\varepsilon\|_{L^2_T(\dot{B}^{\frac{d}{p}-1}_{p,1})}\lesssim \wt X(t)X^\varepsilon(t),
\\&
\|H_2(a)\nabla \wt \theta\|_{L^1_T(\dot{B}^{\frac{d}{p}-2}_{p,1})}^h \lesssim \|a\|_{L^2_T(\dot{B}^{\frac{d}{p}}_{p,1})}\|\nabla \wt \theta\|_{L^2_T(\dot{B}^{\frac{d}{p}-2}_{p,1})}\lesssim \wt X(t)X(t),
\\&
\|\wt \theta\nabla H_3(a^\varepsilon)\|_{L^1_T(\dot{B}^{\frac{d}{p}-2}_{p,1})}^h \lesssim \|\wt \theta\|_{L^2_T(\dot{B}^{\frac{d}{p}-1}_{p,1})}\|\nabla a^\varepsilon\|_{L^2_T(\dot{B}^{\frac{d}{p}-1}_{p,1})} \lesssim \wt X(t)X^\varepsilon(t).
\end{align*}
Decomposing $\theta=\theta^\ell+\theta^h$, we obtain
\begin{align*}
\| \theta\nabla (H_3(a^\varepsilon)-H_3(a))\|_{L^1_T(\dot{B}^{\frac{d}{p}-2}_{p,1})}^h
& \lesssim\| \theta^h\nabla (H_3(a^\varepsilon)-H_3(a))\|_{L^1_T(\dot{B}^{\frac{d}{p}-2}_{p,1})}^h+\| \theta^\ell\nabla (H_3(a^\varepsilon)-H_3(a))\|_{L^1_T(\dot{B}^{\frac{d}{p}-2}_{p,1})}^h
\\& \lesssim \|\theta\|^{h,\varepsilon}_{L^1_T(\dot{B}^{\frac{d}{p}}_{p,1})}\|\nabla \wt a\|_{L^\infty_T(\dot{B}^{\frac{d}{p}-2}_{p,1})}+\|\theta\|^\ell_{L^2_T(\dot{B}^{\frac{d}{p}}_{p,1})}\|\nabla \wt a\|_{L^2_T(\dot{B}^{\frac{d}{p}-2}_{p,1})}
\\ &\lesssim  \|\theta\|^{h,\varepsilon}_{L^1_T(\dot{B}^{\frac{d}{p}}_{p,1})}\|\wt a\|_{L^\infty_T(\dot{B}^{\frac{d}{p}-1}_{p,1})}+\|\theta\|^\ell_{L^2_T(\dot{B}^{\frac{d}{p}}_{p,1})}\|\wt a\|_{L^2_T(\dot{B}^{\frac{d}{p}-1}_{p,1})}
\\&\lesssim \wt X(t)X(t).
\end{align*}
Finally, we estimate the terms coming from $\wt H$. We have
\begin{align*}
&\|\wt v\cdot\nabla\theta^\varepsilon\|_{L^1_T(\dot{B}^{\frac{d}{p}-2}_{p,1})}^h \lesssim \|\wt v\|_{L^2_T(\dot{B}^{\frac{d}{p}-1}_{p,1})}\|\nabla\theta^\varepsilon\|_{L^2_T(\dot{B}^{\frac{d}{p}-1}_{p,1})}\lesssim \wt X(t)X^\varepsilon(t),
\\&
\|v\cdot\nabla\wt \theta\|_{L^1_T(\dot{B}^{\frac{d}{p}-2}_{p,1})}^h \lesssim \|v\|_{L^2_T(\dot{B}^{\frac{d}{p}}_{p,1})}\|\nabla\wt\theta\|_{L^2_T(\dot{B}^{\frac{d}{p}-2}_{p,1})} \lesssim \wt X(t)X(t).
\end{align*}
The composition inequality \ref{prop:comphf} coupled with product laws gives
\begin{align*}
\|(J(a^\varepsilon)-J(a))\div q^\varepsilon\|_{L^1_T(\dot{B}^{\frac{d}{p}-2}_{p,1})}^h& \lesssim \|J(a^\varepsilon)-J(a)\|_{L^2_T(\dot{B}^{\frac{d}{p}-1}_{p,1})}\|\div q^\varepsilon\|_{L^1_T(\dot{B}^{\frac{d}{p}-1}_{p,1})}
\\ &\lesssim  \|\wt a\|_{L^\infty_T(\dot{B}^{\frac{d}{p}-1}_{p,1})}\|q^\varepsilon\|_{L^1_T(\dot{B}^{\frac{d}{p}}_{p,1})}
\\&\lesssim \wt X(t)X^\varepsilon(t)
\end{align*}
and
\begin{align*}
\|J(a^\varepsilon)(\div q^\varepsilon+\Delta \theta)\|_{L^1_T(\dot{B}^{\frac{d}{p}-2}_{p,1})}^h& \lesssim  \|J(a^\varepsilon)(\div q^\varepsilon+\Delta \theta^\varepsilon)\|_{L^1_T(\dot{B}^{\frac{d}{p}-2}_{p,1})}^h+\|J(a^\varepsilon)\Delta \wt \theta\|_{L^1_T(\dot{B}^{\frac{d}{p}-2}_{p,1})}^h
\\&\lesssim \|a^\varepsilon\|_{L^\infty_T(\dot{B}^{\frac{d}{p}}_{p,1})}\| Q^\varepsilon\|_{L^1_T(\dot{B}^{\frac{d}{p}-1}_{p,1})}+\|a^\varepsilon\|_{L^\infty_T(\dot{B}^{\frac{d}{p}}_{p,1})}\|\wt \theta\|_{L^1_T(\dot{B}^{\frac{d}{p}}_{p,1})}
\\&\lesssim \varepsilon X^\varepsilon(t)+\wt X(t)X^\varepsilon(t).
\end{align*}
Then, it is easy to show that
\begin{align*}
\|\wt R\|_{L^1_T(\dot{B}^{\frac{d}{p}-2}_{p,1})}^h\lesssim \wt X(t)X^\varepsilon(t).
\end{align*}
\textbf{Step 3: Conclusion of the proof.} Gathering the estimates of Step 1 and Step 2, the proof of Proposition \ref{prop:relax} is complete.
\end{proof}

\section{Extensions and open problems} \label{sec:ext}
In this work, we justified rigorously the relaxation relation between the Navier-Stokes-Cattaneo-Christov system \eqref{Cattaneo-NSC} and the Navier-Stokes-Fourier system \eqref{FullNSC}. Our analysis opens up several possible extensions and problems. We discuss some of them below.

\begin{enumerate}
    \item[1.] \emph{The fully hyperbolic Navier-Stokes system}. We expect our method to be able to treat a fully hyperbolic version of the Navier-Stokes system. Replacing the constitutive law for a Newtonian fluid 
     \begin{align}
       \tau=2\mu D(u) I_d +\lambda\div u\,\Id
     \end{align}
    where $\tau$ is the stress tensor, by the Maxwell's relation
    \begin{align}
        \label{MaxwellLaw}
        \varepsilon^2_{2}(\partial_t \tau+u\cdot\nabla \tau+g(\tau, \nabla u))+\tau=2\mu D(u) I_d +\lambda\div u\,\Id
    \end{align} for a relaxation parameter $\varepsilon_2>0$, where $g(\tau, \nabla u)\triangleq\tau W(u)-W(u)\tau$ and $W(u)$ is the skew-symmetric part of $\nabla u$, namely, $W(u)=\frac12(\nabla u\!-\!{}^T\!\nabla u)$. To deal with such a system, and show that \eqref{Cattaneo-NSC} with the law \eqref{MaxwellLaw} converges as $\varepsilon_1,\varepsilon_2\to0$ toward the Navier-Stokes Fourier system \eqref{FullNSC} one would need to consider an additional frequency-threshold $J_{\varepsilon_2}=1/\varepsilon_2$ and distinguish four frequency-regimes instead of three. For more information on fully hyperbolic Navier-Stokes systems, see \cite{HuRacke2017,PengZhao2022} and references therein.
\item[2.] \emph{Two-dimensional hyperbolic Navier-Stokes systems.} Due to technical limitations in the product, composition and commutator laws, we are restricted to the case $d\geq3$ in the analysis presented here. It would be interesting to develop a method for the $d=2$ case. To this end, one could adapt the Lagrangian analysis used in \cite{ChikamiDanchin} to the present framework.
\item[3.]  \emph{A complete hyperbolic structure}.
As demonstrated by Angeles in \cite{Angeles22}, the inviscid form of system \eqref{Cattaneo-NSC} lacks hyperbolicity, posing a challenge in establishing the well-posedness, particularly for the Euler-Cattaneo system.
To address this issue, we identify two potential approaches. One option is to utilize the modified Cattaneo-Christov law introduced in \cite{FA2023}, which renders the system hyperbolic. Alternatively, we could follow the methodology introduced by Dhaouadi and Gavrilyuk \cite{DhaouadiGavrilyuk}, who recently proposed a purely hyperbolic way of modelling heat transfer with a finite speed of propagation. In both approaches, the relaxation structure is similar to the one studied in the present paper and achieving the strong relaxation limit associated with their hyperbolic heat transfer can be done using the methodology we developed.
    \end{enumerate}
\bigbreak
\bigbreak
\noindent \textbf{Acknowledgments}\, The authors are grateful to E. Zuazua for his
valuable comments on the manuscript. TCB is supported by the Alexander von Humboldt-Professorship program and the Deutsche
Forschungsgemeinschaft (DFG, German Research Foundation) under project C07 of the
Sonderforschungsbereich/Transregio 154 ``Mathematical Modelling, Simulation and
Optimization using the Example of Gas Networks" (project ID: 239904186).
SK is partially supported by JSPS KAKENHI Grant Numbers JP23H01085,
JP19H05597 and JP20H00118.
JX is supported by the National Natural Science Foundation of China (12271250, 12031006). Part of this work was done during the visit of J. Xu at RIMS, Kyoto University.
 \medbreak
\noindent \textbf{Data availability statement} \,
 Data sharing not applicable to this article as no data sets were generated or analyzed during the current study.
 
\bigbreak
\noindent \textbf{Declarations}
 \medbreak
\noindent \textbf{Conflicts of interest} \, The authors have no competing interests to declare that are relevant to the content of this
article.

\appendix
\section{Reformulation of the system}\label{sec:appreform}
We adapt the reformulation done in \cite{DanchinFullNSC,DanchinXuFullNSC}. Recall that the Navier-Stokes-Cattaneo-Christov system \eqref{Cattaneo-NSC} reads \begin{equation}\label{Cattaneo-NSCT0} 
\left\{
\begin{array}
[c]{l}%
\d_t\rho^\varepsilon+\div(\rho^\varepsilon u^\varepsilon)=0,\\
\rho^\varepsilon(\d_t u^\varepsilon+u\cdot \nabla u)+\nabla (T^\varepsilon \pi(\rho^\varepsilon))= \div\tau^\varepsilon,\\
\rho^\varepsilon C_v(\d_tT^\varepsilon+u\cdot\nabla T^\varepsilon)+T^\varepsilon\pi(\rho^\varepsilon)\div u^\varepsilon+\div \mathfrak{q}^\varepsilon=\div(\tau^\varepsilon \cdot u^\varepsilon),\\
\varepsilon^2(\d_tq^\varepsilon+u^\varepsilon\cdot\nabla \mathfrak{q}^\varepsilon-\mathfrak{q}^\varepsilon\cdot\nabla u^\varepsilon+\mathfrak{q}^\varepsilon\div u^\varepsilon)+\mathfrak{q}^\varepsilon+\kappa\nabla T^\varepsilon=0.\\
\end{array}
\right.
\end{equation} 
Let $\bar{\rho}>0$ and $\bar{T}>0$, linearizing the system \eqref{Cattaneo-NSCT0} around the constant equilibrium $$(\bar{\rho},\bar{u},\bar{T},\bar{q})=(\bar{\rho},0,\bar{T},0)$$ and setting $\mathfrak{a}^\varepsilon=\dfrac{\rho^\varepsilon-\bar{\rho}}{\bar{\rho}}$ and $\vartheta^\varepsilon=T^\varepsilon-\bar{T}$, we obtain
\begin{equation*}
\left\{
\begin{array}
[c]{l}%
\d_t\mathfrak{a}^\varepsilon+\div u^\varepsilon=-\div(\mathfrak{a}^\varepsilon u^\varepsilon),\\
\d_tu^\varepsilon +u^\varepsilon\cdot\nabla u^\varepsilon-\dfrac{\widetilde{\mathcal{A}}u^\varepsilon}{\bar{\rho}(1+\mathfrak{a}^\varepsilon)}+\dfrac{\pi'(\bar{\rho}(1+\mathfrak{a}^\varepsilon))\bar{T}}{1+\mathfrak{a}^\varepsilon}\nabla \mathfrak{a}^\varepsilon+\dfrac{\pi(\bar{\rho}(1+\mathfrak{a}^\varepsilon))}{\bar{\rho}(1+\mathfrak{a}^\varepsilon)}\nabla \vartheta^\varepsilon+\dfrac{\pi'(\bar{\rho}(1+\mathfrak{a}^\varepsilon))}{1+\mathfrak{a}^\varepsilon}\vartheta^\varepsilon\nabla \mathfrak{a}^\varepsilon=0,\\
\d_t \vartheta^\varepsilon+u^\varepsilon\cdot\nabla \vartheta^\varepsilon+(\bar{T}+\vartheta^\varepsilon)\dfrac{\pi(\bar{\rho}(1+\mathfrak{a}^\varepsilon))}{\bar{\rho}C_v(1+\mathfrak{a}^\varepsilon)}\div u^\varepsilon+\dfrac{\div \mathfrak{q}^\varepsilon}{\bar{\rho}C_v(1+\mathfrak{a}^\varepsilon)}=\dfrac{2\mu|Du^\varepsilon|^2+\lambda(\div u^\varepsilon)^2}{\bar{\rho}C_v(1+\mathfrak{a}^\varepsilon)},\\ \smallbreak
\varepsilon^2(\d_t\mathfrak{q}^\varepsilon+u^\varepsilon\cdot\nabla \mathfrak{q}^\varepsilon-\mathfrak{q}^\varepsilon\cdot\nabla u^\varepsilon+\mathfrak{q}^\varepsilon\div u^\varepsilon)+\mathfrak{q}^\varepsilon+\kappa\nabla\vartheta^\varepsilon=0.\\
\end{array}
\right.
\label{TempNSC}
\end{equation*}
where $\widetilde{\mathcal{A}}u^\varepsilon=\mu\Delta u^\varepsilon+(\mu+\lambda) \nabla\div u^\varepsilon$. Denoting $\nu:=\lambda + 2\mu$, $\bar{\nu}:=\nu/\bar{\rho}$, $\chi_0:=\d_\rho P(\bar{\rho},\bar{T})^{-1/2}$ and performing the change of unknowns
$$
a^\varepsilon(t,x)=\mathfrak{a}^\varepsilon(\bar{\nu}\chi_0^2t,\bar{\nu}\chi_0x),\quad v^\varepsilon(t,x)=\chi_0u^\varepsilon(\bar{\nu}\chi_0^2t,\bar{\nu}\chi_0x),
$$
$$\theta^\varepsilon(t,x)=\chi_0\sqrt{\dfrac{C_v}{\bar{T}}}\vartheta^\varepsilon(\bar{\nu}\chi_0^2t,\bar{\nu}\chi_0x)\andf q^\varepsilon(t,x)=\sqrt{\dfrac{C_v}{\bar{T}}}\mathfrak{q}^\varepsilon(\bar{\nu}\chi_0^2t,\bar{\nu}\chi_0x), $$
we arrive at
\begin{equation}
\left\{
\begin{array}
[c]{l}%
\d_ta^\varepsilon+\div v^\varepsilon=F^\varepsilon,\\
\d_t v^\varepsilon-\mathcal{A}v^\varepsilon+\nabla a^\varepsilon +\gamma\nabla \theta^\varepsilon=G^\varepsilon,\\
\d_t \theta^\varepsilon+\beta\div q^\varepsilon+\gamma\div v^\varepsilon=H^\varepsilon,\\
\varepsilon^2\d_tq^\varepsilon+\alpha q^\varepsilon+\kappa\nabla\theta^\varepsilon=\var^2I^\varepsilon,
\end{array}
\right.
\label{LinearHypNSCzero0}
\end{equation}
with $\alpha=\bar{\nu}\chi_0^2$, $\gamma=\dfrac{\chi_0}{\bar{\rho}}\sqrt{\dfrac{\bar{T}}{C_v}}\pi(\bar{\rho})$, $\beta=\dfrac{\chi_0^2}{\bar{\rho}C_v}$  and the nonlinear source terms are given by
\begin{align*}
   & F^\varepsilon=-\div(a^\varepsilon v^\varepsilon), \:\: G^\varepsilon=-v^\varepsilon\cdot\nabla v^\varepsilon -J(a^\varepsilon)\mathcal{A}v^\varepsilon/\nu-H_1(a^\varepsilon)\nabla a^\varepsilon-H_2(a^\varepsilon)\nabla \theta^\varepsilon-\theta\nabla H_3(a^\varepsilon),
    \\ &H^\varepsilon=-v^\varepsilon\cdot\nabla\theta^\varepsilon+J(a^\varepsilon)\div q^\varepsilon+\dfrac{N(\nabla v^\varepsilon,\nabla v^\varepsilon)}{1+a^\varepsilon}-\wt H_1(a^\varepsilon)\theta^\varepsilon \div v^\varepsilon,
    \\ &I^\varepsilon=-v^\varepsilon\cdot\nabla q^\varepsilon +q^\varepsilon\cdot\nabla v^\varepsilon- q^\varepsilon\div v^\varepsilon,
\end{align*}
where $J(a^\varepsilon)=\dfrac{a^\varepsilon}{1+a^\varepsilon}$ and the $H_i$ are smooth functions, written explicitly in \cite{DanchinFullNSC,DanchinXuFullNSC}, such that $H_1(0)=H_2(0)=H_3(0)=\wt H_1(0)=0$.

\section{Classical lemmas and Harmonic analysis}
\subsection{Classical lemmas}
We often used the following well-known result  (see e.g. \cite{CBD1} for its proof). 
\begin{Lemme}\label{SimpliCarre}
Let  $p\geq 1$ and $X : [0,T]\to \mathbb{R}^+$ be a continuous function such that $X^p$ is a.e. differentiable. We assume that there exist  a constant $b\geq 0$ and  a measurable function $A : [0,T]\to \mathbb{R}^+$ 
such that 
 $$\frac{1}{p}\frac{d}{dt}X^p+bX^p\leq AX^{p-1}\quad\hbox{a.e.  on }\ [0,T].$$ 
 Then, for all $t\in[0,T],$ we have
$$X(t)+b\int_0^tX\leq X_0+\int_0^tA.$$
\end{Lemme}

\subsection{Harmonic analysis tools}

We consider the Cauchy problem for the damped heat equation
\begin{equation}
\left\{
\begin{aligned}
&\partial_t u- c_{1} \Delta u+c_2 u=f,\quad x\in\mathbb{R}^{d},\quad t>0,\\
&u(0, x)=u_0(x),\quad\quad\quad~~~  x\in\mathbb{R}^{d},
\end{aligned}
\right.\label{Heat}
\end{equation}
where $c_1\geq 0$ and $ c_2 \geq 0$. The following lemma gives estimates for equations of the form \eqref{Heat}.
\begin{Lemme}\label{maximalL1L2}
Let $s\in\mathbb{R}$, $p\geq 2$, $T>0$ be given time, and $c_{i}\geq0$ $(i=1,2)$ be positive constants. Assume $u_{0}\in\dot{B}^{s}_{p,1}$ and $f\in L^1(0,T;\dot{B}^{s}_{p,1})$. If $u$ is the solution to the Cauchy problem \eqref{Heat} for $t\in(0,T)$, then $u$ satisfies
\begin{equation}\label{maximal11}
\begin{aligned}
&\|u\|_{\widetilde{L}^{\infty}_{t}(\dot{B}^{s}_{p,1})}+c_{1}\|u\|_{L^1_{t}(\dot{B}^{s+2}_{p,1})}+c_{2}\|u\|_{L^1_{t}(\dot{B}^{s}_{p,1})}\lesssim \|u_{0}\|_{\dot{B}^{s}_{p,1}}+\|f\|_{L^1_{t}(\dot{B}^{s}_{p,1})},\quad t\in(0,T),
\end{aligned}
\end{equation}
and for $c_2>0$,
\begin{equation}\label{maximal22}
\begin{aligned}
&\|u\|_{\widetilde{L}^{\infty}_{t}(\dot{B}^{s})}+\sqrt{c_{1}}\|u\|_{\widetilde{L}^2_{t}(\dot{B}^{s+1}_{p,1})} \leq C(\|u_{0}\|_{\dot{B}^{s}_{p,1}}+\frac{1}{\sqrt{c_{1}}}\|f_{2}\|_{\widetilde{L}^2_{t}(\dot{B}^{s-1}_{p,1})})\quad t\in(0,T),
\end{aligned}
\end{equation}
where $C>0$ is a constant independent of $c_1$.
\end{Lemme}
To deal with nonlinearities in our hybrid $L^2-L^p$ framework with need special product laws. First, we state high-frequency product laws which improves the oen derived in \cite{CBD3}.
\begin{Prop}[High frequencies product law] \label{LPP} 
Let  $2\leq p\leq 4$ and   $p^*\triangleq 2p/(p-2).$ 
For all $s>0$, we have
\begin{align}\label{eq:prod4}
&\|ab\|^{h,\varepsilon}_{\dot B^{s}_{2,1}}\lesssim  \|a\|_{\dot B^{\frac dp}_{p,1}}\|b\|^{h,\varepsilon}_{\dot B^{s}_{2,1}}
+\|b\|_{\dot B^{\frac dp}_{p,1}}\|a\|^{h,\varepsilon}_{\dot B^{s}_{2,1}}+ \|a\|^{\ell,\varepsilon}_{\dot B^{\frac dp}_{p,1}}\|b\|^{\ell,\varepsilon}_{\dot B^{s+\frac{d}{p}-\frac{d}{2}}_{p,1}}
+ \|b\|^{\ell,\varepsilon}_{\dot B^{\frac dp}_{p,1}}\|a\|^{\ell,\varepsilon}_{\dot B^{s+\frac{d}{p}-\frac{d}{2}}_{p,1}},
\\&\label{eq:prod42}
\|ab\|^{h,\varepsilon}_{\dot B^{s}_{2,1}}\lesssim  \|a\|_{\dot B^{\frac dp-1}_{p,1}}\|b\|^{h,\varepsilon}_{\dot B^{s+1}_{2,1}}
+\|b\|_{\dot B^{\frac dp-1}_{p,1}}\|a\|^{h,\varepsilon}_{\dot B^{s+1}_{2,1}}+ \|a\|^{\ell,\varepsilon}_{\dot B^{\frac dp}_{p,1}}\|b\|^{\ell,\varepsilon}_{\dot B^{s+\frac{d}{p}-\frac{d}{2}}_{p,1}}
+ \|b\|^{\ell,\varepsilon}_{\dot B^{\frac dp}_{p,1}}\|a\|^{\ell,\varepsilon}_{\dot B^{s+\frac{d}{p}-\frac{d}{2}}_{p,1}},
\\&\label{eq:prod422}
\|ab\|^{h,\varepsilon}_{\dot B^{s}_{2,1}}\lesssim  \|a\|_{\dot B^{\frac dp}_{p,1}}\|b\|^{h,\varepsilon}_{\dot B^{s}_{2,1}}
+\|b\|_{\dot B^{\frac dp-1}_{p,1}}\|a\|^{h,\varepsilon}_{\dot B^{s+1}_{2,1}}+ \|a\|^{\ell,\varepsilon}_{\dot B^{\frac dp}_{p,1}}\|b\|^{\ell,\varepsilon}_{\dot B^{s+\frac{d}{p}-\frac{d}{2}}_{p,1}}
+ \|b\|^{\ell,\varepsilon}_{\dot B^{\frac dp}_{p,1}}\|a\|^{\ell,\varepsilon}_{\dot B^{s+\frac{d}{p}-\frac{d}{2}}_{p,1}}.
\end{align}
For $s=d/2$, we have
\begin{align}
    \label{eq:prod433}
   \|ab\|^{h,\varepsilon}_{\dot B^{\frac{d}{2}}_{2,1}}\lesssim   \|a\|_{\dot{B}^{\frac{d}{p}}_{p,1}}\|(b^{h,\varepsilon},b^{\ell})\|_{ \dot{B}^{\frac{d}{2}}_{2,1}}+\|b\|^{m,\varepsilon}_{\dot{B}^{\frac{d}{p}-1}_{p,1}}\|a\|_{ \dot{B}^{\frac{d}{p}}_{p,1}}+\|a\|^{\ell,\varepsilon}_{\dot{B}^{\frac{d}{p}}_{p,1}}\|b\|^{m,\varepsilon}_{\dot{B}^{\frac{d}{p}}_{p,1}}.
\end{align}
\end{Prop}
\begin{proof} 
We recall the so-called Bony decomposition 
(first introduced by Bony in \cite{Bony})
for the product 
of two tempered distributions $f$ and $g$:
$$ab=T_ab+T'_ba\with T_ab\triangleq\sum_{j\in\Z}\dot S_{j-1}a\,\ddj b\andf
T'_ba\triangleq \sum_{j\in\Z}\dot S_{j+2}b\,\ddj a.$$
Using this decomposition 
 and further splitting  $a$ and $b$ into low and high frequencies,
 we get
 $$ ab=T'_{b}a^{h,\varepsilon}+T_{a}b^{h,\varepsilon}+T'_{\dot{B}^{h,\varepsilon}}a^{\ell,\varepsilon}+T_{a^{h,\varepsilon}}b^{\ell,\varepsilon}+a^{\ell,\varepsilon} b^{\ell,\varepsilon}.$$
 All the terms in the right-hand side, except for the last one, may be bounded following the computations done in \cite{CBD3}. We have
 \begin{align}
\label{eq:prodx12}
&\|T'_{b}a^{h,\varepsilon}\|_{\dot B^{s}_{2,1}}\lesssim\|b\|_{L^\infty}\|a\|^{h,\varepsilon}_{\dot B^{s}_{2,1}},
\\&\|T_{a}b^{h,\varepsilon}\|_{\dot B^{s}_{2,1}}\lesssim\|a\|_{L^\infty}\|b\|^{h,\varepsilon}_{\dot B^{s}_{2,1}}. \label{eq:prodx13}
\end{align} 
Since $a^{\ell,\varepsilon}=\dot S_{J_1+1}a$ and $b^{h,\varepsilon}=({\rm Id}-\dot S_{J_1+1})b,$ we see that 
$$
T'_{\dot{B}^{h,\varepsilon}} a^{\ell,\varepsilon}=  \dot S_{J_1+2}b^{h,\varepsilon}\,\dot\Delta_{J_1+1} a^{\ell,\varepsilon}.$$
Consequently, as 
$\dot S_{J_1+2}b^{h,\varepsilon}=(\dot\Delta_{J_1-1}+\dot\Delta_{J_1}+\dot\Delta_{J_1+1})b^{h,\varepsilon},$
$$\|T'_{\dot{B}^{h,\varepsilon}} a^{\ell,\varepsilon}\|_{\dot B^{s}_{2,1}}\lesssim
\|\dot S_{J_1+2}b^{h,\varepsilon}\|_{L^2}\|\dot\Delta_{J_1+1}a^{\ell,\varepsilon}\|_{L^\infty} \lesssim \|b\|^{h,\varepsilon}_{\dot B^{s}_{2,1}}\|a\|_{L^\infty} 
 .$$
 Similarly, we have
 $$\|T_{a^{h,\varepsilon}} b^{\ell,\varepsilon}\|_{\dot B^{s}_{2,1}}\lesssim \|a\|^{h,\varepsilon}_{\dot B^{s}_{2,1}}\|b^{\ell,\varepsilon}\|_{L^\infty} 
 .$$
Next, we deal with the term $a^{\ell,\varepsilon} b^{\ell,\varepsilon}$ as in \cite{XuZhang2023}. We have
\begin{align*}
\|a^{\ell,\varepsilon} b^{\ell,\varepsilon}\|_{\dot B^{s}_{2,1}}^{h,\varepsilon}&\lesssim
\|\nabla (a^{\ell,\varepsilon} b^{\ell,\varepsilon})\|_{\dot B^{s-1}_{2,1}}^{h,\varepsilon}
\\& \lesssim 
\|\nabla a^{\ell,\varepsilon}\:b^{\ell,\varepsilon}\|_{\dot B^{s-1}_{2,1}}^{h,\varepsilon}+
\|a^{\ell,\varepsilon} \nabla b^{\ell,\varepsilon}\|_{\dot B^{s-1}_{2,1}}^{h,\varepsilon}.
\end{align*}
By symmetry, we only deal with $\nabla a^{\ell,\varepsilon} \: b^{\ell,\varepsilon}.$ We have
\begin{align*}
  \|\nabla a^{\ell,\varepsilon}\:b^{\ell,\varepsilon}\|_{\dot B^{s-1}_{2,1}}^{h,\varepsilon}\leq \|T'_{\nabla a^{\ell,\varepsilon}}\:b^{\ell,\varepsilon}\|_{\dot B^{s-1}_{2,1}}^{h,\varepsilon}+\|T_{\dot{B}^{\ell,\varepsilon}}\nabla a^{\ell,\varepsilon}\|_{\dot B^{s-1}_{2,1}}^{h,\varepsilon}
\end{align*}
Then, standard paraproduct estimates give
\begin{align*}
    \|T'_{\nabla a^{\ell,\varepsilon}}\:b^{\ell,\varepsilon}\|_{\dot B^{s-1}_{2,1}}^{h,\varepsilon} &\lesssim \|\nabla a^{\ell,\varepsilon}\|_{\dot B^{\frac{d}{p^*}-1}_{p^*,1}}\|b^{\ell,\varepsilon}\|_{\dot B^{s+\frac dp -\frac d2}_{p,1}}
    \\&\lesssim \| a\|_{\dot B^{\frac{d}{p}}_{p,1}}^{\ell,\varepsilon}\|b\|_{\dot B^{s+\frac dp -\frac d2}_{p,1}}^{\ell,\varepsilon}.
\end{align*}
Concerning $T_{\dot{B}^{\ell,\varepsilon}}\nabla a^{\ell,\varepsilon}$, we have
\begin{align*}    \|T_{\dot{B}^{\ell,\varepsilon}}\nabla a^{\ell,\varepsilon}\|_{\dot B^{s-1}_{2,1}}^{h,\varepsilon} &\lesssim \sum_{j\geq J_\varepsilon, |j-j'|\leq 1}2^{j(s-1)} \|S_{j'-1}b^{\ell,\varepsilon} \ddj \dot{\Delta}_{j'} \nabla a^{\ell,\varepsilon}\|_{L^2}
    \\&+ \sum_{j\geq J_\varepsilon, |j-j'|\leq 4}2^{j(s-1)} \|[\ddj,S_{j'-1}b]\dot{\Delta}_{j'}\nabla a^{\ell,\varepsilon}\|_{L^2}
    \\&= \cR_1 +\cR_2.
\end{align*}
For $\cR_1$, we have
\begin{align*}
    \cR_1 \leq \|b^{\ell,\varepsilon}\|_{L^\infty}\|\nabla a\|^{h,\varepsilon}_{\dot{B}^{s-1}_{2,1}}.
\end{align*}
For $\cR_2$, using commutator estimates yields
\begin{align} \nonumber
     \sum_{j\geq J_\varepsilon, |j-j'|\leq 4}2^{j(s-1)} \|&[\ddj,S_{j'-1}b]\dot{\Delta}_{j'}\nabla a^{\ell,\varepsilon}\|_{L^2}\lesssim 
     \sum_{j'\geq J_\varepsilon, |j-j'|\leq 4}2^{j(s-1)} \|\nabla b^{\ell,\varepsilon}\|_{L^\infty}\|\dot{\Delta}_{j'}a^{\ell,\varepsilon}\|_{L^2}
     \\&+  \sum_{|j-j'|\leq 4, J_\varepsilon> j' \geq J_\varepsilon-4}2^{j(\frac{d}{2}-\frac{d}{p}-1)}2^{j'(s+\frac{d}{p}-\frac{d}{2})}2^{(j-j')(s+\frac{d}{p}-\frac{d}{2})}\|\nabla b^{\ell,\varepsilon}\|_{L^{p^*}}\|\dot{\Delta}_{j'}a^{\ell,\varepsilon}\|_{L^p} \nonumber
     \\&\lesssim \|\nabla b^{\ell,\varepsilon}\|_{L^\infty}\|a\|^{h,\varepsilon}_{\dot{B}^{s-1}_{2,1}}+2^{(\frac{d}{p^*}-1)J_\varepsilon}\|\nabla b^{\ell,\varepsilon}\|_{L^{p^*}}\|a\|^{\ell,\varepsilon}_{\dot{B}^{s+\frac{d}{p}-\frac{d}{2}}_{p,1}}, \nonumber
\end{align}
where we used that $d/p^*=d/2-d/p$ and $d/2-d/p-1\leq 0$.
Then, employing $\dot B^{\frac dp-\frac d{p*}}_{p,1}\hookrightarrow L^{p^*}$, as $p\leq p^*$,
we have
\begin{align} \nonumber
2^{(\frac{d}{p^*}-1)J_\varepsilon}\|\nabla b^{\ell,\varepsilon}\|_{L^{p^*}} &\lesssim    2^{(\frac{d}{p^*}-1)J_\varepsilon}\|b\|^{\ell,\varepsilon}_{\dot{B}^{\frac dp-\frac d{p*}+1}_{p,1}}
\\&\nonumber \lesssim    2^{(\frac{d}{p^*}-1)J_\varepsilon} 2^{(1-\frac{d}{p^*})J_\varepsilon}\|b\|^{\ell,\varepsilon}_{\dot{B}^{\frac dp}_{p,1}}
\\&\lesssim \|b\|^{\ell,\varepsilon}_{\dot{B}^{\frac dp}_{p,1}}. \nonumber
\end{align}
Gathering these estimates, we obtain
\begin{align*}
  \|\nabla a^{\ell,\varepsilon}\:b^{\ell,\varepsilon}\|_{\dot B^{s-1}_{2,1}}^{h,\varepsilon}\lesssim
\| a\|_{\dot B^{\frac{d}{p}}_{p,1}}^{\ell,\varepsilon}\|b\|_{\dot B^{s+\frac dp -\frac d2}_{p,1}}^{\ell,\varepsilon}+\| b^{\ell,\varepsilon}\|_{\dot{B}^{\frac{d}{p}}_{p,1}}\|a\|^{h,\varepsilon}_{\dot{B}^{s}_{2,1}}+\|b^{\ell,\varepsilon}\|_{\dot{B}^{\frac{d}{p}}_{p,1}}\|a\|^{\ell,\varepsilon}_{\dot{B}^{s+\frac{d}{p}-\frac{d}{2}}_{p,1}}.
\end{align*}
Symmetrically, for $a^{\ell,\varepsilon} \nabla b^{\ell,\varepsilon}$, we obtain 
\begin{align*}
  \|\nabla b^{\ell,\varepsilon}\:a^{\ell,\varepsilon}\|_{\dot B^{s-1}_{2,1}}^{h,\varepsilon}\lesssim
\| b\|_{\dot B^{\frac{d}{p}}_{p,1}}^{\ell,\varepsilon}\|a\|_{\dot B^{s+\frac dp -\frac d2}_{p,1}}^{\ell,\varepsilon}+\| a^{\ell,\varepsilon}\|_{\dot{B}^{\frac{d}{p}}_{p,1}}\|b\|^{h,\varepsilon}_{\dot{B}^{s}_{2,1}}+\|a^{\ell,\varepsilon}\|_{\dot{B}^{\frac{d}{p}}_{p,1}}\|b\|^{\ell,\varepsilon}_{\dot{B}^{s+\frac{d}{p}-\frac{d}{2}}_{p,1}}.
\end{align*}
Gathering the above estimates yields \eqref{eq:prod4}.
To prove \eqref{eq:prod42} and \eqref{eq:prod422}, we modify \eqref{eq:prodx12} and \eqref{eq:prodx13} by
\begin{align}
\label{eq:prodx14}
&\|T'_{b}a^{h,\varepsilon}\|_{\dot B^{s}_{2,1}}\lesssim\|b\|_{\dot{B}^{-1}_{\infty,\infty}}\|a\|^{h,\varepsilon}_{\dot B^{s+1}_{2,1}},
\\&\|T_{a}b^{h,\varepsilon}\|_{\dot B^{s}_{2,1}}\lesssim\|a\|_{\dot{B}^{-1}_{\infty,\infty}}\|b\|^{h,\varepsilon}_{\dot B^{s+1}_{2,1}}. \label{eq:prodx15}
\end{align}
Then, using $B^{\frac{d}{p}-1}_{p,1}\hookrightarrow B^{-1}_{\infty,\infty}$ concludes the proof of \eqref{eq:prod42}. To prove \eqref{eq:prod433}, we decompose $ab$ as
\begin{align}
ab=T_{a}b^{m,\varepsilon}+T'_{\dot{B}^{m,\varepsilon}}a+ab^{\ell}+ab^{h,\varepsilon}.\end{align}
Using standard product law yields
\begin{align}\label{R-E620}
\|ab^{h,\varepsilon}+ab^{\ell}\|^{h}_{\dot{B}^{\frac{d}{2}}_{2,1}}\lesssim \|a\|_{\dot{B}^{\frac{d}{p}}_{p,1}}\|(b^{h,\varepsilon},b^{\ell})\|_{ \dot{B}^{\frac{d}{2}-1}_{2,1}}.
\end{align}
Employing \eqref{R-E955b}, we have
\begin{align}\label{R-E6345}
\|T'_{\dot{B}^{m,\varepsilon}}a\|^{h}_{\dot{B}^{\frac{d}{2}}_{2,1}}\lesssim \|b^{m,\varepsilon}\|_{\dot{B}^{\frac{d}{p}}_{p,1}}\|a\|_{ \dot{B}^{\frac{d}{p}}_{p,1}}.
\end{align}
For $T_{a}b^{m,\varepsilon}$, we  perform a more refined analysis. We have
\begin{align}\label{R-E64745}
    \ddj (T_{a}b^{m,\varepsilon})&=\ddj \left(\sum_{j'\in \Z} S_{j'-1} a\dot{\Delta}_{j'}b^{m,\varepsilon}\right)
    \\&= \sum_{j', |j-j'|\leq 1} S_{j'-1} a\dot{\Delta}_{j'}\ddj b^{m,\varepsilon}+ \sum_{j',|j-j'|\leq 4} [\ddj,S_{j'-1} a]\dot{\Delta}_{j'}b^{m,\varepsilon}. \nonumber
\end{align}
 When $j'\geq J_\varepsilon$, the first and second terms can be treated as \eqref{R-E620}.
When $j'\leq J_1$ and $J\geq J_1$, we have
\begin{align}\label{R-E65}
   2^{js} \|[\ddj,S_{j'-1} a]\dot{\Delta}_{j'}b^{m,\varepsilon} \|_{L^2}&\leq 2^{-J_\varepsilon}2^{j(s+1)}\|[\ddj,S_{j'-1} a]\dot{\Delta}_{j'}b^{m,\varepsilon} \|_{L^2}
   \\& \lesssim \varepsilon \|\nabla S_{j'-1} a\|_{L^{p^*}}2^{js}\|\dot{\Delta}_{j'}b^{m,\varepsilon} \|_{L^p}, \nonumber
\end{align}
where we have used a Young-like inequality as in \cite{XUest}. Therefore, for $j'\leq J_\varepsilon$ and $s=\frac d2$, we have
\begin{align}
    \sum_{j\geq J_\varepsilon} \sum_{j',|j-j'|\leq4}2^{j(\frac d2)} \|[\ddj,S_{j'-1} a]\dot{\Delta}_{j'}b^{m,\varepsilon} \|_{L^2}& \lesssim \varepsilon\|\nabla a\|^{\ell,\varepsilon}_{\dot{B}^{\frac{d}{p}-\frac{d}{p^*}}_{p,1}}\|b^{m,\varepsilon}\|_{\dot{B}^{\frac{d}{2}}_{p,1}}
    \\& \lesssim \varepsilon\varepsilon^{\frac{d}{p^*}-1}\|a\|^{\ell,\varepsilon}_{\dot{B}^{\frac{d}{p}}_{p,1}}\|b^{m,\varepsilon}\|_{\dot{B}^{\frac{d}{2}}_{p,1}} \nonumber
     \\& \lesssim \varepsilon^{\frac{d}{p^*}}\|a\|^{\ell,\varepsilon}_{\dot{B}^{\frac{d}{p}}_{p,1}}\varepsilon^{\frac{d}{p}-\frac{d}{2}}\|b\|^{m,\varepsilon}_{\dot{B}^{\frac{d}{p}}_{p,1}}, \nonumber
\end{align}
since $\frac{d}{p}+\frac{d}{p^*}=\frac{d}{2}.$ The proof of Proposition \ref{LPP} is concluded.
\end{proof}
We now state low-frequency product laws.
\begin{Prop}[Low-frequency product law]\label{prop:PL-low}
We have
\begin{align}
&\label{PL:lowhigh2}\|fg\|^\ell_{\dot{B}^{\frac{d}{2}-2}_{2,1}}\leq \|f\|_{\dot{B}^{\frac{d}{p}-1}_{p,1}}\|g\|_{\dot{B}^{\frac{d}{p}-1}_{p,1}},
    \\
\label{PL:lowhigh1}
    &\|fg\|^\ell_{\dot{B}^{\frac{d}{2}-2}_{2,1}}\leq \|f\|_{\dot{B}^{\frac{d}{p}-2}_{p,1}}\|g\|_{\dot{B}^{\frac{d}{p}}_{p,1}}+\|f\|^\ell_{\dot{B}^{\frac{d}{2}}_{2,1}}\|g\|_{\dot{B}^{\frac{d}{p}-2}_{p,1}}+\|f\|^h_{\dot{B}^{\frac{d}{p}-2}_{p,1}}\|g\|_{\dot{B}^{\frac{d}{p}-1}_{p,1}},
    \\\label{PL:lowhigh3} &\|fg\|^\ell_{\dot{B}^{\frac{d}{2}-2}_{2,1}}\leq \|f\|_{\dot{B}^{\frac{d}{p}-2}_{p,1}}\|g\|_{\dot{B}^{\frac{d}{p}}_{p,1}}+\|f\|^\ell_{\dot{B}^{\frac{d}{2}-1}_{2,1}}\|g\|_{\dot{B}^{\frac{d}{p}-1}_{p,1}}+\|f\|^h_{\dot{B}^{\frac{d}{p}-2}_{p,1}}\|g\|_{\dot{B}^{\frac{d}{p}-1}_{p,1}},
\\\label{PL:lowhigh4}
    &\|fg\|^\ell_{\dot{B}^{\frac{d}{2}-1}_{2,1}}\leq \|f\|_{\dot{B}^{\frac{d}{p}-1}_{p,1}}\|g\|_{\dot{B}^{\frac{d}{p}}_{p,1}}+\|f\|^\ell_{\dot{B}^{\frac{d}{2}-1}_{2,1}}\|g\|_{\dot{B}^{\frac{d}{p}}_{p,1}}+\|f\|^h_{\dot{B}^{\frac{d}{p}-1}_{p,1}}\|g\|_{\dot{B}^{\frac{d}{p}-1}_{p,1}}.
\end{align}
\end{Prop}
\begin{proof}
    Estimate \eqref{PL:lowhigh2} is classical. To prove estimate \eqref{PL:lowhigh1} and \eqref{PL:lowhigh3}, we use Bony's paraproduct decomposition to rewrite $fg$ as
\begin{equation}\label{R-E955f} fg=T_{f}v^\varepsilon+R(f,g)+T_{g}f^\ell+T_{g}f^{h}.\end{equation}
Using that $R$ and $T$ map $\dot{B}^{\frac{d}{p}-2}_{p,1}\times \dot{B}^{\frac{d}{p}}_{p,1}$ to $\dot{B}^{\frac{d}{2}-2}_{2,1}$, if $p<d$ and $d\geq3$, we have
 $$
\|T_{f}g\|^{\ell}_{\dot{B}^{\frac d2-2}_{2,1}} \lesssim \|f\|_{\dot{B}^{\frac dp-2}_{p,1}}\|g\|_{\dot{B}^{\frac dp}_{p,1}} \andf
\|R(f,g)\|^{\ell}_{\dot{B}^{\frac d2-2}_{2,1}} \lesssim  \|f\|_{\dot{B}^{\frac dp-1}_{p,1}}\|g\|_{\dot{B}^{\frac dp-1}_{p,1}}.
$$
Similarly, and with the fact that $T$ maps $\dot{B}^{\frac{d}{p}-1}_{p,1}\times \dot{B}^{\frac{d}{p}-1}_{p,1}$ to $\dot{B}^{\frac{d}{2}-2}_{2,1}$, we have
\begin{align}
\|T_{g}f^\ell\|^{\ell}_{\dot{B}^{\frac d2-2}_{2,1}} \lesssim \|g\|_{\dot{B}^{\frac dp-2}_{p,1}}\|f\|^\ell_{\dot{B}^{\frac dp}_{p,1}} 
\andf \|T_{g}f^\ell\|^{\ell}_{\dot{B}^{\frac d2-2}_{2,1}} \lesssim \|g\|_{\dot{B}^{\frac dp-1}_{p,1}}\|f\|^\ell_{\dot{B}^{\frac dp-1}_{p,1}}.
\end{align}
In order to handle the term with $T_{g}f^h$, we observe that owing to the spectral cut-off, there exists a universal integer $N_0$ such that
$$\Big(T_{g}f^{h}\Big)^{\ell}=\dot{S}_{J_{0}+1}\Big(\sum_{|j-J_0|\leq N_{0}}\dot{S}_{j-1}g\dot{\Delta}_{j}\nabla q^{h}\Big).$$
Hence $\|T_{g}f^{h}\|^{\ell}_{\dot{B}^{\frac d2-2}_{2,1}}\approx 2^{J_0(\frac d2-2)}\sum_{|j-J_0|\leq N_{0}}\|\dot{S}_{j-1}g\dot{\Delta}_{j}f^{h}\|_{L^2}$. 
Then, if $2\leq p \leq \min (d,d^*)$, for $|j-J_0|\leq N_{0}$, we have
\begin{equation*}
\begin{aligned}
2^{J_0(\frac d2-2)}\|\dot{S}_{j-1}g\dot{\Delta}_{j}f\|_{L^2}&\lesssim \|\dot{S}_{j-1}g\|_{L^d}\Big(2^{j(\frac{d}{d^*}-2)}\|\dot{\Delta}_{j}f^{h}\|_{L^{d^{*}}}\Big)\\
&\lesssim \|g\|_{\dot{B}^{0}_{d,1}}\|f^{h}\|_{\dot{B}^{\frac {d}{d*}-2}_{d*,\infty}}
\lesssim \|g\|_{\dot{B}^{\frac dp-1}_{p,1}}\|f^{h}\|_{\dot{B}^{\frac dp-2}_{p,1}},
\end{aligned}
\end{equation*}
where we have used the embeddings $\dot{B}^{\frac dp-1}_{p,1}\hookrightarrow \dot{B}^{0}_{d,1}\hookrightarrow L^{d}$ and $\dot{B}^{\frac dp-2}_{p,1}\hookrightarrow\dot{B}^{\frac dp-2}_{p,\infty}\hookrightarrow\dot{B}^{\frac {d}{d*}-2}_{d*,\infty}$.  If $d\leq p\leq 4$, then it holds that
\begin{equation}\nonumber
\begin{aligned}
2^{J_0(\frac d2-2)}\|\dot{S}_{j-1}g\dot{\Delta}_{j}f^{h}\|_{L^2}&\lesssim\Big(2^{j\frac d4}\|\dot{S}_{j-1}g\|_{L^4}\Big)\Big(2^{j(\frac d4-2)}\|\dot{\Delta}_{j}f^{h}\|_{L^4}\Big)\\ &\lesssim 2^{J_0}\Big(2^{j(\frac dp-1)}\|\dot{S}_{j-1}g\|_{L^p}\Big)\Big(2^{j(\frac dp-2)}\|f^{h}\|_{L^p}\Big) \lesssim\|g\|_{\dot{B}^{\frac dp-1}_{p,1}}\|f^{h}\|_{\dot{B}^{\frac dp-2}_{p,1}}.
\end{aligned}
\end{equation}
Hence, we deduce that
\begin{equation}\label{R-E4}
\begin{aligned}
&\|T_{g}f^{h}\|^{\ell}_{\dot{B}^{\frac d2-2}_{2,1}}\lesssim \|f^{h}\|_{\dot{B}^{\frac dp-2}_{p,1}}\|g\|_{\dot{B}^{\frac dp-1}_{p,1}},
\end{aligned}
\end{equation}
    which concludes the proof of \eqref{PL:lowhigh1} and \eqref{PL:lowhigh3}. Proving \eqref{PL:lowhigh4} follows the same lines replacing $d/p-2$ by $d/p-1$.
\end{proof}

Then, we state a hybrid composition estimate from  \cite{XuZhang2023}.
\begin{Prop}[\cite{XuZhang2023}]
    Let $f(u)$ be a smooth function satisfying $f(0)=0$. If $u^{h,\varepsilon}\in B^s_{2,1}$ and $u^{\ell,\varepsilon}\in B^{\frac{d}{p}}_{p,1}$ for $s>1$ and $s\geq \frac d2$,
     $2\leq p \leq 4$ if $d=1$ and $2\leq p \leq \min\{4,\frac{2d}{d-2}\}$ if $d\geq2$, then we have $f(u)^{h,\varepsilon}\in B^{s}_{2,1}$ and there is a positive constant $C$ independent of $\varepsilon$ such that
    \begin{align}\label{prop:comphf}     \|f(u)\|_{\dot{B}^{s}_{2,1}}^{h,\varepsilon}\leq C\left(1+\|u\|^{\ell,\varepsilon}_{\dot{B}^{\frac{d}{p}}_{p,1}}+2^{\left(\frac{d}{2}-s\right)J_\varepsilon}\|u\|^{h,\varepsilon}_{\dot{B}^{\frac{d}{2}}_{2,1}}\right)\left(\|u\|^{\ell,\varepsilon}_{\dot{B}^{s+\frac{d}{p}-\frac{d}{2}}_{p,1}}+\|u\|^{h,\varepsilon}_{\dot{B}^{s}_{2,1}}\right).
    \end{align}
\end{Prop}

Finally, we state a composition estimate that was proved in \cite{CBD3}.
\begin{Lemme}[\cite{CBD3}]
\label{CP}Let $p\in[2,4]$ and $s>0$. Define $p^*\triangleq 2p/(p-2).$  For $j\in\Z,$ 
denote $$\cI_1\triangleq \int_{\R^d}\ddj(w \cdot \nabla z)z_j.$$
There exists a constant $C$ depending only on $s,$ $p,$ $d,$ such that
$$\displaylines{
\left(2^{js}\cI_1\right)\leq C c_j\|z_j\|_{L^2}\Bigl(\norme{\nabla w}_{ B^{\frac dp}_{p,1}}\norme{z}^{h,\varepsilon}_{{B}^{s}_{2,1}}
+ \norme{z}^{\ell,\varepsilon}_{ B^{\frac{d}{p}}_{p,1}}\norme{w}^{\ell,\varepsilon}_{\dot{B}^{s}_{p,1}}+\norme{z}_{\dot{B}^{\frac dp}_{p,1}}\norme{w}^h_{ B^{s}_{2,1}}
+ \norme{z}^{\ell,\varepsilon}_{\dot{B}^{\frac dp}_{p,1}}\norme{w}^{\ell,\varepsilon}_{ B^{\frac{d}{p}}_{p,1}}\Bigr),}$$
where $(c_j)_{j\geq J_\varepsilon}$ is a sequence such that $\sum_{j\geq J_\varepsilon}c_j=1$ and we recall that
$$\norme{f}^{\ell,\varepsilon}_{\dot{ B}^{s}_{p,1}}=\norme{f}^{\ell}_{\dot{ B}^s_{p,1}}+\norme{f}^{m,\varepsilon}_{\dot{ B}^s_{p,1}}.$$
\end{Lemme}


 %

\vspace{5mm}
\bibliographystyle{abbrv}

\bibliography{Ref.bib}

\vfill \footnotesize

(T. Crin-Barat) \textsc{Chair for Dynamics, Control, Machine Learning and Numerics, Alexander Von Humboldt- Professorship, Department of
Mathematics, Friedrich-Alexander-Universität Erlangen-Nürnberg, 91058 Erlangen, Germany}

E-mail address: {\tt timothee.crin-barat@fau.de}

\medbreak

(S. Kawashima) \textsc{Faculty of Science and Engineering, Waseda University, Tokyo 169-8555, Japan}

E-mail address: {\tt kawashima.shuichi.541@m.kyushu-u.ac.jp}
\medbreak
\par\nopagebreak
\noindent

(J. Xu) \textsc{Department of Mathematics, Nanjing University of Aeronautics and
Astronautics, Nanjing, 211106, P. R. China}

E-mail address: {\tt jiangxu\underline{~}79math@yahoo.com, jiangxu\underline{~}79@nuaa.edu.cn}



\end{document}